\documentclass[11pt]{article}
\usepackage[utf8]{inputenc}
\usepackage{amsfonts, amsmath}
\usepackage{amssymb}
\usepackage{amsthm,amsmath,amscd}
\usepackage{enumerate}
\usepackage{url}
\usepackage{amsbsy}
\usepackage[pdftex]{graphicx}
\usepackage[margin=25pt,font=small,labelfont=bf]{caption}
\usepackage[breaklinks,hyperfootnotes=false,bookmarks=false,colorlinks=true]{hyperref}
\usepackage{subfig}
\usepackage{mdframed}
\usepackage{color}
\usepackage[square,numbers,sort&compress]{natbib}
\newtheorem{innercustomthm}{Theorem}
\newenvironment{customthm}[1]
  {\renewcommand\theinnercustomthm{#1}\innercustomthm}
  {\endinnercustomthm}

\newcommand{\defn}[1]{\textcolor{blue}{\emph{#1}}}

\graphicspath{{figs/}}

\def\qed{\hfill {\hfill $\Box$} \medskip}
\DeclareMathOperator{\good}{Good}
\DeclareMathOperator{\bad}{Bad}
\DeclareMathOperator{\sing}{Sing}
\DeclareMathOperator{\Fano}{Fano}

\newcommand{\RR}{\mathbb R}

\newcommand{\EE}{\mathbb E}
\newcommand{\R}{\mathbb R}

\newcommand{\bna}{\begin{eqnarray}}
\newcommand{\ena}{\end{eqnarray}}
\newcommand{\ba}{\begin{eqnarray*}}
\newcommand{\ea}{\end{eqnarray*}}
\newcommand{\bs}[1]{}
\newcommand{\ud}{\,\mathrm{d}}
\newcommand{\edgecard}{N}
\newcommand{\incidencegraph}{\Delta}
\newcommand{\graphautomorphism}{\rho}
\newcommand{\reflectiongroup}{W}

\DeclareMathOperator{\Aut}{Aut}

\newtheorem{theorem}{Theorem}[section]

\newtheorem{lemma}[theorem]{Lemma}
\newtheorem{proposition}[theorem]{Proposition}
\newtheorem{remark}[theorem]{Remark}
\newtheorem{definition}[theorem]{Definition}

\newtheorem{question}[theorem]{Question}

\DeclareMathOperator{\real}{Real}

\DeclareMathOperator{\Dim}{Dim}
\DeclareMathOperator{\Gr}{Gr}
\DeclareMathOperator{\rank}{rank}

\newcommand{\ra}{\rangle}
\newcommand{\la}{\langle}
\newcommand{\CC}{{\mathbb C}}
\newcommand{\CS}{{\mathcal S}}
\newcommand{\QQ}{{\mathbb Q}}
\newcommand{\ZZ}{{\mathbb Z}}

\newcommand{\bl}{{\bf l}}

\newcommand{\LL}{{L_{2,4}}}
\newcommand{\Ln}{{L_{2,n}}}
\newcommand{\LN}{{L_{d,n}}}
\def\p{{\bf p}}
\def\b{{\bf b}}
\def\pn{\p =(\p_1, \dots, \p_{n}) }
\def\q{{\bf q}}
\def\f{{\bf f}}
\def\g{{\bf g}}
\def\h{{\bf h}}
\def\G{{\bf G}}
\def\D{{\bf D}}
\def\bm{{\bf m}}
\def\v{{\bf v}}
\def\w{{\bf w}}
\def\e{{\bf e}}
\def\r{{\bf r}}
\def\E{{\bf E}}

\def\Q{{\bf Q}}

\def\y{{\bf y}}
\def\z{{\bf z}}

\def\M{{\bf M}}
\def\I{{\bf I}}
\def\PP{{\mathbb{P} }}
\def\N{{\bf N}}
\def\A{{\bf A}}
\def\B{{\bf B}}

\def\S{{\bf S}}
\def\T{{\bf T}}
\def\J{{\bf J}}
\def\X{{\bf X}}
\def\Y{{\bf Y}}
\def\Z{{\bf Z}}

\def\P{{\bf P}}

\newcommand{\genericpoint}{\bl}

\newcommand{\trans}[1]{{#1}^\top}

\DeclareMathOperator{\lin}{lin}

\title{Determining Generic Point Configurations\\ From Unlabeled 
Path or Loop Lengths}

\author{
Ioannis Gkioulekas, 
Steven J. Gortler,
Louis Theran,
and 
Todd Zickler}
\date{}

\usepackage[margin=1in]{geometry}
\begin{document}
\maketitle 

\begin{abstract}
Let $\p$ be a configuration of $n$ points in $\RR^d$ for some 
$n$ and some
$d \ge 2$. Each pair of points defines an edge, which has a Euclidean
length in the configuration. A path is an ordered sequence of the
points, and a loop is a path that has the same endpoints. A path or loop, as a sequence of edges, also has a Euclidean length.

In this paper, 
we study the question of when $\p$ will be uniquely determined (up to an
unknowable Euclidean transform) from a given set of path or loop lengths.
In particular, we consider the setting where the lengths are
given simply as a set of real numbers, and are 
not labeled with the combinatorial data describing the paths 
or loops that gave rise to the lengths.

Our main result is a condition on the set of paths or loops
that is sufficient to guarantee such a unique determination. We also provide 
an algorithm, under a real computational model, for performing 
a reconstruction of $\p$ from such unlabeled lengths.

To obtain our results, we introduce a new family of algebraic 
varieties which we call the
unsquared measurement varieties. The family is parameterized
by the number of points $n$ and the dimension $d$, and our results follow from
a complete characterization of the linear
automorphisms of these varieties 
for all $n$ and $d$.  The linear automorphisms for the special case of $n=4$ and
$d=2$ correspond to the so-called Regge symmetries of the
tetrahedron.
\end{abstract}

\section{Introduction} 
\label{sec:intro}

We are motivated by the following signal processing
scenario.
Suppose there is a ``configuration'' $\pn$ of $n$ points 
in, say, $\RR^2$ or $\RR^3$. Let a ``path'' be a finite sequence of these points, 
and a ``loop'' be a path that begins and ends at the same point.
Each such path or loop in $\p$ has a Euclidean length. 

\begin{figure}[b]
	\centering
	\def\svgwidth{4in}
	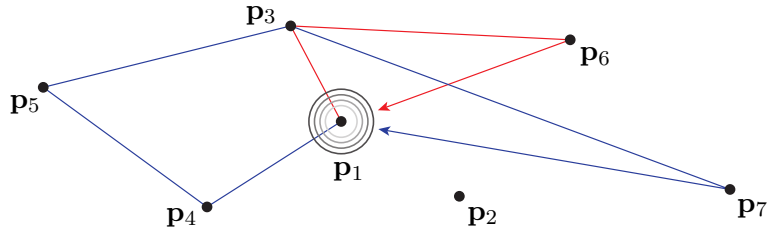
	\caption{An emitter-receiver at point $\p_1$ emits an omnidirectional pulse that bounces among points $\p_i$. The same emitter-receiver records the arrival times of pulse fronts that eventually return. These arrival times measure the lengths of loops that begin and end at $\p_1$. 
}\label{fig:intro}
\end{figure}

Let $\p_1$ be a distinguished point. In our scenario, 
it may represent the location of an omni-directional emitter and receiver of sound or radiation. Let the other points in $\p$ represent the positions of small objects that behave as omnidirectional scatterers.

An omnidirectional pulse is emitted from $\p_1$ and travels outward at, say, unit speed. Whenever the pulse front encounters an object $\p_i$, an additional omnidirectional pulse is created there through scattering. Pulses continue to bounce around in this manner, and the receiver at $\p_1$ records the arrival times of the pulse fronts that return. We allow for the possibility that some pulse fronts might vanish or not be measurable back at $\p_1$.

By recording the times of flight between emission and reception, we effectively measure the lengths of loops traveled. In the case of light, these are travel times of photons that leave $\p_1$ and return after one or more bounces. In the case of sound, these are delays of direct or indirect echoes.

Importantly, each recorded length measurement is a single real number $v$. We do not obtain any \emph{labeling} information about which points were visited or how many bounces occurred during the loop. Nor do we obtain any information about the direction from which energy arrives. 

We wish to understand when we can recover the point configuration (up to Euclidean congruence) from a sufficiently rich sequence of \emph{unlabeled} loop measurements. Once
the measurements are labeled, this becomes a well-studied problem,
which can be reasonably solved when there is a sufficiently
rich set of paths. The
difficulty here arises from the lack of labeling.

In this paper we will prove that if $\p$ is a ``generic'' 
point configuration in $\R^d$, for $d\ge 2$,
and we measure a
sufficiently rich set of loops, namely
one that ``allows for trilateration'' (formally defined later),
then the configuration is uniquely
determined from these measurements up to congruence.  Moreover this
leads to an algorithm, under a real computation model~\cite{bss},
to calculate $\p$ from such data.  
The assumption of genericity (defined later) 
roughly means that there are some special $\p$ where
these conclusions do not hold, but these special cases
are very rare. In deriving our results, we will not concern ourselves with 
noise or numerical issues. We
plan to address some of these issues in future work.

To put this work in the context of previous mathematical results,
Boutin and Kemper~\cite{BK1} have shown that if $\p$ is a generic
point configuration in $\R^d$ with $n \geq d+2$, and we are given the complete
set of all $\binom{n}{2}$
edge lengths as an unlabeled sequence, then
$\p$ is uniquely determined up to Euclidean congruence and point
relabeling. This result can also be generalized to the case
where not all of the  edge lengths have been measured,
but only a subset that is rich enough to allow for trilateration~\cite{dux1}.

In this setting, one can ask: Suppose we have an unlabeled measurement
set, that includes a collection of edge \emph{and} path or loop lengths. Then, is $\p$ still uniquely determined, and 
can it be reconstructed? In this paper, we will answer this in the affirmative, 
under the condition that the measurement ensemble allows for trilateration. We will 
additionally show that the same holds for an unlabeled measurement ensemble that 
includes only loop measurements.

\subsection{Unlabeled Trilateration}
\label{sec:introa}

\begin{figure}[ht]
	\centering
	\def\svgwidth{.8\textwidth}
	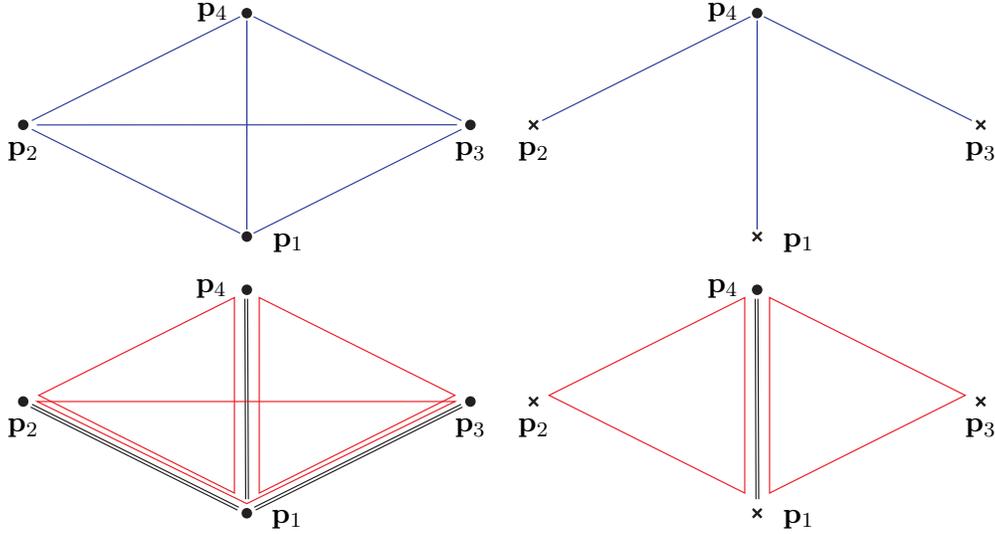
	\caption{Top row: A $K_{4}$ contained within a path measurement ensemble 
consists of six edges (blue lines) (left). During trilateration using path measurement
 data, three points $\p_1$, $\p_2$, $\p_3$ are known, and a fourth point $\p_4$ is
reconstructed from three edge length measurements (right).
Bottom row: A $K_{4}$ contained within a loop measurement ensemble consists of three pings (double black lines) and three triangles (red lines) (left). During trilateration using loop measurement data, three points $\p_1$, $\p_2$, $\p_3$ are known, and a fourth point $\p_4$ is
reconstructed from one ping and two triangle length measurements (right).}\label{fig:pt}
\end{figure}

Our sufficient condition for uniqueness, and a core part of our reconstruction
algorithm, is based on the notion of ``trilateration''. We outline this notion here, using two dimensions for simplicity. Trilateration has a base step and an inductive step. First we describe these steps in the labeled setting, where each length measurement is identified with the sequence of points that created it. Then we move to the unlabeled setting.

{\bf Base step:}
In the labeled path setting, suppose we are given the $6$ labeled edge lengths (edges are very simple paths)
of a tetrahedron in $\RR^2$, as in~\ref{fig:pt} (upper left).
Then we can easily reconstruct the configuration of its four points
(up to an unknowable Euclidean congruence)~\cite{YH38}.
Likewise in the loop setting, 
suppose we are given the $6$ labeled loop lengths comprising
the three ``pings'' (loops that traverse only one edge twice) and three ``triangles'' (loops that traverse only three distinct edges once) as shown in Figure~\ref{fig:pt} 
(lower left). Again, from this labeled data it is straightforward
to reconstruct the configuration of its four points (up to congruence).

{\bf Inductive step:}
Continuing on, suppose we already know the positions of some $3$ points, and we are given
either the labeled edge lengths to some fourth point as in Figure~\ref{fig:pt} (upper right), or the labeled loop lengths of the one ping and two triangles shown in Figure~\ref{fig:pt} (lower right). In either case, we can easily reconstruct the position of the fourth point.


We say that a labeled data set (in either the path or loop setting)
``allows for trilateration'' if it has 
enough labeled measurements so that we can apply a base step
and then iteratively apply the inductive step until we reconstruct all $n$ of the points.

{\bf Unlabeled setting:}
Now suppose $\p$ consists of $n$ points, and we are given a large collection of path or loop lengths that are unlabeled. We can take any ordered subset of $6$ lengths from this data set and \emph{hypothesize} that they arise, in the path or loop setting respectively, from the six ``tetrahedral'' edges of Figure~\ref{fig:pt} (upper left), corresponding to the edges of a ``flat tetrahedron'', or the three pings and three triangles of Figure~\ref{fig:pt} (lower left). In either case, we can attempt to test the hypothesis by checking whether the $6$ path or loop lengths can ``fit together'' to form a tetrahedron in $\RR^2$ (this is essentially a determinant calculation). If they do not fit together appropriately, then we can be sure that our hypothesis was wrong.



Suppose, though, that when testing such a hypothesis,
we find that the $6$ lengths in fact \emph{are} consistent with a tetrahedron. Can we conclude that our hypothesis is correct? Or could the lengths arise from some other set of paths or lengths
among the $n$ points, thereby giving a ``false positive''? In this paper, our main mathematical task will be to show, with some qualifications, that there are no false positives. 

Once this mathematical principle is settled, it suggests the following procedure: Given 
a large collection of unlabeled path or loop lengths, we can search over all ordered six-tuples in the data set until we find a six-tuple that is consistent with the measurements from one tetrahedron. If such a tetrahedral six-tuple exists in the data set, we are guaranteed to find it, and we can reconstruct the position of $4$ points of $\p$ (up to congruence), successfully completing a trilateration base step, without labels.

Continuing on, the same type of hypothesize-and-test approach can be used for the inductive step of trilateration. We take a triplet of previously localized points, along with an ordered triplet of lengths, and we hypothesize that they jointly arise from the geometry of Figure~\ref{fig:pt} (upper right) for paths, or (lower right) for loops. When testing the hypothesis, from the same principle as above, we can argue that there will be no false positives. Thus, positive test results provide the reconstruction of additional points, and if the underlying data set allows for trilateration, we can reconstruct all of $\p$, even without labels.

As we will see, the unlabeled trilateration approach also applies in higher dimensions, $d>2$. Tests on six-tuples of measurements are simply replaced by tests on $\binom{d+2}{2}$-tuples of measurements; and the test criteria based on tetrahedral configurations are replaced by analogous criteria (still essentially determinant calculations) derived from configurations of $(d+1)$ points.

\subsection{General Approach}
\label{sec:introb}

A path or loop length arises from some linear functional over the pairwise distances between points of $\p$. In order to prove the absence of false positives in the hypothesize-and-test approach described in the previous section, we will need to argue the following: Given a generic $\p$ and a set of linear functionals that correspond to the set of paths or loops that we hypothesize (the four sets shown in Figure~\ref{fig:pt} for $d=2$, or their higher-dimensional counterparts), an adversary cannot construct a different configuration $\q$ (say, with $\q$ generic) and a different set of linear functionals that would yield the same measurements.

It is difficult to reason about any specific $\p$ and potential $\q$. Instead, we study the set of \emph{all possible} pairwise distance measurements as we vary the configuration of $n$ points in dimension $d$; this forms an algebraic variety we call the ``unsquared measurement variety'' $L_{d,n}$.  

From this perspective, we can interpret our adversary to be trying to come up with a linear map on $L_{d,n}$ that sends the measurements from $\p$ to those of $\q$.  The key principles we rely on are that, if $\p$ is generic, then much more is true (definitions and proofs about complex algebraic varieties and other algebraic geometry preliminaries are in Appendix~\ref{sec:geometry}):




\begin{theorem}
\label{thm:prin1}
Let $V$ be an irreducible algebraic variety
and
$\genericpoint$ a generic point in the variety.
Let $\E$ be a linear map that maps $\genericpoint$
to some variety $W$. Let all the above be defined over
$\QQ$. Then $\E(V) \subset W$.
\end{theorem}

\begin{theorem}
\label{thm:prin2}
Let $V \in \CC^N$ be an irreducible algebraic variety
and
$\genericpoint$ a generic point in the variety.
Let $\A$ be a bijective linear map on $\CC^N$
that maps $\genericpoint$
to  $V$. Let all the above be defined over
$\QQ$. Then $\A(V) = V$, that is, $\A$ is a linear automorphism of $V$.
\end{theorem}

Thus a bijective linear map on $L_{d,n}$ that maps 
a complete edge-set of 
measurements of a generic $\p$ to 
the measurements of a different (non-congruent)
$\q$  also \textit{sends all of $L_{d,n}$ to itself}.  This gives
the key reduction that unique reconstructability is implied by an absence of 
``adversarial'' linear automorphisms of $L_{d,n}$.

The heart of this paper, which occupies most of it, 
is a complete characterization of the linear 
automorphisms of all of the $L_{d,n}$ varieties. 
We find that the \emph{only} ``unexpected''
automorphisms of $L_{2,4}$ are ones that arise due to the so-called
Regge symmetries~\cite{regge} of the tetrahedron. 
For all other $d$ and $n$, we show that there are
no unexpected automorphisms.
When all is said and done, the paucity of these automorphisms
will imply that each measurement label is uniquely determined.

Returning to the context of Section~\ref{sec:introa}, we will apply Theorem~\ref{thm:prin1} with the matrix $\E$ representing an
adversary's choice of $6$ paths/loops.
Here, 
$V$ is set to $L_{2,n}$ and $W$ set to $L_{2,4}$.
Using this theorem, we will  
ultimately be able to conclude that 
if $6$ path or loop measurements 
among $n$ points ``look like'' they come from a tetrahedron, as in 
Figure~\ref{fig:pt}, then they must indeed 
come from paths or
loops that are  supported over some $4$-point subset of $\p$.
Next, we will use Theorem~\ref{thm:prin2} 
with $V$ set to $L_{2,4}$. This will 
allow us to conclude that an adversary's paths or loops must be \emph{exactly}
those of Figure~\ref{fig:pt}; any other
$6$ paths or loops among $4$ points would correspond to 
an automorphism of $L_{2,4}$, which we have ruled out. Together, these will allow us to rule out
false positive test results.


\subsection*{Acknowledgements}
We would like to thank Dylan Thurston
for numerous helpful
conversations and suggestions throughout this project. His input on Regge symmetries, and on the use of 
covering space maps was essential.
We also thank Brian Osserman for fielding numerous algebraic geometry
queries.

Steven Gortler was partially supported by NSF grant DMS-1564473. Ioannis Gkioulekas and Todd Zickler received support from the DARPA REVEAL program under contract no.~HR0011-16-C-0028.




\section{Definitions and Main Results}\label{sec:defs}

We start by establishing our basic terminology.
\begin{definition}\label{def:constants}
Fix positive integers  $d$  and $n$.
Throughout the paper, we will set $N:= \binom{n}{2}$, 
$C:=\binom{d+1}{2}$, and $D:= \binom{d+2}{2}$.

These constants appear often because they are, respectively, the 
number of pairwise distances between $n$ points, the dimension of the 
group of congruences in $\RR^d$, 
and the number of edges in a complete $K_{d+2}$ graph (the importance of such graphs will be explained later).
\end{definition}

\begin{definition}
\label{def:graph}
A \defn{configuration}, $\pn$ is a  sequence of $n$ points
in $\RR^d$. (If we want to talk about points in $\CC^d$, we will
explicitly call this a \defn{complex configuration}.) The affine span
of a configuration need not be all of $\RR^d$.

We think of the integers in $[1,\dots,n]$ as a \defn{vertex}
of an abstract complete graph $K_n$. An \defn{edge}, $\{i,j\}$, is an
unordered distinct pair of vertices. The complete edge set of $K_n$ has cardinality $N$.

A \defn{path} $\alpha:=[i_1, i_2,
  \dots, i_z]$ is a finite sequence of vertices, 
  with no vertex immediately repeated.  A 
\defn{loop} is a path where $i_1=i_z$. We think of a path or loop as
comprising a sequence of $z-1$ edges. 
(The simplest kind of path, $[i,j]$, is a single edge. The
simplest kind of loop $[i,j,i]$ is called a \defn{ping}. Another
important kind of loop $[i,j,k,i]$ is a \defn{triangle}.)
Because we will only be interested in the geometric lengths
of paths through a configuration, 
two paths are considered equal if they comprise the same
edges, with the same multiplicity, in any order.

A \defn{path measurement ensemble} 
$\pmb{\alpha}:=\{\alpha_1,\dots,\alpha_k\}$ is a
finite sequence of paths. Likewise for a 
\defn{loop measurement ensemble}.

We say that a path or loop $\alpha$ is 
\defn{$b$-bounded}, for some positive integer $b$, 
if no edge appears more than $b$ times in $\alpha$.
We say that a measurement ensemble $\pmb{\alpha}$ is $b$-bounded if it comprises only $b$-bounded
loops or paths.

Fixing a configuration $\p$ in $\R^d$, we define the 
\defn{length} of an edge
$\{i,j\}$ to be the Euclidean distance between the points
$\p_i$ and $\p_j$, a real number. 

We define the length $v$ of a path or loop $\alpha$ to be
the sum of the lengths of its comprising edges. We denote this
as $v = \la \alpha, \p \ra$. (The intuition behind this notation will become evident in Section~\ref{sec:varieties}.) 

A configuration $\p$ and a measurement
ensemble $\pmb{\alpha}$ give rise to a \defn{data set} $\v$ that is the 
finite sequence 
of real numbers made up of the lengths of its paths or loops. We denote this as 
$\v = \la \pmb{\alpha}, \p \ra$. We say that this data set \defn{arises} from this measurement
ensemble. Notably, a data set $\v$ 
itself does not include any labeling information about the measurement ensemble it arose from. 
\end{definition}

\begin{remark}
\label{rem:diam}
In a practical setting, we may not know the actual bound $b$ of a $b$-bounded ensemble,
but instead know that it must exist for other reasons. 
In particular, suppose we have some 
bound on the maximal distance between any pair of points
in $\p$. Then we can safely assume that any sufficiently huge
data value arises from a sufficiently complicated path or loop
and discard it.
Suppose then that we also have some bound on the minimal distance
between any pair of points in $\p$. Then we know that any non-discarded
value must arise from a $b$-bounded loop or path with some 
appropriate $b$.
\end{remark}

\begin{definition}
We use $\p_I$ to refer to a \defn{subconfiguration} of a configuration
$\p$ indexed by an index sequence $I$, that is a 
(possibly reordered)
subsequence of
$\{1,\dots,n\}$. In particular, we use $\p_T$ to refer to a
${d+2}$ point subconfiguration in $\p$, indexed by
a sequence $T=\{i_1,\dots,i_{d+2}\}$ of $\{1,\dots,n\}$. Similarly, we use
$\p_R$ to refer to a ${d+1}$ point subconfiguration  of  $\p$.

We use $\v_J$ to refer to a \defn{sub data set}, 
a (possibly reordered) 
subsequence of the data set $\v$ indexed by an
index sequence $J$, and similarly for a \defn{subensemble} $\pmb{\alpha}_J$.
\end{definition}

\begin{definition}
For $s$ a real number, the \defn{$s$-scaled} configuration $s\cdot\p$
is the configuration obtained by scaling each of the coordinates of
each point in $\p$ by $s$.  Likewise for an $s$-scaled
subconfiguration $s\cdot\p_I$. For $s$ a positive integer, the
\defn{$s$-scaled} path $s\cdot\alpha$ is a path that comprises the
same edges as $\alpha$, but with $s$ times their multiplicity in
each measurement of 
$\alpha$. For $s$ a positive integer, the \defn{$s$-scaled} path
measurement ensemble $s\cdot\pmb{\alpha}$ is a path measurement
ensemble that contains the same sequence of paths as $\pmb{\alpha}$,
except each scaled by $s$. Likewise for an $s$-scaled loop and
$s$-scaled loop measurement ensemble.
\end{definition}

We will be interested in measurement ensembles that are sufficient to 
uniquely determine the configuration in a greedy manner.

\begin{definition}
\label{def:contained}
In the path setting, we say that  
a $K_{d+2}$
 subgraph  of $K_n$ is
\defn{contained} within a path measurement ensemble $\pmb{\alpha}$ if
the ensemble includes a subensemble of size $D$ comprising the edges of this subgraph.
For the two-dimensional case, see Figure~\ref{fig:pt} (upper left).

In the loop setting, we say that  
a $K_{d+2}$
 subgraph  of $K_n$ 
with vertices $\{i_1,\dots, i_{d+2}\}$ 
is
\defn{contained} within a loop measurement ensemble $\pmb{\alpha}$ if
the ensemble includes a subensemble of size $D$ comprising
the $d+1$ pings,
$[i_1, j , i_1]$ for $j$ spanning  $[2,\dots,d+2]$; and also the
triangles, 
$[i_1, j_1, j_2, i_1]$
for $j_1 < j_2$ spanning  $[2,\dots,d+2]$.
That is, the ensemble includes
all pings and triangles in this $K_{d+2}$
 with endpoints at vertex
$i_1$.  For the two-dimensional case, see Figure~\ref{fig:pt} (bottom left).

\end{definition}

\begin{definition}\label{def:ensemble}
We say that a path measurement ensemble \defn{allows for trilateration} if, after reordering the vertices: 
i) it contains an initial \defn{base} $K_{d+2}$ 
over  $\{1,\dots,d+2\}$;
ii) for all subsequent $(d+2) <j \leq n$, it includes as a 
subsequence a \defn{trilateration sequence} comprising the edges
$[i_1,j],\dots,[i_{d+1},j]$
where all $i_k < j$. For the two-dimensional case, see Figure~\ref{fig:pt} (top right).

We say that a loop measurement ensemble \defn{allows for 
trilateration} if, after reordering the vertices:
i) it contains an initial \defn{base} $K_{d+2}$ over $\{1,\dots, d+2\}$;
ii) for all subsequent $(d+2) <j \leq n$, it includes as a subsequence a \defn{trilateration sequence} comprising the triangles  
$[i_1,i_2,j,i_1],\dots,[i_1,i_{d+1},j,i_1]$, and also the ping
$[i_1,j,i_1]$,
where all $i_k < j$.
That is, it includes one ping from $j$ back to one previous $i_1$,
and $d$ triangles back to the previous vertices and including $i_1$.
(See Figure~\ref{fig:pt} (bottom right) for the two-dimensional case.)
\end{definition}

Note that a path (resp. loop) measurement ensemble that allows for trilateration may
include any other additional paths (resp. loops) beyond those specified in Definition~\ref{def:ensemble}. 




Next, we define a strong notion of a generic configuration.

\begin{definition}
\label{def:genR}
We say that a real point in $\RR^{dn}$
is \defn{generic}
if its coordinates do not satisfy any non-trivial 
polynomial equation with coefficients in $\QQ$.
The set of generic real points have full measure and are 
(standard topology) dense in $\RR^{dn}$.

We say that a configuration $\p$ of $n$ points in $\RR^d$ is generic if
it is generic when thought of as a single point in $\R^{dn}$.
\end{definition}

Ultimately, we will be most interested in properties that hold 
not merely at all generic  configurations, but over an open and dense
subset of the configuration space.
 Such a property will
be what we ``generally'' observe when looking at configurations, and
will be stable under any perturbations. There can be exceptional configurations
but they are very confined and isolated.

When a property holds at all generic configurations
and the exceptions are due to only a \emph{finite} number of
algebraic conditions, then we will be able to conclude that the 
property actually holds over a Zariski open subset.
\begin{definition}
\label{def:zopenR}
A non-empty real subset $S$ of  $\RR^{dn}$ is \defn{Zariski open}
if it can be obtained from $\RR^{dn}$ by 
cutting out the set of points that simultaneously 
solve a finite number
of non-trivial polynomial equations. A non-empty real Zariski open subset is open and (standard topology) dense 
in $\RR^{dn}$, and has full measure.
\end{definition}

\subsection{Results}\label{sec:results}

The central conclusion of this paper will be the following
``global rigidity'' 
statement:
\begin{mdframed}
\begin{theorem}
\label{thm:punchline}
Let the dimension be $d\ge 2$. 
Let $\p$ be a generic configuration of $n\ge d+2$ points. Let
$\v= \la \pmb{\alpha}, \p \ra$ where 
$\pmb{\alpha}$
is a path (resp. loop) measurement ensemble
that allows for 
trilateration.

Suppose there is a 
configuration $\q$, 
also of $n$ points,
along with 
a
measurement ensemble $\pmb{\beta}$ 
such that 
$\v=\la \pmb{\beta},\q \ra$.

Then 
there is a vertex relabeling  of $\q$ such that,
up to congruence,
$\q=1/s\cdot\p$,
with $s$ a whole number $\ge 1$.
Moreover, under this vertex relabeling,
$\pmb{\beta} = s\cdot \pmb{\alpha}$.

If we also assume that 
$\pmb{\beta}$ allows for trilateration, then 
there is a vertex relabeling of $\q$ such that,
up to congruence,
$\q=\p$.
Moreover, under this vertex relabeling,
$\pmb{\beta} = \pmb{\alpha}$.

If the measurement ensembles $\pmb{\alpha}$
and $\pmb{\beta}$
are assumed to be   
$b$-bounded,
for any fixed $b$, 
then the above determination holds over some 
Zariski open subset of configurations $\p$. This subset depends only on $b$.
\end{theorem}
\end{mdframed}

When 
$\la \pmb{\alpha},\p \ra$ agrees with
$\la \pmb{\beta},\q \ra$ after some permutation, then
the theorem can be applied after appropriately permuting
$\pmb{\beta}$.

Note that if one lets $\q$ be non-generic \emph{and} puts no restrictions on 
the number of points, 
then one can obtain any target $\v$ by 
letting $\pmb{\beta}$ be a tree of edges and then placing
$\q$ appropriately. 

\begin{figure}[ht]
	\centering
	\def\svgwidth{.9\textwidth}
	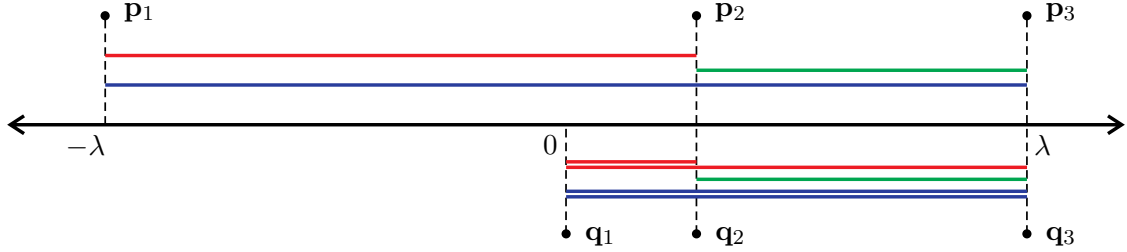
	\caption{Counterexample in one dimension.
  The configuration $\p$ with the shown (upper) three edge measurements
  gives rise to the same length values as the configuration
  $\q$ with the shown (lower) three path measurements. 
  This behavior is stable; as $\p$ is perturbed, $\q$ can
  be appropriately perturbed to maintain this ambiguity.    
    }\label{fig:1dambiguity}
\end{figure}

Theorem~\ref{thm:punchline} fails for $d=1$. A simple counterexample
to the first part of the theorem for the path case is shown in Figure~\ref{fig:1dambiguity}:
Let $\p_1 < \p_2 < \p_3$ be three generic points on the line.
Let $\alpha_1$ measure the edge $[1,2]$, 
$\alpha_2$ measure the edge $[2,3]$ and 
$\alpha_3$ measure the edge $[1,3]$.
This ensemble clearly allows for trilateration.
In this case we will have 
$\v = \la \pmb{\alpha},\p \ra =
 [ \p_2-\p_1, \p_3-\p_2, \p_3-\p_1]$. 
Now let 
$\q_1$ be arbitrary, set 
$\q_2 := \q_1 + (\p_2-\p_1) - 1/2 (\p_3-\p_1)$ and 
$\q_3 := \q_1 + 1/2(\p_3-\p_1)$. 
This will give us
$\q_3-\q_2=  (\p_3-\p_1) - (\p_2-\p_1)=
\p_3-\p_2$. Let us also assume that $\p_3-\p_2<\p_2-\p_1$, then this will give us the ordering: $\q_1 < \q_2 < \q_3$. 
Now, let $\beta_1$ measure the path $[2,1,3]$, 
$\beta_2$ measure the edge $[2,3]$, and
$\beta_3$ measure the path $[1,3,1]$. 
Then in this case, we will also get 
$\v=\la \pmb{\beta},\q \ra$. But our two measurement ensembles
are not related by a scale. 
Since we cannot uniquely
reconstruct a triangle on the line, this will kill off 
any attempts at using trilateration for reconstruction.

In the language we develop later, 
the failure of this example essentially happens due to the 
fact that the variety $L_{1,3}$ is reducible, and
thus Theorem~\ref{thm:prin2} does not apply. The relationship between
these
$\pmb{\alpha}$ and 
$\pmb{\beta}$ 
is not described by an automorphism of $L_{1,3}$ but instead,
only by an automorphism of one of its (planar) components. 

\begin{figure}[ht]
	\centering
	\def\svgwidth{.9\textwidth}
	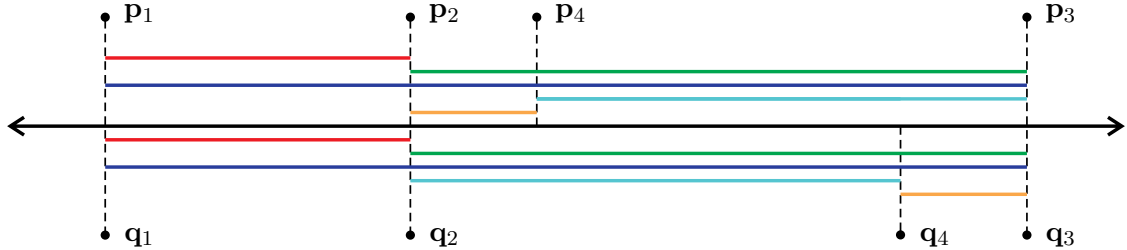
	\caption{2-flips ambiguity.    
    }\label{fig:2flips}
\end{figure}

For the second part of the theorem, where
$\pmb{\beta}$ must also allow for trilateration, we can still
find counterexamples due the fact that unlabeled trilateration
from edge lengths fails in one dimension. In particular,
Lemma~\ref{lem:extend} below does not hold. 
This is shown in Figure~\ref{fig:2flips}: Let $\p$ consist of $4$ points on a line and
$\pmb{\alpha}$ 
consist of $5$ of the $6$ possible edges. In this case, there
is one vertex, say $\p_4$, with only two measured edges, say
$[2,4]$ and $[3,4]$. 
If $\pmb{\beta}$ is obtained from 
$\pmb{\alpha}$ by simply swapping the order of these two edges,
we can maintain $\v$ by appropriately re-locating 
the fourth point. The edge mapping here between 
$\pmb{\alpha}$ and
$\pmb{\beta}$ is
an example of a $2$-flip~\cite{miso}.

Our approach for proving Theorem~\ref{thm:punchline}
will be constructive and provide the basis for a 
computational attack on this reconstruction process. In particular, we will establish the following.
\begin{mdframed}
Let $\p$ be generic in $d \ge 2$ dimensions,
let
$\pmb{\alpha}$ be a $b$-bounded
path (resp. loop) measurement ensemble
that allows for 
trilateration. 
Suppose $\v = \la \pmb{\alpha}, \p \ra$. Then, given $\v$, there
is a 
trilateration-based algorithm, over a real computation model,
that reconstructs $\p$ up to congruence and vertex labeling.
The algorithm will succeed for any 
$\p$ in some Zariski open subset of configurations
that depends only on $b$.
\end{mdframed}

For fixed $d$,  this algorithm (over a real computation model)
will have worst case time 
complexity that  is polynomial 
in $(|\v|, b)$, though with a moderately large exponent.
Unfortunately, 
for  $d = 3$, the complexity includes 
a factor of $b^{90}$,
unless we add some strong extra 
assumptions on 
$\pmb{\alpha}$.
For $d=2$, this factor scales with the more
reasonable $b^6$. 

\subsection{Organization}
The rest of the paper is structured as follows.  Section 
\ref{sec:varieties} develops, in some detail, the properties of 
two algebraic varieties relating to measurements of pairwise
distances between points.  These ``squared'' and ``unsquared''
measurement varieties are the natural setting for our reconstruction
problem.  Some of the structural results are interesting in their 
own right.

In Section~\ref{sec:warm}, we revisit the Boutin-Kemper
\cite{BK1} problem of reconstruction from unlabeled edge
length measurements.  We provide a new proof of a small 
generalization, and, along the way, introduce the techniques
we use to solve the more general path and loop problem. Setting
up the edge measurement ensemble in our language also makes clear
what is new about our setting.

Sections~\ref{sec:maps} and~\ref{sec:Ldn-automorphisms} contain
our key technical results, which classify the linear maps between 
measurement varieties and their linear automorphisms.  It turns 
out that dimension $d=2$ is the most interesting and difficult case,
due to the presence of Regge symmetries~\cite{regge}. The corresponding 
part of the proof makes use of computer algebra to verify that these
extra symmetries do not cause problems in our reconstruction 
application. Our computer algebra script is available
as a supplemental document.

The main results are then proved in Sections~\ref{sec:consist} and 
\ref{sec:ugr}. We conclude with a discussion of reconstruction 
procedures in Section~\ref{sec:recon}; we will provide complete 
detail about the algorithms in a companion document.

To keep the paper self-contained, Appendix~\ref{sec:geometry}
presents the essential algebraic-geometric background. Appendix~\ref{sec:det} contains the proof of a lemma about determinants and sign flips, which we will need for our classification of linear maps. Appendix~\ref{sec:rational:funcs} establishes several results about rational functionals of generic point configurations, which we use throughout the paper. Finally, Appendix~\ref{sec:fano} discusses properties of the Fano varieties of the $\LL$ variety, which will be important for our reconstruction algorithm.




\section{Measurement Varieties}\label{sec:varieties}

In this section, we will study the basic properties
of two related families of varieties, the squared and unsquared
measurement varieties. The structure of these varieties
will be critical to understanding 
the problem of 
reconstruction from
unlabeled measurements.

The squared variety is very well studied
in the literature, but the unsquared variety is much less so.
Since we are interested in integer sums of unsquared 
edge lengths, we will need to understand the structure of this
unsquared variety.

Although we are ultimately interested in measuring real lengths
in Euclidean space, we will pass to the complex setting where we can
utilize some tools from algebraic geometry.

\begin{definition}
\label{def:sms}
Let us index the coordinates of $\CC^N$ as $ij$, with
$i < j$ and both between $1$ and $n$.  We also fix an ordering 
on the $ij$ pairs to index the coordinates of $\CC^N$ as $i$ with 
$i$ between $1$ and $N$.\footnote{This ordering choice does not matter as long 
as we are consistent.  It is there to lets us switch between coordinates 
indexed by edges of $K_n$ and indexed using flat vector notation.  For $n=4$, $N=6$ we will 
use the order: $12,13,23,14,24,34$.}
\end{definition}

Let us begin with a 
\defn{complex configuration} $\p$ of $n$ points in $\CC^d$
with $d \geq 1$. We will always assume  $n \geq d+2$.
There are $\edgecard$ vertex pairs (edges), along which we can measure
the complex \emph{squared} length as 
\ba
m_{ij}(\p) := \sum_{k=1}^{d}(\p^k_i-\p^k_j)^2
\ea
where $k$ indexes over the $d$ dimension-coordinates. Here, we measure
complex squared length using the complex square operation with no
conjugation. We consider the vector $[m_{ij}(\p)]$ over all of the vertex 
pairs, with $i<j$,
as a single point in $\CC^{\edgecard}$, which we denote as $m(\p)$.

\begin{definition}
Let $M_{d,n}\subset \CC^{\edgecard}$ be the
the image of $m(\cdot)$ over all $n$-point complex configurations in $\CC^d$. 
We call this the \defn{squared measurement variety} of $n$ points
in $d$ dimensions.
\end{definition}
When $n \le (d+1)$, then $M_{d,n}= \CC^{\edgecard}$.

\begin{definition}
If we restrict the domain to be real configurations, then 
we call the image under $m(\cdot)$ the \defn{Euclidean squared measurement set} denoted as 
$M^{\EE}_{d,n} \subset \RR^{\edgecard}$.
This set has real dimension 
$dn-C$. 
\end{definition}

The following theorem reviews some basic facts.
Most of the  ideas are discussed 
in~\cite{ciprian}, but we include a detailed proof here for completeness and ease of reference.

\begin{theorem}
\label{thm:Mvariety}
Let $n \ge d+2$.
The set $M_{d,n}$
is linearly isomorphic to 
$\CS^{n-1}_d$,
the variety of complex, symmetric 
$(n-1)\times(n-1)$
matrices of rank $d$ or less.
Thus, $M_{d,n}$ is a variety, and also defined over $\QQ$.
It is irreducible.
Its dimension is $dn-C$.
Its singular set $\sing(M_{d,n})$ consists of squared measurements of configurations
with affine spans of dimension strictly less than $d$.
If $\p$ is a generic complex configuration in $\CC^d$ or a 
generic configuration
in $\RR^d$, 
then $m(\p)$ is generic in 
$M_{d,n}$.
\end{theorem}

\begin{proof}
Such an isomorphism is developed in~\cite{YH38}
and further, for example, in~\cite{gower}, see also
\cite[Section 7]{cgr}. The basic idea is as follows. We can, wlog,
translate the entire complex configuration $\p$ 
in $\CC^d$ such that the last
point $\p_n$ is at the origin. We can then think of this as a
configuration of $n-1$ vectors in $\CC^d$.  Any such complex
configuration gives rise to
a symmetric 
$(n-1)\times(n-1)$
complex Gram matrix (where no conjugation is used),
$G(\p)$, of rank at most
$d$.  Conversely, any symmetric complex matrix $\G$ of rank $d$ or
less can be (Tagaki) factorized, giving rise to a complex configuration of
$n-1$ vectors in $\CC^d$, which, along with the origin, gives us an $n$-point
complex configuration $\p$ so that $\G=G(\p)$.

With this in place, 
let $\varphi$ be the 
invertible linear map 
from the space of 
$(n-1)\times(n-1)$ 
symmetric complex
matrices $\G$, to $\CC^{\edgecard}$
(indexed by vertex pairs $ij$, with $i<j$)
defined as
$\varphi(\G)_{ij}:= G_{ii} + G_{jj} -2G_{ij}$
(where $G_{in}$ and $G_{nj}$ is interpreted as
$0$).
(For invertibility see~\cite[Lemma 7]{cgr}.) 

When $\G=G(\p)$ is the gram matrix of a 
complex configuration $\p$ in $\CC^d$, then $\varphi(\G)$ 
computes the squared edge lengths of $\p$.
Since every rank-$d$ constrained matrix $\G$ arises 
as the Gram matrix, $G(\p)$ from some complex
configuration $\p$ in $\CC^d$, we see that the image of $\varphi$ 
acting on 
$\CS^{n-1}_d$,
is contained in 
$M_{d,n}$.
Conversely, 
since every point in $M_{d,n}$ arises from 
a complex configuration $\p$, and $\p$ gives rise to a Gram matrix 
$G(\p)$, we see that the image of $\varphi$ acting on rank 
constrained matrices is onto $M_{d,n}$.
This gives us our isomorphism of varieties (Lemma~\ref{lem:bij}.)

Irreducibility of $M_{d,n}$ follows from the fact that 
it is the image of an affine space (complex configuration space)
under a polynomial (the squared-length map).
To get the dimension, the above isomorphism, along with
the existence and uniqueness of the spectral 
decomposition, gives that the dimension is 
$d(n-1) - \binom{d}{2}$, which is what we want.

For the description of the singular set of rank-constrained matrices,
see for example~\cite[Page 184]{harris} 
(which can also be applied to
the symmetric case). Meanwhile, we know that
$\G=G(\p)$ has rank $<d$
iff $\p$ has a deficient affine span 
in $\CC^d$ (see for example~\cite[Lemma 26]{cgr}). 

The statement on genericity follows from Lemma~\ref{lem:genPush}.
\end{proof}

\begin{remark}
We note, but will not need, the following:
For $d\ge 1$,
the smallest complex variety containing   
$M^{\EE}_{d,n}$ is $M_{d,n}$.
\end{remark}
We note the following minimal instances where $n=d+2$.
In these cases, the variety has codimension $1$.

The variety $M_{1,3} \subset \CC^3$ is defined by the vanishing of the \defn{simplicial volume determinant}, that is, the determinant of the following matrix
\ba
\begin{pmatrix}
2m_{13}& (m_{13} + m_{23} - m_{12})\\
(m_{13} + m_{23} - m_{12})& 2m_{23}\\
\end{pmatrix}
\ea
where we use $(m_{12}, m_{13}, m_{23})$ to represent the coordinates of $\CC^3$. This is the Gram matrix, $\varphi^{-1}(m(\p))$, described in the proof of Theorem~\ref{thm:Mvariety}. 

The variety $M_{2,4} \subset \CC^6$ is defined by the vanishing of the 
determinant of the matrix
\ba
\begin{pmatrix}
2m_{14}& (m_{14} + m_{24} - m_{12})& (m_{14} + 
m_{34} - m_{13})\\
(m_{14} + m_{24} - m_{12})& 2m_{24}& (m_{24} + m_{34} - m_{23})\\
(m_{14} + m_{34} - m_{13})& (m_{24} + m_{34} - m_{23})& 
2m_{34}
\end{pmatrix}.
\ea

The variety $M_{3,5} \subset \CC^{10}$ is defined by the vanishing of the determinant of the matrix
\ba
\begin{pmatrix}
2m_{15}& (m_{15} + m_{25} - m_{12})& (m_{15} + 
m_{35} - m_{13})&  (m_{15} + m_{45} - m_{14})\\
(m_{15} + m_{25} - m_{12})& 2m_{25}& (m_{25} + m_{35} - m_{23})&  
(m_{25} + m_{45} - m_{24})\\
(m_{15} + m_{35} - m_{13})& (m_{25} + m_{35} - m_{23})& 
2m_{35}&  (m_{35} + m_{45} - m_{34})\\
(m_{15} + m_{45} - m_{14})& (m_{25} + m_{45} - m_{24})&
 (m_{35} + m_{45} - m_{34})& 2m_{45}
\end{pmatrix}.
\ea
These same polynomial calculations can be done by 
constructing the Cayley-Menger determinants.

When $n>d+2$, then $M_{d,n}$ has higher codimension, and 
requires the simultaneous vanishing of more
than one minor, characterizing the rank $d$.

Next we move on to unsquared lengths.


\begin{definition}
Let the \defn{squaring map} $s(\cdot)$ be the map from $\CC^{\edgecard}$ 
onto $\CC^{\edgecard}$ that
acts by squaring each of the $\edgecard$ coordinates of a point.
Let $L_{d,n}$ be the preimage of $M_{d,n}$ under the squaring map.
(Each point in $M_{d,n}$ has $2^{\edgecard}$ preimages in $L_{d,n}$, arising
through coordinate negations).
We call this the \defn{unsquared measurement variety} of $n$ points
in $d$ dimensions.
\end{definition}

\begin{definition}
We can define the \defn{Euclidean length map}
of a real configuration $\p$ as
\ba
l_{ij}(\p) := \sqrt{\sum_{k=1}^{d}(\p^k_i-\p^k_j)^2}
\ea
where we use the positive square root.
We call the image of $\p$ under
$l$ the \defn{Euclidean unsquared measurement set} denoted as 
$L^{\EE}_{d,n} \subset \RR^{\edgecard}$.
Under the squaring map, we get
$M^{\EE}_{d,n}$. 
We denote by $l(\p)$,
the vector $[l_{ij}(\p)]$ over all vertex 
pairs. We may consider $l(\p)$ either as a point in 
the real valued $L^{\EE}_{d,n}$
or as a point in 
the complex variety $L_{d,n}$.
\end{definition}

Indeed, $L^{\EE}_{d,n}$
is the set we are truly interested in,
but it will be easier to work with the whole variety $L_{d,n}$. For example,  Theorem~\ref{thm:prin2}
requires
us to work with varieties, and not, say,
with real ``semi-algebraic sets''.
Also, the proof of Proposition~\ref{prop:irr} will require us
to work in the complex domain.

\begin{remark}
The locus of $\LL$ where the edge lengths of a triangle,
$(l_{12}, l_{13}, l_{23})$, are held fixed is studied in 
beautiful detail in~\cite{marco}, where this is shown to be a Kummer surface.
\end{remark}

The following theorem is the main result of this section.
\begin{theorem}
\label{thm:Lvariety}
Let $n \ge d+2$.
$\LN$ is a variety. It has pure dimension
$dn-C$.
Assuming that $d \geq 2$, we also have the following:
$\LN$ is irreducible.
If $\bm$ is generic in $M_{d,n}$, then each point in 
$s^{-1}(\bm)$ is generic in $\LN$. If $\p$ is a generic configuration in 
$\RR^d$, then $l(\p)$ is generic in $\LN$.
\end{theorem}
The proof is in the next subsection.
The non-trivial  part will be showing irreducibility, which 
we will do in Proposition~\ref{prop:irr} below.
Indeed,
in one dimension, the variety $L_{1,3}$ is reducible
and thus also has no generic points. We elaborate on this below.

\begin{remark}
We note, but will not need the following:
For $d\geq 2$,
the smallest complex variety containing   
$L^{\EE}_{d,n}$ is $L_{d,n}$.
\end{remark}

Returning to our minimal examples:
The variety $L_{1,3} \subset \CC^3$ is defined by the vanishing of the determinant of the
following matrix
\ba
\begin{pmatrix}
2l^2_{13}& (l^2_{13} + l^2_{23} - l^2_{12})
\\
(l^2_{13} + l^2_{23} - l^2_{12})& 2l^2_{23}
\end{pmatrix}
\ea
where we use $(l_{12}, l_{13}, l_{23})$ to represent the coordinates of $\CC^3$.

The variety $L_{2,4} \subset \CC^6$ is defined by the vanishing of the determinant of the matrix
\ba
\begin{pmatrix}
2l^2_{14}& (l^2_{14} + l^2_{24} - l^2_{12})& (l^2_{14} + 
l^2_{34} - l^2_{13})\\
(l^2_{14} + l^2_{24} - l^2_{12})& 2l^2_{24}& (l^2_{24} + l^2_{34} - l^2_{23})
\\
(l^2_{14} + l^2_{34} - l^2_{13})& (l^2_{24} + l^2_{34} - l^2_{23})& 
2l^2_{34}
\end{pmatrix}.
\ea

The variety $L_{3,5}\subset \CC^{10}$ is defined by the vanishing of the determinant of the matrix
\ba
\begin{pmatrix}
2l^2_{15}& (l^2_{15} + l^2_{25} - l^2_{12})& (l^2_{15} + 
l^2_{35} - l^2_{13})&  (l^2_{15} + l^2_{45} - l^2_{14})\\
(l^2_{15} + l^2_{25} - l^2_{12})& 2l^2_{25}& (l^2_{25} + l^2_{35} - l^2_{23})&  
(l^2_{25} + l^2_{45} - l^2_{24})\\
(l^2_{15} + l^2_{35} - l^2_{13})& (l^2_{25} + l^2_{35} - l^2_{23})& 
2l^2_{35}&  (l^2_{35} + l^2_{45} - l^2_{34})\\
(l^2_{15} + l^2_{45} - l^2_{14})& (l^2_{25} + l^2_{45} - l^2_{24})&
 (l^2_{35} + l^2_{45} - l^2_{34})& 2l^2_{45}
\end{pmatrix}.
\ea

\begin{figure}[ht]
	\begin{center}
		\includegraphics[width=2.5in]{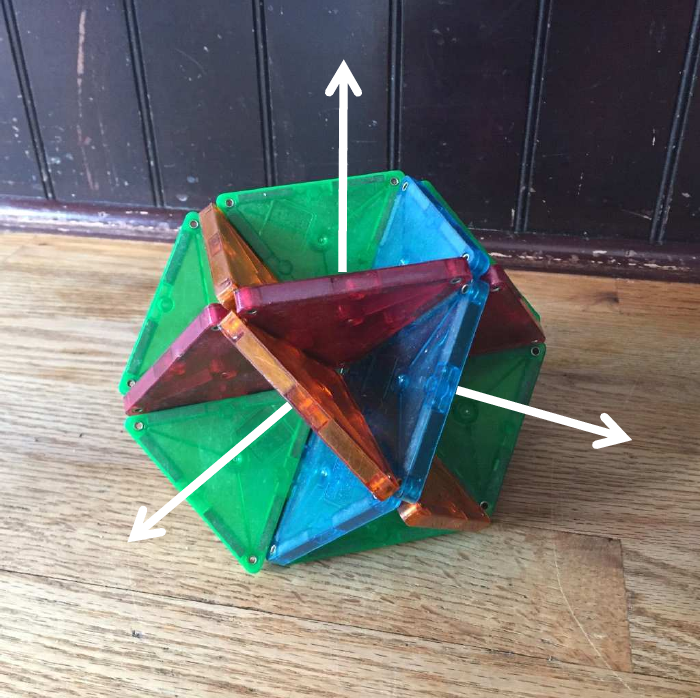}
	\end{center}
	\caption{A model of the real locus of $L_{1,3}$, a subset of $\RR^3$. It comprises $4$ planes. Coordinate axes are in white.}\label{fig:l13}
\end{figure}

\begin{remark}
It turns out that $L_{1,3}$ is reducible and consists of the four hyperspaces defined,
respectively, by the vanishing of one of the following equations:
\ba
l_{12} + l_{23} - l_{13} \\
l_{12} - l_{23} + l_{13} \\
-l_{12} +l_{23} + l_{13} \\
l_{12} + l_{23} + l_{13} 
\ea
This reducibility can make the one-dimensional case quite different from dimensions 2 and 3, as already discussed in Section~\ref{sec:results}. See also Figure~\ref{fig:l13}.

Notice that the first octant of the real locus of $3$ of these hyperspaces 
arises as the Euclidean lengths of a triangle in $\RR^1$ (that is,
these make up $L^{\EE}_{1,3}$). 
The specific hyperplane
is determined by the order of the $3$ points on the line.
\end{remark}

At this point, we would like to generalize our notion of 
measurement ensembles from Definition~\ref{def:graph}.

\begin{definition}\label{def:functionals}
A \defn{length functional} $\alpha$ is a linear mapping from $\LN$ to $\CC$.
We write its application to $\bl \in \LN$ as $\la \alpha,\bl \ra$.
In coordinates, it has the form $\sum_{ij} \alpha^{ij} l_{ij}$, with $\alpha^{ij} \in \CC$.
When $\p$ is a real configuration, and thus $l(\p)$ is well defined,
then we can also define
$\la \alpha, \p \ra:= \la \alpha, l(\p)\ra$.

We say that a length functional is \defn{rational} if all of its coordinates are in $\QQ$.
We say it is \defn{non-negative} if all of its coordinates are non-negative.
We say it is \defn{integer} if all of its coordinates are integer numbers.
We say it is \defn{whole} if all of its coordinates are whole numbers.
We say that an integer or whole length functional 
is \defn{$b$-bounded} if all of its coordinates have magnitudes  
no greater than $b$.

Given a sequence of $k$ length functionals $\alpha_i$, we define its
\defn{ensemble  matrix} $\E$ as the $k\times \edgecard$ matrix
whose $i$-th row is equal to the coordinates of the $i$-th length
functional. The ensemble matrix gives rise to a linear map from $\LN$ to
$\CC^k$. The definitions of non-negative, integer, whole, and $b$-bounded length functionals
can be extended to such an $\E$ by enforcing their respective coordinate conditions on all rows of $\E$.

A path or loop (as in Definition~\ref{def:graph}) gives rise
to a unique whole length functional. Analogously, a path or loop measurement ensemble gives
rise to a unique whole length ensemble matrix.

\end{definition}

\subsection{Proof}
\label{sec:lproof}

We will now develop the proof of Theorem~\ref{thm:Lvariety}. The main issue will be
proving the irreducibility of $\LN$. 
The special case of $n=d+2$ follows from~\cite{cmirr}, but
we are interested in the general case, $n \ge d+2$.
The basic idea we will use is that a variety whose smooth locus
is connected must be irreducible.
More specifically, 
our strategy is to define a ``good'' locus of points in 
$\LN$, and show that this locus is connected,
made up of smooth points, and is Zariski dense in $\LN$.
This, along with Theorem~\ref{thm:conIrr}, will prove irreducibility.

We will show connectivity using a specific path construction. This
will rely centrally on the complex setting that we have placed ourselves in.  
Showing (algebraic) smoothness will mostly be a technical matter.

\begin{definition}
Let the \defn{zero} locus $Z$ of $\CC^N$ be the points where at least one
coordinate vanishes.

Let the \defn{bad} locus $\bad(M_{d,n})$ of $M_{d,n}$
be the union of its singular locus
$\sing(M_{d,n})$ together with the points in $M_{d,n}$ that are in $Z$.
We will call the remaining locus $\good(M_{d,n})$ \defn{good}.

Let the \defn{bad} locus $\bad(L_{d,n})$ of $L_{d,n}$ be the 
preimage of the bad locus of $M_{d,n}$ under the squaring map $s$.
We will call the remaining locus $\good(L_{d,n})$ \defn{good}. 

We refer to points on the good locus
as \defn{good} points, and analogously for \defn{bad} points.
\end{definition}

\begin{lemma}
\label{lem:Mcon}
$\good(M_{d,n})$ is path-connected.
\end{lemma}
\begin{proof}
Let $\bm_1$ and $\bm_2$ be any two 
good points in $M_{d,n}$. These correspond to 
two configurations $\p$ and $\q$. A path in configuration space, connecting $\p$ to $\q$,
will remain, under $m(\cdot)$, on $\good(M_{d,n})$
when the affine span of the configuration does not drop in dimension,
and no edge between any two points has zero squared length.
This can always be done, as we have $n \geq d+2$ points.
(This is even true for one-dimensional configurations
in the complex setting, as a zero squared length is a condition
that has complex-codimension of at least $1$, and thus the bad locus is non-separating.)
\end{proof}

We next record a lemma that follows from basic results
of covering space theory.  
See 
\cite[Sections 53, 54]{munk} for more details.
\begin{definition}
A \defn{path} 
$\tau$ on a space $X$ is a continuous  map from the unit interval to $X$.
A \defn{loop} is a path with $\tau(0)=\tau(1)$.
Let $p$ be a map from a space $\tilde{X}$ to $X$.
A \defn{lift} $\tilde{\tau}$ of $\tau$ (under $p$) is a map
such that $p(\tilde{\tau})=\tau$. It is a path on $\tilde{X}$.
\end{definition}
Intuitively, a lift is just tracing out 
the path $\tau$ in the preimage through $p$. In what follows, 
$\CC^\times$ is the punctured complex plane.
\begin{lemma}\label{lem:z2-cover}
Let $p$ be the map $\CC^\times\to \CC^\times$ given by $z\mapsto z^2$.
Let $x:=p(z)$.
A loop $\tau$ starting
at $x$ uniquely lifts to a loop $\tilde{\tau}$ starting at $z$
if $\tau$ winds 
around the origin an even number of times, and otherwise 
it lifts to a path that ends at $-z$.
\end{lemma}
\begin{proof}[Proof sketch]
See 
\cite[Chapters 53, 54]{munk} for definitions.
The map $\CC^\times\to \CC^\times$ given by $z\mapsto z^2$
is a  covering map. 
Call the base $B$ and the cover $F$ and the 
covering map $p$.  
Each loop $\tau$ in $B$, starting at $x$, lifts uniquely to a path $\tilde{\tau}$
in $F$, starting at $z$.
The path $\tilde{\tau}$
ends at a uniquely defined point $z' \in p^{-1}(x)$
under the \defn{lifting correspondence}. 
In our case the fiber is $\{z,-z\}$. Moreover every $z'$
in the fiber can be reached under the lifting of some loop $\tau$  (see~\cite[Theorem 54.4]{munk}).

The fundamental group of the base is
$\pi_1(B) = 
\pi_1(\CC^\times)\cong \ZZ$.
The covering map determines an
induced map
$p_* : \pi_1(F)\to \pi_1(B)$. The image of the induced map
consists of 
loops that wind around the origin an even number of 
times in $F$ so it is isomorphic to $2\ZZ$.
The lifting correspondence induces a bijective  map from the 
group $\pi_1(B)/p_*(\pi_1(F))\cong \ZZ_2$ to the fiber above $x$, and (only) loops in 
$p_*(\pi_1(F))$ lift to loops in $F$.
(see~\cite[Theorem 54.6]{munk}). 

Thus, this lift, starting from $z$,
is a path from $z$ to $-z$ if and only if 
$\tau$ 
winds around the origin an odd number of 
times.
\end{proof}

\begin{figure}[ht]
	\centering
	\def\svgwidth{2.5in}
	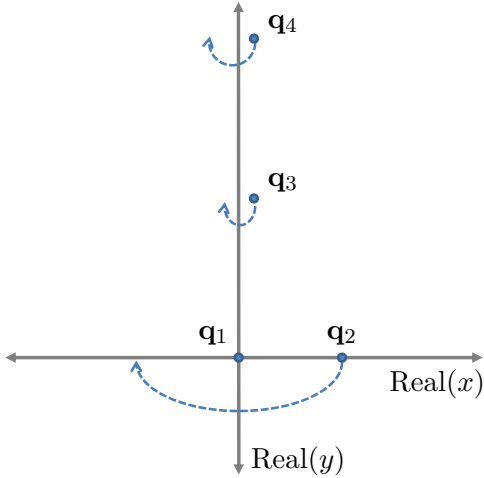
	\caption{Our gadget. The imaginary $x$-direction is coming out of the page. Our path ends with 
the reflection of the configuration $\q$ along the $x$-axis.}\label{fig:gadget}
\end{figure}

\begin{figure}[ht]
	\centering
	\def\svgwidth{0.8\textwidth}
	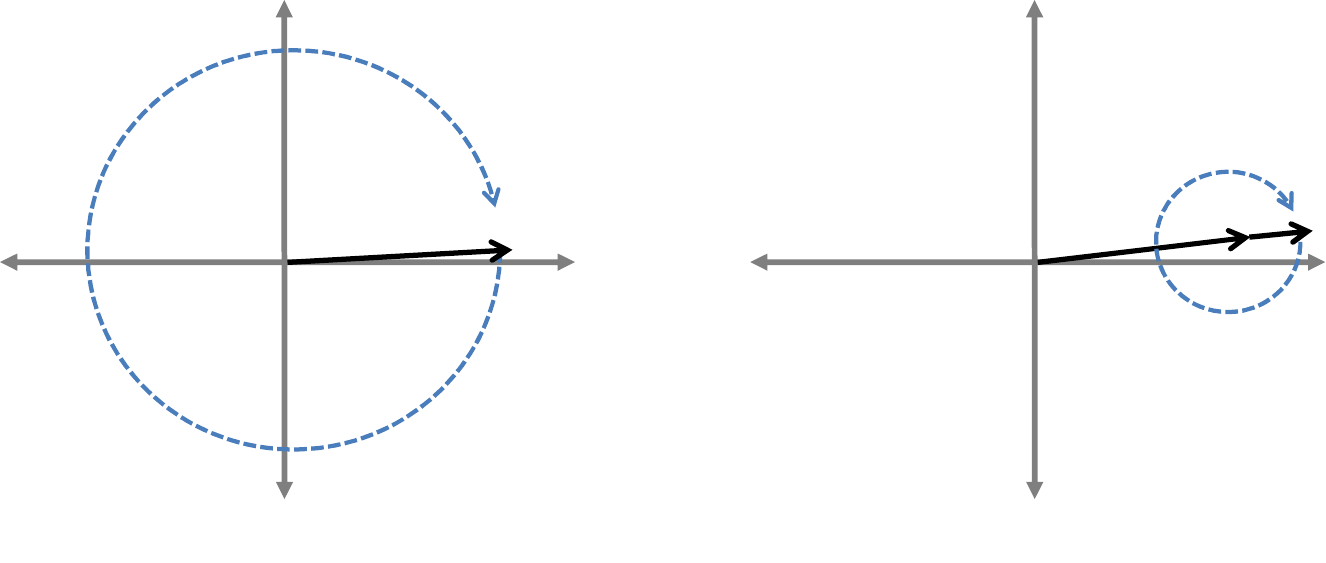
	\caption{Since the squared length along edge $\{1,2\}$ arises from its $x$ component, our path along this edge measurement winds once about the origin in $\CC$. For any other edge, the $x$ component of the squared distance is dominated by the other coordinates and the resulting path stays far from the origin in $\CC$.}\label{fig:wind}
\end{figure}

Looking at the product space $(\CC^{\times})^\edgecard$, we can also view the squaring map
$s$ as a covering map mapping this product space to itself, 
and we can apply Lemma~\ref{lem:z2-cover} coordinate-wise.

\begin{lemma}
\label{lem:Lneg}
Assume $d \ge 2$.
Suppose $\bl$ and $\bl'$ are two points in $\LN$ that 
differ only by a negation along one
coordinate. Then, there is a path that connects 
$\bl$ to $\bl'$ and stays in $\good(L_{d,n})$.
\end{lemma}
\begin{proof}
W.l.o.g., we will negate the coordinate corresponding to the edge
lengths between vertices $1$ and $2$.
But first, we need to develop a little gadget.

Let $\q$ be a special configuration with the following properties:
$\q_1$ is at the origin, $\q_2$ is placed one unit along the first axis of 
$\CC^d$; and the remaining points are arranged so that they all lie
within $\epsilon$ of the second axis in $\CC^d$, but such that
they are greater than one unit apart along the second axis
from each other and also from $\q_1$. 
(Note that this step requires that $d\geq 2$.)
Moreover we choose the remaining
points so that $\q$ has a full $d$-dimensional affine span.
This configuration has the following property:
the squared distances of all of the edges are dominated by the
contribution from the second coordinate, except for the squared distance
along the edge $\{1,2\}$, which is dominated by the contribution
from its first coordinate. See Figure~\ref{fig:gadget}.

Let $a(t)$ be the path in configuration space, parameterized by
$t \in [0,\pi]$ where, for each $i$, we multiply the first
coordinate of $\q_i$ by $e^{-t\sqrt{-1}}$. 
This path ends at 
$a(\pi)$, 
a configuration
which is a reflection of $\q$. 

Under $m$, this gives us
a loop $\tau:=m(a)$ in $M_{d,n}$ that starts and ends at the point
$\y:= m(\q)$. 
By construction, the loop
$\tau$ avoids any
singularities or vanishing coordinates.
Fixing one point $\z$ in $s^{-1}(\y)$, the loop $\tau$ lifts
to a path $\tilde{\tau}$ in $L_{d,n}$
that ends at some 
point $\z'$  in 
the fiber $s^{-1}(\y)$. Moreover, this path remains in $\good(L_{d,n})$. 

If we project $\tau$
onto the coordinate of $\CC^\edgecard$ corresponding
to the edge $\{1,2\}$, we see that the image
maps to a loop that
winds around the origin of $\CC$ exactly once. 
If we project this loop onto any
of the other coordinates, we obtain a loop  that cannot wind
about the origin of $\CC$ at all. See Figure~\ref{fig:wind}.
By Lemma~\ref{lem:z2-cover},
the lifted loop $\tilde{\tau}$ 
in $\LN$ must end at the point $\z'$ that arises from
$\z$ by negating the first coordinate.

Going now back to our problem, let $\p$ be any configuration such that 
$m(\p) = s(\bl)$. Let $w$ be a configuration path 
from $\p$ to our special $\q$. 
Let $\omega := m(w)$. From  Lemma~\ref{lem:Mcon} this path can 
be chosen to avoid any singular points or points where a coordinate
vanishes.
Let the concatenated path
$\sigma$ be $\omega^{-1} \circ \tau \circ \omega$.
This is a loop in $M_{d,n}$ that starts and ends at $m(\p)$.
The projection of $\sigma$
onto the coordinate of $\CC^\edgecard$ corresponding
to the edge $\{1,2\}$, defined by forgetting all other 
coordinates, 
winds around the origin
exactly once (any loops due to $\omega$ cancel out),
while the other coordinate projections are simply connected in 
$\CC^{\times}$ (any loops due to $\omega$ cancel out). Thus, fixing the point $\bl$ in $\LN$,
from Lemma~\ref{lem:z2-cover},
$\sigma$ must lift to a path $\tilde{\sigma}$ that ends at 
$\bl'$. Moreover, this path stays in the good locus.
\end{proof}

\begin{lemma}
\label{lem:Lcon}
For $d \ge 2$, $\good(L_{d,n})$ is path-connected.
\end{lemma}
\begin{proof}
Let $\bl_1$ and $\bl_2$ be two good points in $\good(L_{d,n})$.
Define $\bm_i := s(\bl_i)$. 
Let $\tau$ be a path 
in $M_{d,n}$
from 
$\bm_1$ to $\bm_2$ that avoids the singular set of $M_{d,n}$,
and such that no coordinate ever vanishes
(as guaranteed by~\ref{lem:Mcon}).
Fixing $\bl_1$, the path $\tau$
lifts to a path $\tilde{\tau}$ in $\LN$
that remains in the good locus 
and that connects $\bl_1$ to some point  $\bl_2'$ 
in the fiber 
$s^{-1}(s(\bl_2))$. 
The only remaining issue is that
$\bl_2'$ may have some of its coordinates negated from 
our desired target point $\bl_2$.
This can be solved by repeatedly applying the good negating paths
guaranteed by Lemma~\ref{lem:Lneg}.
\end{proof}

We now move on to the technical matters of smoothness.

\begin{lemma}
\label{lem:Asmooth}
Every point $\bl \in \good(\LN)$ is 
smooth
and with 
$\Dim_\bl(\LN)=dn-C$.
Every point in $\bad(\LN)-Z$ is singular.
\end{lemma}
\begin{proof}
Every good point  in $M_{d,n}$ is
(algebraically) smooth, and thus, from Theorem~\ref{thm:smpt}, is 
analytically smooth 
of dimension $dn-C$. 
Also, from Theorem~\ref{thm:smpt}, 
every singular point in $M_{d,n}$ is not analytically smooth.

The differential $\ud s$ of the squaring map $s$ on $\CC^\edgecard$ is
represented by an $\edgecard\times \edgecard$ Jacobian matrix $\J$ at
each point in $\CC^\edgecard$. At points in $\CC^\edgecard$ where none
of the coordinates vanish, $\J$ is invertible.
Thus, from the inverse function theorem,  every good point in $\LN$ is
analytically smooth 
of dimension $dn-C$. 
Also every bad point in $\LN-Z$ is not analytically smooth.

Again using Theorem~\ref{thm:smpt}, we have each good point 
(algebraically) smooth
and with 
$\Dim_\bl(\LN)=dn-C$.
Similarly, we also have that 
every bad point in $\LN-Z$ is singular.
\end{proof}
Note that there may be some bad points of $\LN$ in $Z$ that are still smooth.

\begin{remark}
The above lemma can be proven directly using more machinery from algebraic geometry.
In particular, away from $Z$, the squaring map from $\CC^N$ to itself is an 
``\'etale morphism''~\cite[page 18]{milne-etale}. 
This property transfers
to the map $s(\cdot)$ acting on $\LN-Z$,
as this property transfers under a ``base change''.
 The results then follows immediately.
\end{remark}

\begin{lemma}
\label{lem:closeG}
The Zariski closure of $\good(L_{d,n})$ is $\LN$.
\end{lemma}
\begin{proof}
Recall the following principle:
Given any point $z$ in $\CC^\times$, we can always find a 
neighborhood $B$
of $z^2$,
so that there is a 
well defined, single valued, continuous 
square root function
from $B$ to $\CC$, with $\sqrt{z^2}=z$.

Returning to our setting,
let $\bl$ be any point in $\LN$, and
let $\bm:=s(\bl)$ be its image in 
$M_{d,n}$ under the coordinate squaring map.
The good points of $M_{d,n}$ are  dense
in $M_{d,n}$. (Letting
$\bm=m(\p)$ for some $\p$,
there is always a nearby 
configuration $\p'$ with a full span and no edge with vanishing
squared length. Moreover, the 
map $m(\cdot)$ is continuous.)
Thus we can always find an arbitrarily close
point $\bm'$ that is in $\good(M_{d,n})$.

Next we argue that we can find a point $\bl'$
such that $s(\bl')=\bm'$ (putting it in $\good(\LN)$)
with $\bl'$ is arbitrarily close to $\bl$. Given $\bm'$,
in order to determine $\bl'$ we need to select a ``sign'' for the 
square-root on each coordinate $ij$.
When $l_{ij} \neq 0$ then using the above principle, 
we can pick a sign so that $l'_{ij}$ is near to 
$l_{ij}$.
When $l_{ij}=0$ then we can use any sign to 
obtain an 
$l'_{ij}$ that is sufficiently close to $0$.

Since this can be done for each
$\bl$, then $\LN$ is in 
the standard-topology closure of $\good(L_{d,n})$.

Thus, from Theorem~\ref{thm:Zdense}, 
$\LN$ is in 
the Zariski
closure of $\good(L_{d,n})$.
Since $\LN$ itself is closed
and contains $\good(L_{d,n})$,
we are done.



\end{proof}

\begin{lemma}
\label{lem:dim} 
Every component of 
$\LN$ is of  dimension equal to 
$dn-C$.
\end{lemma}
\begin{proof}
From Lemma~\ref{lem:Asmooth}
each good point has a local dimension of  $dn-C$.
Thus, the good locus is covered by a set of components of $\LN$, all of dimension 
$dn-C$.
The Zariski closure of $\good(L_{d,n})$ is $\LN$ (Lemma~\ref{lem:closeG}).
Thus, no new components need to be added during the Zariski closure.
\end{proof}

We can now prove irreducibility.

\begin{proposition}
\label{prop:irr}
For $d\geq2$,
$\LN$ is irreducible.
\end{proposition}
\begin{proof}
From  Lemma~\ref{lem:Asmooth}, all of the points in $\good(L_{d,n})$ are smooth. 
From Lemma~\ref{lem:Lcon}, $\good(L_{d,n})$ is path-connected,
and thus connected in the subspace topology from $\CC^n$.

Next, we will use this smoothness and connectedness, together with 
Theorem~\ref{thm:conIrr}, to argue that 
all of $\good(L_{d,n})$ lies in one component $V_1$
of $\LN$. Suppose  we have the irreducible 
decomposition $\LN = \cup_i V_i$.
Let 
$G_i := \good(L_{d,n}) \cap V_i$. 
As varieties, the $V_i$  are closed subsets of $\CC^n$
(Theorem~\ref{thm:Zdense}), and thus the $G_i$
are closed subsets of $\good(L_{d,n})$ in the subspace topology.
Suppose that at least two
such $G_i$, say $G_1$ and $G_2$ are non-empty.
This would imply that there is a  pair of 
distinct
non-empty closed sets $G_1$ and $G_{>1}$ with union equal to $\good(L_{d,n})$.
Since $\good(L_{d,n})$ is connected, this implies that $G_1 \cap G_{>1}$
is non-empty.
But a smooth point in $\LN$ that is shared
between two components
would contradict Theorem~\ref{thm:conIrr}. Thus our claim is established.

The Zariski closure of $\good(L_{d,n})$ is $\LN$ (Lemma~\ref{lem:closeG}).
But since $\good(L_{d,n})$ is contained in the variety $V_1$, this closure must
be contained in $V_1$. Thus $\LN=V_1$, and $\LN$ must be irreducible.
\end{proof}

And now we can complete the proof of our theorem:

\begin{proof}[Proof of Theorem~\ref{thm:Lvariety}]
$\LN$ can be seen to be a variety by pulling back the 
defining equations of the variety $M_{d,n}$ through $s$.
Dimension is  Lemma~\ref{lem:dim}.
Irreducibility is Proposition~\ref{prop:irr}.
The statements on genericity follow from Lemma~\ref{lem:genPull}
and Theorem~\ref{thm:Mvariety}. 
\end{proof}

\section{Warm-up: The case of edge measurement ensembles}
\label{sec:warm}

As a warm-up for our main techniques, we will briefly look at
the case when our measurement ensemble consists only of
edge measurements. This case has been studied carefully in~\cite{BK1}.
Here we will look at these issues from the point of view of 
linear maps acting on the variety $M_{d,n}$.

We will start with the case where the measurement ensemble
consists of the complete edge set of cardinality $\edgecard$. 
This has been cleverly applied in~\cite{echo} to determine
the shape of a room from acoustic echo data.
Then we will consider the case of a trilateration 
ensemble of edges. This has been used
in~\cite{dux1} in the context of molecular scanning.

\begin{definition}
An \defn{edge measurement ensemble} 
$G:=\{G_1,\dots,G_k\}$ is a
finite sequence of distinct edges of $K_n$.
It is the same thing as a graph on $n$ vertices
with some ordering on its edges.
For an edge measurement ensemble 
$G$, we will write
$\la G, \p \ra^2$ to denote the sequence
of squared edge lengths.
\end{definition}

\subsection{Edge measurements of $K_n$ revisited}

We start with  a central result of
Boutin and Kemper~\cite{BK1}, stated 
in our terminology.

\begin{theorem}
\label{thm:bkMain}
Let $n \geq d+2$. 
Let $\p$ be a generic configuration of $n$ points
in $d$ dimensions. Let
$\v= \la G, \p \ra^2$, where 
$G$
is an edge measurement ensemble made up of exactly the $\edgecard$ edges of
$K_n$ in some order.

Suppose there is a 
configuration $\q$, 
also of $n$ points,
along with 
an edge
measurement ensemble $H$, 
where 
$H$
is an edge measurement ensemble made up of exactly the $\edgecard$ edges of
$K_n$, in some other order such that
$\v=\la H,\q \ra^2$.

Then 
there is a 
vertex relabeling  of $\q$ such that,
up to 
congruence,
$\q=\p$.
Moreover, under this vertex relabeling,
$G=H$.
\end{theorem}
By way of comparison, the labeled setting
is classical.  Since it is useful, we record it as a lemma.
\begin{lemma}[{\cite{YH38}}]\label{lem:mds}
Suppose that $\p$ and $\q$ are configurations
of $n$ points and that for all $N$ edges $ij$ of $K_n$,
we have $|\p_i - \p_j| = |\q_i - \q_j|$.  Then 
$\p$ and $\q$ are congruent.
\end{lemma}

In this section we will develop a new proof of 
Theorem~\ref{thm:bkMain}.
This will also allow us to develop machinery and general 
approaches that we will reuse later in the paper. 

\begin{definition}\label{def:linear-automorphism}
A \defn{linear automorphism} of a variety 
$V$ in $\CC^N$ is 
a non-singular linear transform on $\CC^N$
(that is, a non-singular $N\times N$ complex matrix $\A$) that bijectively
maps $V$ to itself.\footnote{In our setting, $V$ will always be a cone, so 
linear isomorphisms (as opposed to affine ones) are natural.}
\end{definition}

As promised in the introduction,
our main focus will be on understanding linear automorphisms
of $M_{d,n}$. We will first use our understanding of the singular
locus of $M_{d,n}$ to reduce to the case of $M_{1,d}$.
Next we will show that any edge permutation
that acts as an automorphism on $M_{1,d}$ must map 
triangles to triangles. From this, we will be able to  conclude that
the edge permutation must arise from a vertex relabeling.

The following will allow us to reduce the $d$-dimensional case
to the $1$-dimensional setting.

\begin{lemma}\label{lem:1d-Mdn}
Any linear  automorphism $\A$ of $M_{d,n}$ is 
a linear  automorphism of $M_{1,n}$.
\end{lemma}
This argument is inspired by the main technique 
used in  \cite{cip2}.
\begin{proof}
The singular set of $M_{d,n}$ is $M_{d,n-1}$ by Theorem 
\ref{thm:Mvariety}.  Thus, 
from Theorem~\ref{thm:Tmap}, $\A$ must be a linear 
automorphism of $M_{d-1,n}$.  We then see, by induction, 
that $\A$ is also a linear automorphism of $M_{1,n}$.
\end{proof}

Next we will look at projections of varieties onto lower 
dimensional coordinate spaces.

\begin{definition}\label{def:dependent-axes}
Let $V\subset \CC^N$ be an irreducible affine variety, where
$\{e_1, \ldots, e_N\}$ denotes  the coordinate basis for $\CC^N$.
Let $S\subset \CC^N$ be a linear subspace.  Let
$\pi_{S}$ denote the linear quotient map taking 
$\CC^N$ to $\CC^N/S$.

Let $I$ be a subset of $[N]$, the \defn{coordinate subspace}
$S_I$ is the linear span  $S_I = \lin\{e_i : i\in I\}$.
The map $\pi_{S_{\bar{I}}}(\cdot)$, where $\bar{I}$ is the complement of $I$ in $[N]$, 
is the quotient map that ignores the coordinates not in $I$.

A coordinate subspace $S_I$ is 
\defn{independent} in $V$ if the dimension of $\pi_{S_{\overline{I}}}(V)$
is $|I|$. Otherwise $S_I$ is \defn{dependent} in $V$.
\end{definition}

\begin{lemma}\label{lem:proj-and-aut}
Let $V\subset \CC^N$ be a variety
and $\A$ a linear automorphism of $V$ and
$S\subset \CC^N$ be a linear subspace.

Then the dimension of $\pi_{S}(V)$ is equal to that of 
$\pi_{\A(S)}(V)$.
\end{lemma}
\begin{proof}
Notice that $\pi_{\A(S)}(V)$ is linearly isomorphic to 
$\pi_{S}(\A^{-1}(V))$, which is  equal to $\pi_{S}(V)$ because
$\A$ is an automorphism of $V$.
\end{proof}

\begin{definition}\label{def:gen-perm}
An $N\times N$ matrix $\P$ is a
\defn{permutation} if each row and 
column has a single non-zero entry, and this entry
is $1$.
A matrix $\P'=\D\P$, 
where $\D$ is diagonal and invertible, is a 
\defn{generalized permutation} if each row and 
column has one non-zero entry.  
A generalized permutation has \defn{uniform scale} 
if it is a scalar multiple of a permutation matrix.  
\end{definition}

\begin{definition}
A generalized
permutation acting on an edge ensemble
is 
\defn{induced by a vertex relabeling}
when it has the same
non-zero pattern as an edge permutation that arises from a vertex relabeling.
\end{definition}

\begin{lemma}\label{lem:matroid}
Let $V\subset \CC^N$ be an irreducible variety,
and $\A$ a linear automorphism of $V$ that is 
a generalized permutation.  
Then $\A$ maps (in)dependent coordinate subspaces to 
(in)dependent ones.
\end{lemma}
\begin{proof}
A generalized permutation maps a 
coordinate subspace $S_I$ (as in Definition~\ref{def:dependent-axes})
to some other coordinate subspace $S_{I'}$, though not necessarily
with uniform scale. Likewise for 
$S_{\bar{I}}$ and $S_{\bar{I'}}$.
The conclusion is now an application of 
Lemma~\ref{lem:proj-and-aut} (using $S_{\bar{I}}$ as $S$).
\end{proof}

Now we will define the 
combinatorial notion of infinitesimally dependent and independent
sets of edges in $d$ dimensions, which will agree with the notion
of dependent and independent coordinates of $M_{d,n}$.

\begin{definition}
\label{def:maps}
Let $d$ be some fixed dimension and $n$ a number of vertices.
Let $E:= \{E_1,\ldots, E_k\}$ be  an edge measurement 
ensemble.
The ordering on the edges of $E$
fixes  an association between each edge in $E$ 
and a coordinate axis of $\CC^{k}$. 
Let $m_E(\p) :=
\la E, \p \ra^2$
be the map from $d$-dimensional 
configuration space to $\CC^{k}$
measuring the squared lengths of the edges
of $E$. 

We denote by $\pi_{\bar{E}}$
the linear map  from $\CC^N$
to $\CC^{k}$ 
that forgets the edges not in $E$, and is consistent with the 
ordering of $E$.
Specifically, we have an association between each edge of
$K_n$ and an index in $\{1,\dots,N\}$, and thus we can think of each 
$E_i$ as simply its index in $\{1,\dots,N\}$. Then,
$\pi_{\bar{E}}$ is defined by the conditions:
$\pi_{\bar{E}}(e_j) = 0$ when $j\in \bar{E}$
and
$\pi_{\bar{E}}(e_j) = e'_i$ when $E_i=j$,
where 
$\{e_1, \ldots, e_N\}$ denotes  the coordinate basis for $\CC^N$
and
$\{e'_1, \ldots, e'_k\}$ denotes  the coordinate basis for $\CC^k$.
We call $\pi_{\bar{E}}$ an \defn{edge forgetting map}.

The map $m_E(\cdot)$ 
is simply the composition of the complex measurement map $m(\cdot)$ 
and $\pi_{\bar{E}}$.

Finally, we denote by $M_{d,E}$ the Zariski closure of the image
of $m_E(\cdot)$ over all $d$-dimensional configurations. 
\end{definition}

\begin{definition}
\label{def:ind}
We say the 
an edge measurement 
ensemble $E$  is \defn{infinitesimally independent} in $d$
dimensions if, starting from a generic complex configuration 
$\p$ in $\CC^d$, we can
\emph{differentially}
vary each of the $|E|$ 
squared lengths independently by
appropriately differentially varying our  configuration $\p$. 
This exactly
coincides with the notion of infinitesimal independence from graph 
rigidity theory~\cite{L70}.

Infinitesimal independence means 
that the image of the differential of $m_E(\cdot)$ at $\p$ is $|E|$-dimensional.
Note that this rank can only drop on some (non-generic)
subvariety of configuration space.

An edge measurement ensemble that is 
not infinitesimally independent in $d$ dimensions is 
called \defn{infinitesimally dependent} in $d$ dimensions.
Note that in this case the rank can never rise to 
$|E|$.
\end{definition}

The following is implicit in the rigidity theory literature. 
\begin{proposition}
\label{prop:infind}
An edge measurement ensemble
$E$ is infinitesimally  independent in $d$ dimensions
iff the image of $m_E(\cdot)$ over all complex 
configurations of $n$ points has dimension $|E|$.  
\end{proposition}
\begin{proof}[Proof sketch]
From 
the constant rank theorem (as used in ~\cite[Proposition 2]{asimow}),
the dimension of the image of $m_E(\cdot)$ 
least as big as the rank of the  differential at a generic $\p$.
Sard's Theorem~\cite[Theorem 14.4]{harris} tells us that inverse image of 
almost every point 
in the image consists entirely
of configurations $\p$, where the differential
has rank at least as big as the
the dimension of the image of $m_E(\cdot)$.
\end{proof}

\begin{remark}\label{rem:complex-rigid}
These notions are usually studied in the real setting, but the
tools used in the proof sketch above (constant rank and algebraic Sard)
work the same way in the complexified setting.  Complexification 
is used to study rigidity problems in, e.g., \cite{ST10,cgr,OP07,ciprian}.
\end{remark}

The following is a standard result from rigidity theory
(see, e.g., \cite[Corollary 2.6.2]{GSS}).  We give a 
proof for completeness.
\begin{proposition}
\label{prop:simplex}
Let $E$ be an edge measurement ensemble.
Suppose $|E| \le \binom{d+2}{2}$ and $E$ is 
infinitesimally dependent
in $d$ dimensions.
Then $|E| = \binom{d+2}{2}$ and
$E$ consists of the edges of a  $K_{d+2}$ subgraph (in some order).
\end{proposition}
\begin{proof}
Assume, w.l.o.g., that $E$ is infinitesimally
dependent and inclusion-wise minimal
with this property.  Let $\p$ be generic.
If $E$ does not consist of the edges of a $K_{d+2}$ subgraph, then it has a vertex $v$ 
of degree at most $d$.  Since $\p_v$ has $d$ infinitesimal degrees of 
freedom, the $\le d$ squared lengths of each edge in 
edge set $E'$ incident on $v$ can be 
differentially varied independently.  

In particular, the coordinate directions in $\CC^{|E|}$ associated with
$E'$ are in the image $X$ of $dm_E(\cdot)$ at $\p$.  
Because $X$ is not all of $\CC^{|E|}$, linear duality provides 
a non-zero $\omega\in \left(\CC^{|E|}\right)^*$ vanishing on $X$.
(called an equilibrium stress  in the rigidity literature).
By the above observation $\omega(e_j) = 0$ for each $j\in E'$.

This means that $\omega$, also acts on $\left(\CC^{|E|}/S_{E'}\right)^*$, 
and vanishes on 
$X/S_{E'}$.
By linear duality once more, this means that $E\setminus E'$ 
is infinitesimally dependent, contradicting the minimality of $E$.
\end{proof}

We will use this result many times in the paper, but in the
present section
all we will use is the
$1$-dimensional case where the infinitesimally dependent triples must correspond to 
triangles.

Next we look at edge permutations that preserve 
cycles, and specifically triangles:

\begin{definition}\label{def:cycle-isomorphism}
Let $\Gamma$ and $\Delta$ be simple graphs and $\P$ a bijection 
between the edge sets.  The map $\P$ is a \defn{cycle 
isomorphism} if the image of any cycle (and only cycles)
in $G$ is a
cycle in $H$.
A \defn{triangle isomorphism} is 
like a cycle isomorphism, but relaxed to only 
require that exactly triangles are mapped to triangles.
Any triangle isomorphism of $K_n$ is necessarily a cycle
isomorphism by considering the dimension 
of the cycle space (see, e.g., \cite{O11} for definitions).
\end{definition}
\begin{lemma}\label{lem:vweak-whitney}
For any $n\ge 3$, any cycle automorphism of $K_n$ 
is induced by a vertex relabeling.
\end{lemma}
This is a consequence of Whitney's famous ``$2$-isomorphism theorem''
\cite{W33}, but the special case here is easy to establish.
\begin{proof}
The statement is clear for the triangle $K_3$ and $K_4$ is direct
computation.  
For $n\ge 4$ consider the wheel graph $W_{n-1}$.  Any cycle 
automorphism $\P$ must fix the ``rim'' cycle, so the hub vertex
must be a fixed point. Thus, the action $\P$ permutes the spoke 
edges, and, consequently, the rim vertices.  The general case
follows, since every vertex star in $K_n$ has a $W_{n-1}$
subgraph,
so any cycle automorphism of $K_n$ permutes vertex stars and 
thus the vertices.
\end{proof}

Since we wish to deal with generalized permutations, we will
also need to reason about scaling of individual edges.

\begin{lemma}\label{lem:voldet-scale}
Let $m_{12}$, $m_{13}$ and $m_{23}$ be the squared
edge lengths of a $1$-dimensional triangle, and
suppose that $s_{12}$, $s_{13}$ and $s_{23}$
are scalars such that the simplicial 
volume determinant 
\[
\begin{split}
\det
\begin{pmatrix}
2m_{13}& (m_{13} + m_{23} - m_{12})\\
(m_{13} + m_{23} - m_{12})& 2m_{23}\\
\end{pmatrix} = & \\  & 2 (m_{12}m_{13} + m_{12}m_{23} + m_{13}m_{23})
- (m^2_{12} + m^2_{13} + m^2_{23})
\end{split}
\]
(see Section~\ref{sec:varieties}) is mapped to a 
multiple of itself under the scaling $m_{ij}\mapsto s_{ij}m_{ij}$.  Then the $s_{ij}$ are all equal.
\end{lemma}
\begin{proof}
The hypothesis means that the desired statement holds for any specialization 
of the $m_{ij}$.  Consider the case where $m_{23} = 0$.
The presence of the monomials $m_{12}^2$ and 
$m_{12}m_{13}$ then imply that $s_{12}^2 = s_{12}s_{13}$, that is, $s_{12} = s_{13}$.  Continuing the same way, 
we see that $s_{12} = s_{23}$. 
\end{proof}

We can now present a proof of
the following slight generalization  of
\cite[Lemmas 2.3 and 2.4]{BK1} (that deals 
with generalized permutations instead of 
permutations).  Minor modifications of 
the arguments in \cite{BK1}, which are different
from ours, can yield the same result. 
We will use this generalization in 
Section~\ref{sec:bign} below. This result forms the core of 
Boutin and Kemper's result.
\begin{theorem}[{\cite[Lemmas 2.3 and 2.4]{BK1}}]\label{thm:bk-linear}
Suppose that $\A$ is a generalized permutation
that is a linear automorphism of $M_{d,n}$.  Then 
$\A$ is induced by a vertex relabeling and has 
uniform scale.
\end{theorem}
\begin{proof}
First we check that $\A$ is induced by a vertex relabeling.
We can  apply Lemma~\ref{lem:1d-Mdn} 
to reduce to the $1$-dimensional case.  Specialized
to $M_{1,n}$, Lemma~\ref{lem:matroid} and Proposition~\ref{prop:infind}
imply that infinitesimally
dependent edge sets in the rigidity theoretic sense 
of Definition~\ref{def:ind} are mapped to each other.  
Proposition~\ref{prop:simplex} then implies that
$\A$ is a triangle isomorphism composed with an edge scaling.
Lemma~\ref{lem:vweak-whitney} then tells us that 
$\A$ is induced by a vertex relabeling.

Next we need to prove uniform scale.  
Let $\pi_{\bar{K}}$ be an edge forgetting map 
that ignores all of the edges in the complement
of an ensemble $K$, consisting of the edges of a fixed
triangle. Under any ordering of the edges of $K$, we have 
$\pi_{\bar{K}} (M_{1,n}) = M_{1,3}$.
(which is cut out from $\CC^3$ by the simplicial 
volume determinant as in Lemma~\ref{lem:voldet-scale}).

We know that we can factor $\A$ into 
$\D\P$, where $\D$ is diagonal and 
$\P$ is a 
permutation induced by a vertex relabeling.
Since a vertex relabeling is a linear automorphism of $M_{1,n}$, then so too is $\D$.

Since $\D$ is diagonal, and $\pi_{\bar{K}}$ is an
edge forgetting map, then $\pi_{\bar{K}} \D = \D' \pi_{\bar{K}}$ 
for an appropriate $3\times 3$ diagonal scaling matrix $\D'$, 
 making 
$\D'$ an automorphism of $M_{1,3}$. So it has to 
send the simplicial volume determinant to a multiple of itself.
This is the situation of Lemma~\ref{lem:voldet-scale},
and we conclude that the scaling on each triangle is uniform.

That $\A$ has a uniform scale then follows from applying 
the above argument repeatedly to overlapping triangles
until we have determined the scale on every edge.
\end{proof}

The proof of this section's main theorem now follows directly.

\begin{proof}[Proof of Theorem~\ref{thm:bkMain}]
By assumption there exists a
edge permutation matrix
$\P$ such that  $\v = \P(m(\p))$ and 
some
edge permutation matrix $\Q$ such that
$\v = \Q(m(\q))$, and thus
$\Q^{-1}\v$ is in $M_{d,n}$.
Thus letting $\A:= \Q^{-1}\P$, 
also a edge permutation, we have 
$\A(m(\p)) \in M_{d,n}$.

i) 
From Theorem~\ref{thm:prin2},
we know that $\A(M_{d,n}) = M_{d,n}$, i.e, $\A$ is a linear automorphism
of $M_{d,n}$.

ii)
From 
Theorem~\ref{thm:bk-linear},
we can conclude that $\A$ arises from a 
vertex relabeling under which 
$G = H$.

iii)
After relabeling, we are now in the situation of Lemma~\ref{lem:mds},
and so we conclude that $\q$ is congruent to $\p$.
\end{proof}

\begin{remark}\label{rem:matcomp}
A related problem to the the unlabeled rigidity problem for $K_n$ is to 
recover a generic rank $d$ symmetric matrix from its unlabeled set of entries.  This 
seems easier to reason about.  Let $\CS^{n}_d$ be the variety of symmetric, $n\times n$ matrices
of rank at most $d$.  
The 
linear automorphisms of $\CS^{n}_d$ are of the ``factored'' form $\trans{\B}\G\B$, where $\G\in \CS^{n}_d$ 
and $\B$ is any $n\times n$ non-singular matrix (see, e.g., \cite{cip2}).
This factored form implies that if a linear automorphism 
of  $\CS^{n}_d$
is an edge permutation, it is also a vertex relabeling.  Thus the fixed rank ``matrix completion''
(see, e.g., \cite{CS,KTT}) version of 
Theorem~\ref{thm:bkMain} (and an appropriate generalization
to the non-symmetric case) follows immediately 
through this.

Unfortunately, this line of thinking does not seem to lead directly 
to Theorem~\ref{thm:bkMain}.  The issue is that the linear 
isomorphism $\varphi$ between $M_{d,n}$ and $\CS^{n-1}_d$, described in Theorem~\ref{thm:Mvariety},
does not  imply 
that linear automorphisms of $\CS^{n-1}_d$ 
have a factored form when acting on $M_{d,n}$, where
 points in $\CC^N$ are expressed (in so-called 
distance matrix form)
as symmetric $n\times n$
matrices with zero diagonal entries.
In fact,
there are linear automorphisms of $M_{1,3}$
(necessarily not generalized permutations)
which do not have such a factored form.
\end{remark}

\subsection{Trilateration}

Now we wish to extend this result to the case where our edge 
measurement ensemble 
is not complete, but does allow for trilateration in the sense of
Definition~\ref{def:ensemble}. For this case 
we will need to restrict
this discussion to $d\ge 2$, as we have already 
encountered one-dimensional counterexamples  in 
Section~\ref{sec:defs}.
The key idea is to use Theorem~\ref{thm:bkMain}, but only, 
and iteratively, applied to $K_{d+2}$
subsets of $K_n$. This will lead to the next
theorem. \begin{theorem}
\label{thm:triBK2}
Let $d \ge 2$. 
Let $\p$ be a generic configuration of $n\ge d+2$ points. Let
$\v= \la G, \p \ra^2$ where 
$G$
is an edge measurement ensemble
that allows for 
trilateration. 

Suppose there is a 
configuration $\q$, 
also of $n$ points,
along with 
an edge
measurement ensemble $H$,
such that 
$\v=\la H,\q \ra^2$.

Then 
there is a vertex relabeling  of $\q$ such that,
up to congruence,
$\q=\p$.
Moreover, under this vertex relabeling,
$G=H$.
\end{theorem}

\begin{remark}
As is often the case in rigidity theory (see, e.g., \cite{ght}),
for both Theorems~\ref{thm:bkMain} 
and~\ref{thm:triBK2}, we can
weaken the restriction that $\p$ be generic. All of the undesired 
exceptions arise due to a finite number of algebraic conditions, and thus, as per Remark~\ref{rem:gtoz},
there exists some Zariski open subset $O$ of the configuration space such that
these results hold whenever $\p$ is in $O$.
\end{remark}

\paragraph{Theorem~\ref{thm:triBK2}, sketch}
Informally, Theorem~\ref{thm:triBK2} says that, generically, 
we do not need to know the edge labels for the 
trilateration reconstruction process described 
in Section~\ref{sec:introa} to succeed.

The intuitive reasoning is as follows.  The trilateration 
process starts from a known $K_{d+1}$ and then locates
each additional point by ``gluing'' a new $K_{d+2}$ (with one
new point) onto a $K_{d+1}$ inside the already visited $K_{v}$
over the $v$ previously reconstructed vertices.
The idea, then, is to find the labels as we locate points
by using Theorem~\ref{thm:bkMain} iteratively: initially 
to find a ``base'' $K_{d+2}$ 
(the bigger base is necessary 
because a $K_{d+1}$ is too small to find using 
Theorem~\ref{thm:bkMain})
to start the trilateration 
process, and then, after measuring all the edges between the 
visited points, to find subsequent $K_{d+2}$ 
that add one more
point.  When $d\ge 2$, there is only one way 
to do the gluing, because generic $(d+1)$-simplices
do not have any ``self-congruences''.

Even though the steps above are conceptually very simple,
the details are a bit involved.  This is partially 
 due to notational 
overhead of measurement ensembles, and partially because
Theorem~\ref{thm:bkMain} is about whole configurations, and 
we will need to work with subconfigurations at every step.
Because the edge measurement ensemble result has a proof that
is very similar in structure to that of our main Theorem~\ref{thm:punchline}, 
we go through this warm-up result 
carefully below.

\paragraph{Theorem~\ref{thm:triBK2}, details}
We now fill in the sketch above.  A key 
technical result we need is that, generically,  $D$ edge 
measurements look like they come from a 
$K_{d+2}$ subensemble exactly when they really do.

\begin{proposition}
\label{prop:ind}
Let $d \ge 1$.
Let  $\w:=(w_1, \dots, w_D)$
describe a point in $\CC^D$ that
happens to be a point of $M_{d,d+2}$,
with no two of the $w_i$ identical.

Suppose that there is a generic configuration $\p$ 
in $\RR^d$ (or simply such that $m(\p)$ is generic in $M_{d,n}$),
and an edge measurement ensemble $E$, of size $D$,
such that $w_i=\la E_i,\p \ra^2$.

Then
there must be a  subconfiguration  
$\p_T$ of $\p$ with $d+2$ points 
such that
$m_E(\p)=\w = m(\p_T)$.
\end{proposition}
\begin{proof}
Since each of the $w_i$ is unique, the $D$ edges of
$E_i$ must be distinct.
Our measurement sequence $\w$ arises from the squared
lengths of $D$ edges in $m(\p)$
thus 
$\w = 
\pi_{\bar{E}}(m(\p))$

i) From Theorem~\ref{thm:prin1}, we see that 
$\pi_{\bar{E}}(M_{d,n}) \subset M_{d,{d+2}}$. 

ii) 
From Proposition~\ref{prop:simplex},
an edge measurement ensemble with $D$ edges is infinitesimally \emph{independent}
in $d$ dimensions unless they consist of the edges from
of a $K_{d+2}$ simplex. 
Thus if $E$ does not consist of the edges of a $K_{d+2}$, then 
$\pi_{\bar{E}}(M_{d,n})$ would be $D$-dimensional (Proposition~\ref{prop:infind})
and could not lie
in  $M_{d,d+2}$. 

Thus
there must be a generic  subconfiguration
$\p_T$ of
$\p$ with $d+2$ points
such that
$m_E(\p)=
\w = \P(m(\p_T))$
where $\P$ is some edge permutation
on the $D$ edges of a $K_{d+2}$.
Meanwhile we have $\w \in M_{d,d+2}$, thus there must be a configuration
of $d+2$ points $\q$ such that $\w=m(\q)$.

ii) Then from Theorem~\ref{thm:bkMain}, 
after a vertex relabeling of $\p_T$, we must have
$\p_T$ congruent to $\q$ and
$\w = m(\p_T)$.
\end{proof}

The following lemmas  will tell us that 
if we have built up a part of $\p$ using 
trilateration from unlabeled edge measurements, we 
can extend what we know to one more point.

\begin{lemma}
\label{lem:noSim}
If $\q$ is a generic configuration in any fixed dimension, then 
no two subconfigurations of
at least three points in $\q$ can be similar to each other,
unless the two subconfigurations consist of the same points,
in the same order.
\end{lemma}
\begin{proof}
Consider the ``ratio spectrum'' of a configuration $\r$, which is the 
sequence of pairwise edge-length ratios  
(ordered, in some 
way based on the ordering of the points in $\r$). 
Agreement of ratio spectra of a pair of configurations, each
of three or more points
can be expressed as a non-trivial
polynomial condition on the point coordinates, defined over $\QQ$.  
Meanwhile, similar configurations have the same ratio spectra, and so
must
satisfy this polynomial condition,
certifying non-genericity.
\end{proof}

\begin{lemma}
\label{lem:extend}
Let $d \ge 2$.
Let $\p$ and $\q$ be generic configurations 
of $d+2$ points in $\RR^d$
that are related by  
a similarity $\sigma$.
Moreover, suppose that both 
$\p$ and $\q$ share an unordered subset 
$S$ of 
$d+1$ points.
 Then
$\p=\q$.
\end{lemma}
Only the case of congruence is needed now, but similarities will be needed later in Section~\ref{sec:ugr}. 
\begin{proof}
Using the shared point set $S$,
we can find a subconfiguration $\r$ of $\p$ that is similar (in fact identical) to one subconfiguration of $\q$. From 
Lemma~\ref{lem:noSim}
we conclude that there is no
other subconfiguration of $\q$ similar to $\r$.

Thus the assumed  
similarity $\sigma$ between $\p$ and $\q$,  must leave the points of
$\r$ fixed. Since $\r$ has $d+1$ points, $\sigma$ is the identity,
and we are done.
\end{proof}
\begin{remark}
The statement of Lemma~\ref{lem:extend} is not true for $d=1$, even if we restrict 
$\sigma$ to be a congruence, because
a subconfiguration $\{\p_i, \p_j\}$ is congruent to
the subconfiguration $\{\p_j, \p_i\}$.
\end{remark}

The next lemma describes our main inductive step.

\begin{lemma}\label{lem:partial-trilat}
Let $d\ge 2$ and let $\p$ and $\q$ be configurations
so that $\p$ is generic and $m(\q)$ is generic in 
$M_{d,n'}$ (for some $n'$).  
Let 
$G$ and $H$ be two edge measurement ensembles such that
$\la G, \p\ra^2 = \la H, \q\ra^2$.

Suppose
that we have two 
``already visited'' subconfigurations
$\p_V$ and $\q_{V'}$ with 
$\p_V =\q_{V'}$.

Suppose  we can 
find $F$, a set of  
$d+1$ distinct edges in 
$G$ 
connecting some unvisited
vertex $\p_i\in \p_{\bar{V}}$ to 
some visited 
subconfiguration $\p_R$ of $\p_V$ with
$d+1$ vertices.

Then we can find an unvisited
$\q_{i'}\in \q_{\bar{V'}}$ such that 
the two subconfigurations $\p_{V\cup\{i\}}$
and $\q_{V'\cup\{i'\}}$ are equal.
\end{lemma}
\begin{proof}
Let $\p_T$ be a subconfiguration consisting of, in some order, all the
points of $\p_R$ along with $\p_i$.
Let $\w:=m(\p_T)$, each $w_i$ distinct due to genericity.

Using the existence of $F$, the fact that
$\la G, \p\ra^2 = \la H, \q\ra^2$ 
together with
$\p_V =\q_{V'}$, we can 
find an edge measurement ensemble $E$ under which 
we can apply Proposition~\ref{prop:ind} to 
$\q$ using this same 
$\w$.

This guarantees a 
$d+2$ point subconfiguration 
$\q_{T'}$ of $\q$ such that
$\w=m(\q_{T'})$.
From  Lemma~\ref{lem:mds},
we  conclude that
$\p_T$ and $\q_{T'}$ 
are related by a 
congruence.  

By construction $\p_T$ contains the subconfiguration
$\p_R$, which is also a subconfiguration of $\q_{V'}$.
From genericity of $\q$  and Lemma~\ref{lem:noSim}, 
$\p_R$ is congruent to no other
subconfiguration of $\q$. Thus $\p_R$ must be a subconfiguration
of $\q_{T'}$.
Similarly, from genericity of $\p$ and Lemma~\ref{lem:noSim}, 
the remaining vertex $\q_i'$ of $\q_{T'}$ not included in $\p_R$
must be unvisited, 
ie. in $\q_{\bar{V'}}$.

Then from Lemma~\ref{lem:extend}, we must have 
$\p_T=\q_{T'}$ and thus 
$\p_{V\cup\{i\}}=\q_{V'\cup\{i'\}}$.
\end{proof}

Applying the above iteratively yields the following:

\begin{lemma}
\label{lem:triBK}
Let the dimension $d \ge 2$. 
Let $\p$ be a generic configuration of $n\ge d+2$ points. Let
$\v= \la G, \p \ra^2$, where 
$G$
is an edge measurement ensemble
that allows for 
trilateration. 

Suppose that there is a
configuration $\q$, 
of $n'$ points, 
that is 
generic (or simply such that $m(\q)$ is generic in $M_{d,n'}$),
along with 
an edge
measurement ensemble $H$ 
such that 
$\v=\la H,\q \ra^2$.

Let $\q_{V'}$ be the subconfiguration of $\q$ indexed by the vertices within the support of  $\pmb{\beta}$.
Then
there is a vertex relabeling  of $\q_{V'}$ such that,
up to congruence,
$\q_{V'}=\p$.
Moreover, under this vertex relabeling,
$G=H$.
\end{lemma}
\begin{proof}
For the base case, 
the trilateration assumed in $G$
guarantees a $K_{d+2}$ contained in $G$, over
a ${d+2}$ point subconfiguration  $\p_T$ of  $\p$.
Define $\w:=m(\p_T)$, with each $w_i$ distinct due to genericity.
We have 
$\w \in M_{d,d+2}$.

Using the
fact that
$\la G, \p\ra^2 = \la H, \q\ra^2$ 
we can apply Proposition~\ref{prop:ind} to this $\w$,  $\q$
and appropriate subensemble $E$ of $H$.
From this, we conclude 
that
there is a ${d+2}$ point
subconfiguration  $\q_{T'}$ of $\q$ such that
$m_{E'}(\q)=\w=m(\q_{T'})$.
From Lemma~\ref{lem:mds},
up to congruence, we have
$\p_T = \q_{T'}$.

Then, going forward inductively, assume that we have a 
two ``visited'' subconfigurations such that 
$\p_{V}$ and
$\q_{V'}$, 
are related by a 
global congruence. 

Continuing with 
the trilateration process allowed by  
$G$, we can iteratively apply (with the global
congruence factored out)
Lemma~\ref{lem:partial-trilat} until we have visited all of  
$\p$. 
At this point we will have, that  
up to congruence, 
$\q_{V'}=\p$.

Since $\p$ is generic, then no two distinct edges  can have the same squared length. 
The same is true for $\q$.
This gives us, after vertex relabeling,  equality between all of 
$G$, 
and $H$. 
\end{proof}

With some other added assumptions, we can remove the genericity
assumption from $\q$. To see this, we first use the following definition from~\cite{asimow}.

\begin{definition}
Let $d$ be a fixed dimension.
Let $E$ be an edge measurement ensemble
 with $n\ge d+1$.
We say $E$ is 
\defn{infinitesimally rigid} in $d$ dimensions,
if, starting at some  generic 
(real or complex) configuration $\p$,
there are no differential  motions of $\p$ in $d$ dimensions that preserve 
all of the squared lengths among the edges of $E$,
except for differential congruences.

When an edge measurement ensemble is infinitesimally  rigid, then the lack
of differential motions holds over a Zarksi-open subset of
configurations, that includes all generic configurations.

Letting $m_E(\p)$ be the map from
configuration space to $\CC^{|E|}$
measuring the squared lengths of the edges
of $E$, 
infinitesimal rigidity means 
that the image of the differential of $m_E(\cdot)$ at $\p$ is $(dn-C)$-dimensional.
Note that this rank can only drop on some (non-generic)
subvariety of configuration space.

\end{definition}
The following proposition follows exactly as Proposition~\ref{prop:infind}.
\begin{proposition}
\label{prop:infrig}
If $E$ is  infinitesimally rigid, then the image of $m_E(\cdot)$ acting on all configurations
is $dn-C$-dimensional. Otherwise, the dimension of the image is smaller.
\end{proposition}

\begin{lemma}
\label{lem:triBK2}
In dimension $d \ge 1$, let  $\p$ and $\q$ be two 
configurations with the same number of points $n \ge d+1$.
Suppose that 
$G$ and
$H$
are two edge measurement ensembles, each with 
the same number $k$, of edges,
and with $G$
infinitesimally  rigid in $d$ dimensions.
And suppose that  
$\v:=\la G,\p \ra^2 =
\la H,\q \ra^2$.

If $\p$ is a generic configuration, then 
$m(\q)$ is generic in $M_{d,n}$.
\end{lemma}
\begin{proof}
Recall the notation introduced in Definition~\ref{def:maps}.
The varieties  $M_{d,G}$ and  $M_{d,H}$, both subsets of 
$\CC^k$,
are 
defined over $\QQ$. 
They are irreducible
since they arise as images of 
$M_{d,n}$,
which is irreducible, under a  
polynomial (in fact linear) map.
$M_{d,G}$
is of 
dimension $dn-C$ from Proposition~\ref{prop:infrig}.
Likewise, $M_{d,H}$ is of dimension
$dn-C$ if 
$H$ is infinitesimally rigid, otherwise it is 
of smaller dimension.

Our assumptions give us $\v \in M_{d,G}$
and
$\v \in M_{d,H}$.

We claim $M_{d,G} = M_{d,H}$.
Suppose not, then 
$M_{d,G} \cap M_{d,H}$ is an algebraic
variety, defined over $\QQ$, of dimension strictly less 
than $dn-C$ (due to irreducibility),  
and thus could contain no generic points of $M_{d,G}$.
But we have assumed that $\v$ is in both, and thus also in
this intersection set. But since $\p$ is generic,
then $\v$ is  generic in $M_{d,G}$ (Lemma~\ref{lem:genPush}).
This contradiction thus establishes our claim.

Since
$M_{d,G} = M_{d,H}$, then $\v$ is also a generic point of 
of $M_{d,H}$.

Finally, since $M_{d,H}$ is the image of $M_{d,n}$ under 
the linear map $\pi_{\bar{H}}(\cdot)$,
and since they 
have the same dimension, 
then from Lemma~\ref{lem:genPull} 
the preimage of $\v$
under $\pi_{\bar{H}}(\cdot)$, which is $m(\q)$,
must be a generic point in 
$M_{d,n}$. 
\end{proof}

And we can now finish the proof of our theorem:
\begin{proof}[Proof of Theorem~\ref{thm:triBK2}]
An edge measurement ensemble that allows for 
trilateration is always infinitesimally
rigid.  
The result then  follows directly from 
Lemmas~\ref{lem:triBK2}
and~\ref{lem:triBK}.
\end{proof}

\subsection{Digression: Unlabeled generic global rigidity}

The previous section leads to a very natural 
open question that we briefly discuss here.
Are there edge measurement ensembles
that do not allow for trilateration, but such that
a generic configuration $\p$ can still
be reconstructed from their unlabeled squared edge lengths? 

For this discussion, we will use the following definitions from~\cite{conGR, ght}.

\begin{definition}\label{def:global-rigidity}
Let 
$G$
be an edge measurement ensemble
 with $n\ge d+1$.
We say that $G$ 
is \defn{generically globally rigid} in $d$ dimensions if,
 starting with some  generic complex configuration $\p$,
there are no other configurations $\q$ in $d$ dimensions
with the same \underline{labeled} squared edge lengths
except for congruences.

When an edge measurement ensemble is generically globally rigid, then this 
uniqueness property holds over a Zariski-open
subset of configurations that includes all generic 
configurations.

\end{definition}
Typically these definitions are done in the real setting, but
there is no change when moving to the complex setting by results from
\cite{cgr}.

\begin{remark}
Gortler, Healy and Thurston \cite{ght} showed that if an edge measurement ensemble is not 
generically globally rigid in $d$ dimensions, then starting
with any generic configuration $\p$ there will always be other
non-congruent  configurations $\q$ in $d$ dimensions
with the same labeled squared edge lengths.
\end{remark}

Edge measurement ensembles
that allow for $d$-dimensional trilateration are certainly generically globally rigid
in $d$ dimensions, but there are plenty of edge 
measurement ensembles  that are generically
globally rigid but do not allow for trilateration. One simple such example in two dimensions
is when $G$ comprises the edges of the complete bipartite graph $K_{4,3}$ 
(generic global rigidity follows from the combinatorial considerations of~\cite{conGR,jj}
and can be directly confirmed
using the algorithm from~\cite{conGR,ght}). 
This graph 
does not even contain a single triangle!

If an edge measurement ensemble 
$G$ is not generically globally rigid, then we generally
cannot recover $\p$ when given 
\emph{both} $\v$ and $G$ (that is, labeled data).
The recovery problem is 
simply not well-posed. When an edge 
measurement ensemble is generically globally rigid, then  
generally this labeled 
recovery problem will be well-posed, though it still might be
intractable to perform~\cite{saxe}. 
We note that testing whether an edge measurement ensemble is generically
globally rigid can be done with an efficient randomized algorithm~\cite{ght}.

With this in mind, we pose the following question:

\begin{question}
In any dimension $d \ge 2$, 
let $\p$ be a generic configuration of $n$ points. Let
$\v= \la G, \p \ra^2$, where 
$G$
is an edge measurement ensemble
that is generically 
globally rigid in $d$ dimensions.

Suppose there is a 
configuration $\q$, 
also of $n$ points,
along with 
an edge
measurement ensemble $H$ 
such that 
$\v=\la H,\q \ra^2$.

Does this imply the following conclusion:
There is a vertex relabeling  of $\q$ such that,
up to congruence,
$\q=\p$.
Moreover, under this vertex relabeling,
$G = H$.
\end{question}

\subsection{Road map}
Let us now return to our situation of path or loop ensembles in 
$d \ge 2$ dimensions.
In this case, the data arises as \emph{sums} of edge lengths. 
If these were sums of \emph{squared} lengths, 
the problem would be far trickier, because the full group 
of linear automorphisms of $M_{d,n}$, as discussed in Remark 
\ref{rem:matcomp}, is isomorphic to ${\rm GL}(n)$.  This 
makes our ``adversary'' (introduced in Section~\ref{sec:introb})
far more powerful, once it is 
no longer constrained to use only edge permutations.


Luckily, our data arises as sums of \emph{unsquared} 
edge lengths, and thus
our problem will instead be governed by the structure
of linear maps acting on $L_{d,n}$ instead of $M_{d,n}$.
Linear automorphisms of $L_{d,n}$ will turn out to 
be much more constrained, making our problem more tractable.

Our basic strategy will still be to rely on trilateration,
so we need the appropriate generalization of
Proposition~\ref{prop:ind}. (This will
be Theorem~\ref{thm:consist} below.)
Thus we will look in our data set at sub data sets of 
size $D$.
Any such $D$-tuple of measurements can be represented by 
a $D\times \edgecard$ matrix $\E$.

Suppose $\E$ represents a very simple measurement
ensemble, where each row gives us, say, the edge
length of one edge of some $K_{d+2}$ subgraph of
$K_n$ in an appropriate order. 
Then $\E(l(\p))$ will lie in $L_{d,d+2}$.
Conversely, 
as in Proposition~\ref{prop:ind},
we do not expect that $\E(l(\p))$ will lie in 
$L_{d,d+2}$, \emph{unless} 
$\E$ has the property that all of $\E(\Ln)$ lies
in $L_{d,d+2}$. Thus, our main task will be understanding which 
$\E$ have this property. 
We will show in Section~\ref{sec:maps} that, 
essentially, the only such $\E$ are maps that
ignore all of the edges of $K_n$ except for
those of a single $K_{d+2}$. 
(There will also be the possibility that
the matrix $\E$ has rank less than $D$, which 
we will need to understand as well in our reconstruction 
algorithm).

Then we will be left with understanding what are the
$D\times D$ matrices $\A$ that are linear automorphisms of 
$L_{d,d+2}$. 
This is done in Sections~\ref{sec:min3} and~\ref{sec:auto}.

When all this dust settles 
and we have established 
Theorem~\ref{thm:consist},
we will essentially know that
if $\p$ is generic and $D$ measurements ``look consistent'' 
with the $D$ edges of some $K_{d+2}$, then they do in fact 
arise from simply measuring the lengths of such edges.

From this we will be able to apply trilateration, 
using the same reasoning  as in Lemma~\ref{lem:triBK},
to obtain
our Lemma~\ref{lem:punchline1}
which covers the full $\p$. Finally, we will apply the same reasoning
from Theorem~\ref{thm:triBK2} to our measurement setting to complete the
proof of Theorem~\ref{thm:punchline}.

\section{Linear maps from $L_{d,n}$ to $\CC^D$}
\label{sec:maps}

Let $d \ge 1$.
Let $D:= \binom{d+2}{2}$.
In this section,
$\E$ will be 
a $D\times \edgecard$ matrix
representing a rank $r$ linear map from $\LN$ to
$\CC^D$, where 
$r$ is some number $\le D$ . 
Our goal is to study linear maps where the 
dimension of the image is strictly less than $r$. 
In particular this will occur when $\E(\LN) = L_{d,d+2}$.

\begin{definition}
We say that $\E$ has 
$K_{d+2}$ support 
if it depends only on measurements supported over the $D$ 
edges corresponding to a $K_{d+2}$ subgraph of $K_n$. 
Specifically,  
all the columns of the matrix 
$\E$ are zero, except for at most $D$ of
them, and these non-zero columns index edges 
contained within a single $K_{d+2}$.
\end{definition}

The main result of this section is:
\begin{theorem}
\label{thm:linImage}
Let $\E$ be  a 
$D\times \edgecard$ matrix with rank $r$.
Suppose that the image $\E(L_{d,n})$,
a constructible set,  
is not of dimension $r$.
Then  $r=D$ and $\E$ has $K_{d+2}$ support.
\end{theorem}
\begin{remark}
Theorem~\ref{thm:linImage} does not hold when 
$\LN$ is replaced by $M_{d,n}$. As described in Remark~\ref{rem:matcomp}, the linear automorphism group of
$\CS^{n-1}_d$ is quite large,  and thus provides 
automorphisms $\A$ of $M_{d,n}$ that have
dense support. Thus, even if some $\E$ has $K_{d+2}$ support
the composite map $\E\A$ would not, and it 
could still have a small-dimensional image. 
\end{remark}

The proof relies (crucially) on the more technical, linear-algebraic 
Proposition~\ref{prop:inf}, proved below.  The idea
leading to it is as follows.

If a
point $\bl$ is smooth in $L_{d,n}$ then so is 
any $\bl'$ obtained by negating various coordinates of $\bl$.
Thus, the collection of complex analytic 
tangent spaces to $L_{d,n}$,
$T_{\bl}L_{d,n}$,
at $\bl$ and its orbit under coordinate negations gives us a
an arrangement $\mathcal{T}$ 
of $2^N$ linear spaces.  Any $\E$ meeting the 
hypothesis of Theorem~\ref{thm:linImage} necessarily drops 
rank on every subspace in $\mathcal{T}$.  For reasons of dimension,
this would not be possible if $\E$
or 
$T_{\bl}L_{d,n}$,
were appropriately general.  
On the other hand, we know that the geometry of our situation is 
sufficiently special that this \textit{does} happen for $\E$
with rank $D$ and
with $K_{d+2}$ support.  Proposition~\ref{prop:inf} asserts that
this is the \textit{only} possibility.
The proof will  rely on
the fact that $K_{d+2}$ is the only graph on $D$ or 
fewer edges that is  infinitesimally dependent
(Proposition~\ref{prop:simplex}). 

First we present the proof of Theorem~\ref{thm:linImage},
which effectively reduces our problem to the linear 
situation covered in Proposition~\ref{prop:inf}.
\begin{proof}[Proof of Theorem~\ref{thm:linImage}]
Clearly, the image of the map must be contained in an $r$-dimensional
linear space spanned by the columns of $\E$.
Suppose that either $r<D$, or $\E$ does not have $K_{d+2}$ support.
Then, from Proposition~\ref{prop:inf} below,
for any generic point
\footnote{See Footnote~\ref{fn:1d}.}
$\bl$,
there must be a coordinate flip $\bl'$
such that
$\Dim(\E(T_{\bl'}L_{d,n}))=r$. 
Then, from the Local Submersion 
Theorem~\cite[page 20]{gp},
the map must be locally surjective onto the 
$r$-dimensional linear space. Thus the image cannot have smaller 
dimension (as a constructible set).
\end{proof}

We are now ready to state the key technical result in this 
section.
\begin{proposition}\label{prop:inf}
Let $\E$ be  a 
$D\times \edgecard$ matrix with rank $r$.
Suppose there is a 
generic point
$\bl\in L_{d,n}$ such that $\bl$ and 
all of its coordinate flips $\bl'$  have the property
that 
$\Dim(\E(T_{\bl'}L_{d,n})) <r$.
Then $r=D$ and $\E$ has $K_{d+2}$ support.
\end{proposition}

\subsection{Proof of Proposition~\ref{prop:inf}}
The rest of the section is occupied with the proof, 
which we break down into steps.
We use a technical lemma about coordinate 
negation and determinants that is relegated to Appendix~\ref{sec:det}.

\begin{definition}
\label{def:sf}
A \defn{sign flip matrix} $\S$ is a diagonal matrix 
with $\pm 1$ on the diagonal.  A \defn{coordinate flip}
of a point or subspace it its image under a sign 
flip matrix.
\end{definition}

\begin{definition}
\label{def:tans}
Let $\bm$ be a generic point in $M_{d,n}$, 
and $T_{\bm}M_{d,n}$ be its complex analytic  tangent.
We can describe 
$T_{\bm}M_{d,n}$ by a
$(dn-C)\times \edgecard$
complex matrix $\T_\bm$. (The row ordering is not relevant).

Let $E$ be an edge measurement ensemble. 
Recall that the map $m_E(\cdot)$ 
is the composition of 
$m(\cdot)$ with $\pi_{\bar{E}}$, and
that from the infinitesimal rigidity 
of $K_n$,  $m(\cdot)$ is a submersion  at 
a generic $\p$.
So using the
chain rule, 
we see that 
infinitesimal independence of $E$ 
is the same as the columns of $\T_\bm$ 
corresponding to $E$ being linearly independent.
The same is true 
(as the Jacobian of $s(\cdot)$ at $\bl$ 
is diagonal and
bijective) 
of the matrix $\T_\bl$ that expresses
the tangent space 
$T_{\bl}L_{d,n}$ at a generic point $\bl$ in 
$L_{d,n}$~\footnote{\label{fn:1d}For $d=1$, where there are no generic points of $L_{1,n}$, we can use points
$\bl$ with no vanishing coordinates and 
where $s(\bl)$ is generic in $M_{1,n}$.}.
\end{definition}

The first step is to restrict to an interesting range
of $n$.
\begin{lemma}\label{lem:flips:n-D-range}
Proposition~\ref{prop:inf} holds when $n<d+2$.
\end{lemma}
\begin{proof}
When  $n \le d+1$, 
$T_{\bl}L_{d,n}$ is equal to the full embedding space, and 
thus $\Dim(\E(T_{\bl}L_{d,n})) =r$.  Proposition~\ref{prop:inf}
is then trivial in this case.
\end{proof}
Thus, from now on, we may assume that 
$n\ge d+2$.

Let $\T$ be a $(dn - C)\times N$
matrix with rows spanning the tangent space $T_{\bl}L_{d,n}$.
The complex analytic  tangent space at a
smooth point of a variety with pure dimension
has the same dimension
as the variety, which explains the shape of $\T$.

\paragraph{Block form and column basis}
Each column of $\E$ and $\T$ corresponds to an edge in $K_n$. We are
going to make use of edge-permuted versions of these matrices that have
particular block structures. To this end, we are now going to look at
the columns of $\E$ and determine which subsets can form a basis,
$\E_2$, of a linear space of dimension $r$.  So we permute and then
partition the columns of $\E$ into a block form
\[
	\begin{pmatrix}
    	\E_1 & \E_2
    \end{pmatrix}.
\]
where $\E_1$ is $D\times (N-r)$
and $\E_2$ is $D\times r$.
We define a column basis, $\E_2$ of $\E$, to be \defn{good} when $r=D$ and 
the columns of $\E_2$ correspond to the edges of a $K_{d+2}$. 
Any other column basis $\E_2$ will be called \defn{bad}.

Suppose that $\E$ has $K_{d+2}$ support 
and $r=D$. 
then the $r$ columns of $\E$ corresponding 
to the edges of this $K_{d+2}$ 
must form the only column basis of $\E$. 
Moreover, it is good. 

\begin{lemma}\label{lem:flips:bad-basis}
If  $\E$ does not 
have $K_{d+2}$ support or $r < D$,
then there is a bad column basis for $\E$. 
\end{lemma}
\begin{proof}
If
$r < D$, then by definition, 
no column basis can be good.  From now on, then, 
assume that $r=D$.

If $\E$ is supported on only $D$ columns, there is a unique column 
basis $\E_2$.  
Thus in this case, non-$K_{d+2}$ support for $\E$ will imply that the
unique column basis is bad.

Suppose instead there are more than $D$ 
non-zero columns of $\E$.
Thus, starting from, say, a good basis $\E_2$, we 
can exchange a non-zero column of $\E_1$ with an appropriate one 
from $\E_2$ to obtain another basis which 
is bad: removing an edge from a $K_{d+2}$
and replacing it with any other edge results in a 
graph that cannot be a $K_{d+2}$ (it has more vertices).
\end{proof}
\begin{remark}\label{rem:flips:bad-basis-iff}
In light of the paragraph preceding this lemma, 
Lemma~\ref{lem:flips:bad-basis} can be made into 
an ``if and only if'' statement.
\end{remark}

Going back to $\T$ and applying the same column used obtain 
$	\begin{pmatrix}
    	\E_1 & \E_2
    \end{pmatrix}$,
we get a block form
\ba
    \begin{pmatrix}
        \T_1  & \T_2 
    \end{pmatrix}
\ea
where 
$\T_1$ is $(dn-C)\times (N-r)$ and
$\T_2$ is $(dn-C)\times r$.

\begin{lemma}\label{lem:flips:T2-cols}
Assuming that $\E_2$ is a bad basis of $\E$ and $\bl$ is generic, 
the matrix $\T_2$ has rank $r$ (and 
in particular linearly independent columns) 
\end{lemma}
\begin{proof}
Since $(\E_1,\E_2)$ arises from a bad basis, and we have
only applied column permutations, the columns
of $\T_2$ corresponds to a subgraph $G$ of $K_n$
with at most
$D$ edges which is not $K_{d+2}$.  
Proposition~\ref{prop:simplex} tells us that the edges of 
$G$ are infinitesimally independent.
So, by 
genericity of $\bl$, these columns are 
linearly independent (Definition~\ref{def:tans}).
\end{proof}

\paragraph{Row rank}

\begin{lemma}\label{lem:flips:T1-T2-rows}
Assuming that $\E_2$ is a bad basis of $\E$ and $\bl$ is generic.
Then the block matrix 
\(
	\begin{pmatrix}
    	\T_1 & \T_2
    \end{pmatrix}
\) contains $r$ rows,
\(
	\begin{pmatrix}
    	\T'_1 & \T'_2
    \end{pmatrix}
\), such that 
$\T'_2$ forms a non-singular matrix.
\end{lemma}
\begin{proof}
Since we have a bad basis, 
from Lemma~\ref{lem:flips:T2-cols},
$\T_2$ has $r$ linearly independent columns
and thus $r$ linearly independent rows.
We can select any set of rows corresponding to a row basis of $\T_2$.
\end{proof}

Similarly, we have
\begin{lemma}\label{lem:flips:E1-E2-rows}
Let $\E_2$ be a column basis for $\E$.
Then the block matrix 
\(
	\begin{pmatrix}
    	\E_1 & \E_2
    \end{pmatrix}
\) 
contains $r$ rows,
\(
	\begin{pmatrix}
    	\E'_1 & \E'_2
    \end{pmatrix}
\), such that 
$\E'_2$ forms a non-singular matrix.
\end{lemma}

Next, we derive an implication of $\E$ dropping rank on
the tangent space.

\begin{lemma}
\label{lem:flips-row-rank}
Suppose there is a 
generic point
$\bl\in L_{d,n}$ such that $\bl$ and 
all of its coordinate flips $\bl'$  have the property
that 
$\Dim(\E(T_{\bl'}L_{d,n})) <r$.
Let $\E_2$ be a bad basis for $\E$.
Let $\S_1$ be any  
any $(N-r)\times (N-r)$ sign flip matrix, and 
$\S_2$, any $r\times r$ sign flip matrix.

Then
the $r\times r$ matrix 
$\Z:= 
\E'_1 \S_1 \trans{\T'_1}
+ \E'_2 \S_2 \trans{\T'_2}
$ is singular.
\end{lemma}
\begin{proof}
Let $\S$ be the $N\times N$ be the
sign flip matrix with the $\S_i$ as its
diagonals. Let $\bl'$ be the point obtained from $\bl$ under the sign flips of $\S$.
Then we have
$\Dim(\E(T_{\bl'}L_{d,n})) =
\rank(\E \S \trans{\T}) = 
\rank(
\E_1 \S_1 \trans{\T_1} 
+ \E_2 \S_2 \trans{\T_2}) 
\ge
\rank(
\E'_1 \S_1 \trans{\T'_1} 
+ \E'_2 \S_2 \trans{\T'_2})$

If for some $\S$, the matrix $\Z$ were non-singular,
then we would have a certificate that
$\E$ does not drop rank on that coordinate flip of the 
tangent space, in contradiction to the hypothesis 
on $\Dim(\E(T_{\bl'}L_{d,n}))$.
\end{proof}
\begin{remark}
The rank of $\Z$ may change as the $\S_i$ do, but it 
cannot rise to $r$.
\end{remark}

\paragraph{Conclusion of the proof}

From Lemma~\ref{lem:flips:bad-basis},
if  $\E$ did not 
have $K_{d+2}$ support or $r < D$,
then there would be a bad column basis $\E_2$ for $\E$. 
From Lemma~\ref{lem:flips:T1-T2-rows}, for a generic $\bl$, 
$\T'_2$ would be a non-singular matrix.

Suppose there is a 
generic point
$\bl\in L_{d,n}$ such that $\bl$ and 
all of its coordinate flips $\bl'$  have the property
that 
$\Dim(\E(T_{\bl'}L_{d,n})) <r$.
Then
from Lemma~\ref{lem:flips-row-rank},
for any choice of  $\S_2$, 
the matrix 
$\Z$ is singular. 
Since $\E_2$ is a basis, 
$\E'_2$ is non-singular matrix
(Lemma~\ref{lem:flips:E1-E2-rows}),
thus
$\Z':= \S_2 {\E'_2}^{-1} \Z = 
\S_2 ({\E'_2}^{-1} \E'_1 \S_1 \trans{\T'_1})
+  \trans{\T'_2}$ 
is singular for any choice of $\S_2$.
Thus, Lemma~\ref{lem:nonsing} on determinants and sign flips
applies to $\Z'$, and we
conclude that $\T'_2$ is singular.

The resulting contradiction 
completes the proof of Proposition 
\ref{prop:inf}.
\qed

\section{Automorphisms of $L_{d,n}$}\label{sec:Ldn-automorphisms}

In this section we will characterize the linear automorphisms of
$L_{d,n}$ for all $d$ and $n$.
One key feature will be that we are no longer 
restricted to the case of edge permutations.

We will need to consider a few distinct cases for $d$ and $n$.


\begin{definition}\label{def:signed-perm}
Set $N := \binom{n}{2}$ and identify the rows and 
columns of an $N\times N$ matrix with the 
edges of $K_n$.

A \defn{signed permutation} is an $N\times N$
matrix $\P'$ that is the product $\S\P$ of a 
sign flip matrix $\S$ and a permutation 
matrix $\P$.

A signed permutation $\P':=\S\P$ is 
\defn{induced by a vertex relabeling} if 
$\P$ is induced by a vertex relabeling of $K_n$.
\end{definition}

\subsection{Automorphisms of $L_{d,n}$, $n\ge d+3$}
\label{sec:bign}

Let $d\ge 1$.
This section will be concerned with $L_{d,n}$ where $n$ is larger than
the minimal value, $d+2$.

\begin{theorem}\label{thm:no-regges}
Let $n \ge d+3$.  Then any linear automorphism 
$\A$ of $L_{d,n}$ of is a scalar multiple
of a signed permutation that is induced by a 
vertex relabeling.
\end{theorem}

The plan is to use 
machinery from Section~\ref{sec:maps}
to show that the automorphism must be in the form 
of a generalized edge permutation. We will then be able to
switch over 
to the $M_{d,n}$ setting, where we can 
apply Theorem~\ref{thm:bk-linear}.

\begin{definition}\label{def:support-map}
Let $\A$ be an $N\times N$ matrix.  We identify the 
rows and columns of $\A$ with the edges 
of $K_n$.  This induces a
map $\tau_\A$ from subgraphs of $K_n$ to 
subgraphs of $K_n$ by mapping the subgraph 
associated with a collection of rows 
to the column support of this sub-matrix.
\end{definition}
\begin{lemma}\label{lem:tet-a-tet}
Let 
$n\ge d+2$ and suppose that $\A$ is a 
linear automorphism of $L_{d,n}$.
 Then the associated combinatorial map
$\tau_\A$ induces a permutation on  $K_{d+2}$ 
subgraphs of $K_n$.
\end{lemma}
\begin{proof}
If $\E$ is any $D \times N$ matrix of 
rank $D$, with $\E(L_{d,n}) \subset  L_{d,d+2}$,
then
the map $\E\A$  also has these properties.
Thus, by Theorem~\ref{thm:linImage} both $\E$ and 
$\E\A$ have $K_{d+2}$ support. 
There is such 
an $\E$ for each $K_{d+2}$ subgraph: simply take the matrix
of the edge forgetting map $\pi_{\bar{K}}$, where $K$ is
an edge measurement ensemble comprising the edges of 
this $K_{d+2}$.
This situation is only possible 
if $\tau_\A(T)$ maps each $K_{d+2}$ 
subgraph $T$ to another $K_{d+2}$  subgraph.  

If the map on $K_{d+2}$  subgraphs induced by $\tau_\A$ is not 
injective, then the matrix  $\A$ would have 
more than $D$  rows supported by only $D$ columns, and 
thus $\A$ would be singular.  Since 
$\A$ is a linear automorphism of $L_{d,n}$ it has to 
be invertible, and the resulting contradiction 
completes the proof.
\end{proof}

This lets us prove the following.
\begin{lemma}\label{lem:gen-perm}
Let $n\ge d+3$ and let $\A$ be a linear automorphism
of $L_{d,n}$.  Then $\A$ is a generalized permutation.
\end{lemma}
\begin{proof}
Suppose, w.l.o.g., that the row corresponding to 
the edge $e:=\{1,2\}$ has two non-zero entries corresponding 
to edges $\{i,j\}$ and $\{k,\ell \}$.  
By  Lemma~\ref{lem:tet-a-tet}, 
any $K_{d+2}$ 
 subgraph 
$T$ containing the edge $e$ must be mapped by $\tau_\A$ to a $K_{d+2}$  
subgraph $T'$ that contains the vertex set 
$X:=\{i,j\}\cup \{k,\ell\}$.

Since $|X|\ge 3$ there are at most 
$\binom{n-3}{d-1}$ 
choices 
for $T'$.  Meanwhile, there are $\binom{n-2}{d}$ choices for 
$T$.  Since $n\ge d+3$, we have $\binom{n-2}{d} > \binom{n-3}{d-1}$, contradicting the permutation
of $K_{d+2}$ subgraphs
guaranteed by 
Lemma~\ref{lem:tet-a-tet}.

Thus each row of $\A$ can have at most one non-zero entry.
As a non-singular matrix, this makes $\A$ a generalized
permutation.
\end{proof}
At this point, we want to move back 
to the setting of $M_{d,n}$, which we do with this next 
result.
\begin{lemma}\label{lem:L-to-M}
Let $\A := \D\P$ be a generalized permutation, where
$\D$ is an invertible
diagonal matrix and $\P$ is a permutation matrix.
If $\A$ is a linear automorphism of $L_{d,n}$ 
then $\D^2\P$ is a
linear automorphism of $M_{d,n}$.
\end{lemma}
\begin{proof}
Let $\bl^2$ denote the vector of coordinate-wise square
of a vector $\bl \in \CC^N$; in this proof squares of vectors
are coordinate-wise.  Now we check that
\ba
	\bl^2 \in M_{d,n} & \Rightarrow &  \\
    \bl \in \LN & \Rightarrow & \\
    \D\P\bl\in \LN & \Rightarrow & \text{($\A$ is an automorphism)} \\
    (\D\P\bl)^2\in \ M_{d,n} & \Rightarrow & \\
    \D^2(\P\bl)^2 \in M_{d,n} & \Rightarrow &  \text{($\D$ is diagonal)} \\
    (\D^2\P)\bl^2 \in M_{d,n}& & \text{($\P$ is a permutation)}
\ea
\end{proof}

\begin{proof}[Proof of  Theorem~\ref{thm:no-regges}]
From Lemma~\ref{lem:gen-perm},
any linear automorphism $\A$ of 
$L_{d,n}$ with $n\ge d+3$ is a generalized permutation
$\A=\D\P$.
Lemma~\ref{lem:L-to-M} implies that $\A$ 
gives rise to a generalized edge permutation $\D^2\P$
that is a linear automorphism of $M_{d,n}$. Theorem 
\ref{thm:bk-linear} then tells us that 
$\D^2\P=s^2 \P$ has uniform scale and also is induced
by a vertex relabeling.  Finally
$\A$ is then a scalar multiple of a signed permutation 
(Lemma~\ref{lem:L-to-M} ``forgets'' the signs) as 
required.
\end{proof}

\subsection{Automorphisms of $L_{d,d+2}$, with $d\ge 3$}
\label{sec:min3}
Our next case is when $n$ is minimal, but we will only deal with
the case of $d \ge 3$.

\begin{theorem}\label{thm:no-regges3}
Let $d \ge 3$.  
Then any linear automorphism 
$\A$ of $L_{d,d+2}$  is a scalar multiple
of a signed permutation that is induced by a 
vertex relabeling.
\end{theorem}

The plan is to use some of the structure of the singular locus
of $L_{d,d+2}$ to reduce our problem to that of $L_{d-1,d+2}$.
Then we can directly apply 
Theorem~\ref{thm:no-regges}.

\begin{lemma}
\label{lem:singComp}
Let $d\ge 3$. 
$L_{d-1,d+2}$ is an irreducible subvariety of $\sing(L_{d,d+2})$.
\end{lemma}
\begin{proof}
Looking first at the squared measurement variety, 
from Theorem~\ref{thm:Mvariety}, we know that 
$\sing(M_{d,d+2})=M_{d-1,d+2}$. 

Let $Z$ be the locus of $\CC^N$ where at least one coordinate vanishes, and 
let  $S:= L_{d-1,d+2} - Z$.
Thus from Lemma~\ref{lem:Asmooth}, the points in $S$, 
are (algebraically) singular in $L_{d,d+2}$.
So $S$ is contained in  $\sing(L_{d,d+2})$.

From Theorem~\ref{thm:Lvariety}, when $d\ge 3$, we have
$L_{d-1,d+2}$ 
is irreducible.
The set $S$ is obtained from 
$L_{d-1,d+2}$ by removing a strict subvariety, which must
be of lower dimension due to irreducibility. Thus $S$ is a full-dimensional 
constructible subset of the irreducible
$L_{d-1,d+2}$.
Thus the Zariski closure of $S$ is $L_{d-1,d+2}$.

Since $\sing(L_{d,d+2})$  is an algebraic variety, it must
contain the Zariski closure of $S$ which is $L_{d-1,d+2}$.
\end{proof}

\begin{lemma}
\label{lem:singBigSpan}
$L_{d-1,d+2}$ has a full-dimensional affine span.
\end{lemma}
\begin{proof}
Since $L_{d-1,d+2}$ contains $L_{1,d+2}$, we just need to show that
this smaller variety has a full-dimensional affine span.

For a fixed $i$, 
let us look at configuration $\p$ of $d+2$ points with  $\p_i$
placed at $1$ and the rest of the points
placed at the origin. Then $\bl:=l(\p)$ has all zero coordinates
except for the $d+1$ edges connecting $\p_i$ to the other points.
Under the symmetry of $L_{1,d+2}$ under sign negation, we can find
points in $L_{1,d+2}$ with the signs of the $\bl$ flipped at will.
Thus using affine combinations of these flipped points together with
the origin
we produce a point on the $l_{ij}$ axis, for any $j$.
Iterating over the $i$ gives us our result. 
\end{proof}

Now we wish to explore the decomposition of 
$\sing(L_{d,d+2})$ 
into
its irreducible components.

For each $ij$, 
Let $Z_{ij}$ be the subvariety of 
$\sing(L_{d,d+2})$ with a zero-valued $ij$th coordinate.
As discussed above in
Lemma~\ref{lem:Asmooth} any singular point that is not contained in
$L_{d-1,d+2}$ 
must have at least one zero coordinate (in order to be
in the ``bad locus'' described there). 
Thus we can 
write 
$\sing(L_{d,d+2})$ as the union of 
$L_{d-1,d+2}$  and the $Z_{ij}$.

For $d \ge 3$, $L_{d-1,d+2}$ is irreducible,
and thus from Lemma~\ref{lem:comp} (applied to the union
of components of $\sing(L_{d,d+2})$) it
must be fully contained in at least one component $C$ of 
$\sing(L_{d,d+2})$. And, again from
from Lemma~\ref{lem:comp} (applied to the union of 
$L_{d-1,d+2}$ and the $Z_{ij}$), 
$C$ must be fully contained in 
either  $L_{d-1,d+2}$ or one of the $Z_{ij}$.
Meanwhile, $L_{d-1,d+2}$
it is not contained in
any $Z_{ij}$. 
Thus we can conclude that:
\begin{lemma}
\label{lem:bigComp}
Let $d \ge 3$.  
$L_{d-1,d+2}$ is a component of 
$\sing(L_{d,d+2})$. 
\end{lemma}

From Lemma~\ref{lem:comp}
(applied to the union of 
$L_{d-1,d+2}$ and the $Z_{ij}$),
any other component
of $\sing(L_{d,d+2})$ must be contained in one of the 
$Z_{ij}$
Thus, we can also
conclude:

\begin{lemma}
\label{lem:singSmallSpan}
Let $d \ge 3$.  
Any component of $\sing(L_{d,d+2})$ 
that is not $L_{d-1,d+2}$ cannot
have a full-dimensional affine span.
\end{lemma}

Now with this understanding of 
$\sing(L_{d,d+2})$ 
established we can move on to the automorphisms.

\begin{lemma}
\label{lem:singAuto}
Let $d \ge 3$. Any linear automorphism $\A$ of $L_{d,d+2}$ must be a 
linear automorphism of $L_{d-1,d+2}$. 
\end{lemma}
\begin{proof}
From Theorem~\ref{thm:Tmap}, $\A$ must be a linear  automorphism
of $\sing(L_{d,d+2})$. 
And from Theorem~\ref{thm:comps} must map 
components of $\sing(L_{d,d+2})$
to components of $\sing(L_{d,d+2})$. 

From Lemma~\ref{lem:bigComp}, 
$L_{d-1,d+2}$ is a component of this singular set and from 
Lemma~\ref{lem:singBigSpan} it has a full-dimensional affine span. 
Meanwhile, from Lemma~\ref{lem:singSmallSpan}, no other component
can have a full-dimensional affine span. Thus, as a bijective linear map, 
$\A$ must map $L_{d-1,d+2}$ to itself.
\end{proof}

And we can finish the proof.

\begin{proof}[Proof of Theorem~\ref{thm:no-regges3}]
The theorem now follows by combining Lemma~\ref{lem:singAuto} together
with Theorem~\ref{thm:no-regges}.
\end{proof}

\subsection{Automorphisms of $L_{2,4}$}\label{sec:auto}

The method of the previous section fails for $\LL$ as $L_{1,4}$
is reducible. In fact, the theorem itself fails 
in this case. The group of linear automorphisms 
is, in fact, larger
than expected. 

In particular, Regge~\cite{regge} (see also, Roberts \cite{regge2}) 
showed that the following linear map always takes the
Euclidean lengths of the edges of a tetrahedral 
configuration in $\RR^2$ to those of a different tetrahedral
configuration in $\RR^2$.
\begin{equation}\label{eq:reggemap}\tag{$\star$}
\begin{aligned}
l'_{13} &=& l_{13} \\
l'_{24} &=& l_{24} \\
l'_{12} &=& (-l_{12}+ l_{23}+ l_{34}+ l_{14})/2 \\
l'_{23} &=& (l_{12}-l_{23}+ l_{34}+ l_{14})/2 \\
l'_{34} &=& (l_{12}+ l_{23}- l_{34}+ l_{14})/2 \\
l'_{14} &=& (l_{12}+ l_{23}+ l_{34}- l_{14})/2
\end{aligned}
\end{equation}

\begin{remark}
In light of Theorem~\ref{thm:no-regges3}, we see that there are no analogues to Regge symmetries in dimensions greater than $2$.
\end{remark}

Below we will fully characterize the automorphism group  of $\LL$.
Luckily for our reconstruction application, when we restrict
our automorphisms to have only non-negative entries, only 
the expected symmetries will remain.

\begin{definition}\label{def:real-linear-automorphism}
A linear automorphism  $\A$ of $L_{2,4}$ is \defn{real} if its matrix has
only real entries, \defn{rational} if its matrix has only rational 
entries, and \defn{non-negative} if its matrix contains only 
real and non-negative entries. 
\end{definition}

Clearly there are $24$ linear automorphism that arise by simply permuting
the $4$ vertices. There are also the $32$ linear automorphisms
that arise from optionally negating up to $5$ 
of the coordinate axes in $\CC^6$.
Combining these  gives us a discrete group of $768$ linear
automorphisms.

Because any global scale will be an automorphism, the group  of
linear automorphisms of $L_{2,4}$ is not a discrete group.  We now
define several groups that will play a role in our analysis.

\begin{definition}\label{def:L24-groups}
Define $\Aut(L_{2,4})$ to be the linear automorphisms of $L_{2,4}$.  Let the
group $\PP \Aut(L_{2,4})$ be induced on the equivalence classes
of $\A\in \Aut(L_{2,4})$ under the relation ``$\A'$ is a complex scale of $\A$''.
We define $\PP \Aut(\sing(L_{2,4}))$ via a similar 
construction.  Importantly, 
we will see below that 
$\PP \Aut(\sing(L_{2,4}))$ 
is the automorphism group of a projective subspace arrangement
and thus is a discrete group.
Also, we have 
$\PP \Aut(L_{2,4}) < \PP \Aut(\sing(L_{2,4}))$.  Thus, all the 
``projectivized'' groups we define are discrete.

We also consider the real subgroup $\Aut_{\RR}(L_{2,4})$.  This has a 
counterpart $\PP \Aut_{\RR}(L_{2,4})$ of equivalence classes up to 
real scale, and  $\PP_{+} \Aut_{\RR}(L_{2,4})$, on equivalence classes defined 
up to \textit{positive} scale.  It is well-defined to refer to an element 
of $\PP_{+} \Aut_{\RR}(L_{2,4})$ as being non-negative, since any equivalence
class containing a non-negative $\A$ consists entirely of non-negative matrices.
\end{definition}
The main theorem of this section characterizes the linear automorphisms of $L_{2,4}$
as follows. The proof is in the next subsections.

\begin{theorem}
\label{thm:auto2}
The group $\PP \Aut(L_{2,4})$
is of order $11520=768\cdot 15$ and is 
generated by linear automorphisms of $L_{2,4}$ that are
rational.

The group $\PP_{+} \Aut_{\RR}(L_{2,4})$ is of order $23040$
and is isomorphic to the Weyl group $D_6$.  The subset of 
non-negative elements of $\PP_{+} \Aut_{\RR}(L_{2,4})$ is a 
subgroup of order $24$ and acts by relabeling the vertices of $K_4$.
\end{theorem}
\begin{remark}\label{rem:dpt-mo-question}
That $\PP_{+} \Aut_{\RR}(L_{2,4})$ contains a subgroup 
isomorphic to $D_6$ is based on conversations with
Dylan Thruston (see \cite{T17}) and has antecedents 
in~\cite{doyle}.
See \cite{I17,W17} for other 
geometric connections.
\end{remark}
\begin{remark}
The group $\PP_{+} \Aut_{\RR}(L_{2,4})$ is 
in fact generated by 
the edge permutations induced by vertex relabelings, 
sign flip matrices, and the one Regge symmetry of \eqref{eq:reggemap} (see supplemental script).
\end{remark}
The rest of this section develops the proof of Theorem~\ref{thm:auto2}.

\paragraph{The Singular Locus of $\LL$}

In this section, we will study the singular locus of 
$\LL$.
This will be used for the proof of 
Theorem~\ref{thm:auto2}, which characterizes 
the linear automorphisms of $\LL$. In particular, a
linear
automorphism of a variety must also be a linear 
automorphism of its singular locus.

\begin{theorem}
The singular locus $\sing(\LL)$ consists of the union of $60$ $3$-dimensional 
linear subspaces. These subspaces can be partitioned into three types, which
we call I, II and III. 

Type I: 
There are $32$ subspaces of this type. They arise from configurations of $4$ collinear points, and together
make up $L_{1,4}$.
They are each defined by (the vanishing of) three equations of the following form:
\ba
    l_{12} - s_{13}l_{13} + s_{23}l_{23} \\
    l_{12} - s_{14}l_{14} + s_{24}l_{24} \\
    s_{13}l_{13} - s_{14}l_{14} + s_{34}l_{34} 
\ea
where each $s_{ij}$ takes on the values $\{-1,1\}$. 

Type II: 
There are $24$ subspaces of this type. They arise when one pair of vertices is
collapsed to a single point. For example, if we collapse $\p_1$ with $\p_2$, we get the 
equations:
\ba
l_{12} \\
l_{13} - s_{23}l_{23} \\
l_{14} - s_{24}l_{24}
\ea
This gives us $4$ subspaces, and we obtain this case by collapsing any of the
$6$ edges.

Type III: 
There are $4$ subspaces of this type.
They arise by setting the three edges lengths of one triangle to zero. For example:
\ba
l_{12} \\
l_{13}  \\
l_{23} 
\ea
\end{theorem}
\begin{proof}
The singular locus of a variety $V$ is defined by adding
to the ideal  $I(V)$,
the equations that express a rank-drop in the  
Jacobian matrix of  a set of equations generating $I(V)$. 

We first verify in the Magma CAS that the ideal
defined by our single simplicial volume
determinant equation is 
radical.\footnote{Magma 
does this check over the field $\QQ$, but
since $\QQ$ is a perfect field, this implies that the ideal
is also radical under any field extension~\cite[Page 169]{milne}.}
This also follows from~\cite{cmirr}.

In Magma, we calculate the Jacobian of this equation to express
the singular locus. Magma is then able to  
factor this algebraic set into  components (that are irreducible over $\QQ$), 
and in this case outputs the above decomposition. (See supplemental script.)
\end{proof}

\paragraph{Flats and intersection graph}

Theorem~\ref{thm:Tmap} tells us that any linear automorphism of $\LL$ must 
be a linear automorphism of its singular set, and so must map each of its singular 
three-dimensional subspaces to some three-dimensional singular subspace. As a linear automorphism, 
it must also preserve the intersection lattice 
of the three-dimensional singular subspace arrangement. Therefore, by finding the set 
of linear automorphisms that preserve the intersection lattice of these subspaces, 
we can constrain our search for automorphisms of $\LL$ to just that set. 
Combinatorial descriptions of an intersection lattice of a subspace arrangement 
can be constructed in many ways. Here, it suffices to consider a partial 
description that comprises the three-dimensional singular subspaces and their 
one-dimensional intersections.  

\begin{definition}
We denote by \defn{${\cal V}_3$} the set of singular three-dimensional subspaces of $\LL$. We denote by \defn{${\cal V}_1$} the set of one-dimensional subspaces created as the intersections of all pairs and triples of spaces in ${\cal V}_3$.
\end{definition}

\begin{lemma} The set of one-dimensional subspaces ${\cal V}_1$ consists of $46$ elements. These come in $3$ classes:

Type I: 
There are $6$ one-dimensional subspaces of this type. 
They are generated by vectors of the form 
\ba
\e_i
\ea
where $\e_i$ is one of the coordinate axes of $\CC^6$.

Type II: 
There are $24$ one-dimensional subspaces of this type. 
They are generated by vectors of the form 
\ba
\e_i \pm \e_j \pm \e_k \pm \e_l
\ea
where $i,j,k,l$ correspond to the four edges of
a 4-cycle.
These measurements correspond to 
collapsing two sets of two vertices that are connected by four edges.

Type III:
There are $16$ one-dimensional subspaces of this type. 
They are generated by vectors of the form 
\ba
\e_i \pm \e_j \pm \e_k \ea
where $i,j,k$ 
correspond to three edges incident to one vertex.
These measurements correspond to 
collapsing one triangle.
\end{lemma}

\begin{proof}
This follows directly from calculating the intersections of all pairs and triples of the $60$ singular subspaces of $\LL$. This has been done in the Magma CAS. 
(See supplemental script.)
\end{proof}

\begin{definition}
We define \defn{$\incidencegraph$} as the bipartite graph that has one set of vertices corresponding to the three-dimensional singular subspaces of $\LL$ (one vertex for each three-dimensional subspace), the other set of vertices corresponding to the one-dimensional intersection subspaces ${\cal V}_1$ (one vertex for each one-dimensional subspace), and an edge between vertex $i$ of the first set and vertex $j$ of the second set whenever the $i$th three-dimensional subspace includes the $j$th one-dimensional subspace.
\end{definition}

\begin{definition}
A \defn{graph automorphism} of a bipartite (two-colored) graph is a permutation $\graphautomorphism$ of the vertex set such that the color of vertex $i$ is the same as the color of $\graphautomorphism(i)$, and vertices $(i,j)$ form an edge if and only if $\left(\graphautomorphism(i),\graphautomorphism(j)\right)$ also form an edge.
\end{definition}

By finding the automorphisms of the graph $\incidencegraph$
we can constrain our search for automorphisms of 
$\{{\cal V}_3,{\cal V}_1\}$, and thus of $\LL$.

\begin{lemma}
\label{lem:graphauto}
The bipartite graph $\incidencegraph$ has $11520$ automorphisms.
Under this automorphism group, the graph has three orbits. One orbit corresponds to the set of $60$ three-dimensional singular subspaces. 
Another orbit corresponds to the subset of $30$ one-dimensional subspaces
in ${\cal V}_1$
of type I and II.
A third orbit corresponds to the subset of $16$ one-dimensional subspaces 
of type III.
\end{lemma}

\begin{proof}
We have computed this using Nauty~\cite{nty}
within Magma.
(See supplemental script.)
\end{proof}

\subsubsection{Graph automorphisms to arrangement automorphisms}
A priori, it might be the case that some of these
graph automorphisms do not arise from a linear transform of $\CC^6$
act as an automorphism on the \emph{subspace arrangement} $\{{\cal V}_3, {\cal V}_1\} \subset \CC^6$. We rule this out.

\begin{lemma}
\label{lem:l24Lin}
Each of the graph automorphisms of $\incidencegraph$ gives rise to  a unique   linear automorphism of the arrangement
$\{{\cal V}_3,{\cal V}_1\}$
on $\LL$, up to a global scale.
Each equivalence class of such linear maps contains a rational-valued matrix. 
\end{lemma}
\begin{proof}
Each graph automorphism $\graphautomorphism$ 
gives rise to a permutation of the spaces in ${\cal V}_3$. A $6\times 6$ matrix $\A$ describing a linear transform that maps the three-dimensional subspaces in the same manner must satisfy $540=60\cdot 9$ linear homogeneous constraints, nine for each pair $\left(i,\graphautomorphism(i)\right), i\in {\cal V}_3$. 

Magma gives us a generating set of size $6$ for the group
of graph automorphisms.
For each of the $6$ generators of the graph automorphism group, we write out the system of linear constraints. When doing so, we discover that this system always has a solution that is unique, up to a global scale. The $540\times 36$ constraint matrix can always be written as a rational-valued matrix, since the subspace arrangement $\{{\cal V}_3, {\cal V}_1\}$ can be defined using rational-valued coefficients.
(See supplemental script.)
\end{proof}

\subsubsection{Arrangement automorphisms are $L_{2,4}$ automorphisms}

It might also be possible that there are 
linear transforms which preserve 
the subspace arrangement $\{{\cal V}_3, {\cal V}_1\}$, but do 
not preserve the entire $\LL$ variety.  We rule this out as well.

\begin{lemma}
\label{lem:l24Auto}
Each of the graph automorphisms of $\incidencegraph$ gives rise to  a unique   linear automorphism
on $\LL$, up to a global scale.
Each equivalence class of such linear maps contains a rational-valued matrix. 
\end{lemma}
\begin{proof}
From Lemma~\ref{lem:l24Lin}, each of the graph automorphisms
gives rise to a, unique up to scale, rational-valued
linear automorphism of our arrangement.
When we pull back the single defining equation of $\LL$ through each such invertible linear map,
we verify that we recover said equation. Thus this map
is a linear automorphism of $\LL$.
\end{proof}

\subsubsection{Reflection group}
Next, we make a definition that will be helpful in establishing the connection between 
$\PP_{+}\Aut_{\RR}(L_{2,4})$
and the Weyl group $D_6$. For definitions, see~\cite{humphreys}.

\begin{definition}
We define the \defn{reflection group} $\reflectiongroup$ as the real
matrix group generated by the set of reflections in $\RR^6$ across the 
$30$ hyperplanes that are orthogonal to the $30$ one-dimensional real intersection subspaces of type I and II.
\end{definition}

The following lemma was based on conversions with
Dylan Thurston.

\begin{lemma}
\label{lem:reflgrp}
The reflection group $\reflectiongroup$ is of order $23040$, and is isomorphic to the Weyl group $D_6$. The reflection group leaves the variety $\LL$ invariant.
\end{lemma}

\begin{proof}
From the $30$ vectors that generate $\reflectiongroup$, we generate a larger set of $60$ vectors $\phi$ that has the same reflection group as follows: For each vector $\f$ in the original 30-set, we create two vectors $\pm 2\f/\Vert\f\Vert$ in the 60-set. 
Next, we verify that the set $\phi$ is a (reduced, crystallographic) root system by: i) applying each generator of the  group $\reflectiongroup$ to the set $\phi$ and verifying that it leaves the set invariant; and ii) verifying that the set satisfies the integrality condition $\forall \f,\g \in \phi, 2(\f\cdot \g)/\Vert\f\Vert \in \ZZ$. 

A reflection group of a root system is a Weyl group. To prove the first part of the lemma, we need only classify the root system (and thus the Weyl group) according to the finite catalog of rank 6 possibilities. We use the procedure described in \cite[page 48]{humphreys}, which we summarize here. 

We begin by choosing any vector $\h \in \QQ^6$ that is not proportional or perpendicular to a vector in $\phi$, and then we identify the subset of \emph{positive roots} $\phi^+:=\{\f : \f\in \phi, (\h\cdot \f)>0\}$. Since $\phi$ is a root system, it will be the case that $|\phi^+| = |\phi|/2 = 30$. Among the positive roots, we identify the subset of \emph{simple roots} as the vectors $\f \in \phi^+$ that cannot be decomposed as $\g_1+\g_2$ for some $\g_i\in\phi^+$. By construction, simple roots form a basis for the embedding vector space, so in the present case there will be $6$ of them. Finally, we can classify the group by examining the pattern of pairwise angles between simple roots.

Applying this calculation to our root system, we find that the
pairwise angles between the simple roots are $0$ or $2\pi/3$. We draw
a Dynkin diagram that has one vertex for each simple root and
an edge $(i,j)$ whenever the angle between roots $i$ and $j$ is
$2\pi/3$.  Doing so, we find that this diagram is of type $D_6$. This means  that the reflection group is isomorphic to the Weyl group $D_6$, which is
of order $23040$. This proves the first part of the lemma.

To prove the second part of the lemma, we use the fact that the reflection group $\reflectiongroup$ is generated by the 6 reflections from the simple roots. We pull back the single defining equation of $\LL$ through each of these 6 linear maps, and we verify that we recover said equation.

Note that the group could also be identified from its
computed order.
(See supplemental script.)
\end{proof}

\subsubsection{Proof}
The proof of our theorem is  now nearly complete.

\begin{proof}[Proof of Theorem~\ref{thm:auto2}]
From Theorem~\ref{thm:Tmap}, a linear automorphism of
$\LL$ must be a linear automorphism of its singular set
${\cal V}_3$, and thus must preserve the incidence
structure of 
$\{{\cal V}_3,{\cal V}_1\}$. Any linear automorphism
of this incidence structure must give rise to a
graph automorphism of $\incidencegraph$.
By Lemma~\ref{lem:graphauto}, there 
are $11520$ graph automorphisms of $\incidencegraph$, and from 
Lemma~\ref{lem:l24Auto}, each gives rise to a 
rational valued linear automorphism of $\LL$, unique up to scale. 
Summarizing, we have shown that $\PP \Aut(L_{2,4}) = \PP \Aut(\sing(L_{2,4})$,
and that both of these groups are isomorphic to the automorphism group of the 
graph $\Delta$.  Lemma~\ref{lem:l24Auto} also implies that each equivalence 
class in $\PP \Aut(L_{2,4})$ contains a rational representative, so 
this group can be generated by rational matrices.

Because of the rational generators mentioned above, the group 
$\PP \Aut_{\RR}(L_{2,4})$ is isomorphic to the others.  It then follows that
the order of $\PP_{+} \Aut_{\RR}(L_{2,4})$ is $23040 = 2\cdot 11520$.

Next, we deal with the classification of $\PP_{+} \Aut_{\RR}(L_{2,4})$.
By Lemma~\ref{lem:reflgrp} (specifically the second statement), the 
elements of $W$ generate some subgroup $G$ of  $\PP_{+} \Aut_{\RR}(L_{2,4})$.
In fact,  no two elements of $\reflectiongroup$ are related by a positive scale,
so $W$ is isomorphic to this $G$.  The first part of Lemma~\ref{lem:reflgrp} says that
$W$ has the same order as  $\PP_{+} \Aut_{\RR}(L_{2,4})$, so $W$ and $\PP_{+} \Aut_{\RR}(L_{2,4})$ are isomorphic.

For the third part of the theorem, we need only test 
$23040$ matrices and retain those that have only non-negative entries. 
This has been done in the Magma CAS, and indeed, it yields only the $24$ 
edge permutations induced by vertex relabelings.  
(See supplemental script.)
This is, in particular, 
a subgroup of $\PP_{+} \Aut_{\RR}(L_{2,4})$.

\end{proof}

\subsection{Automorphisms of $L_{1,3}$}

For completeness, we describe here the 
linear automorphisms of $L_{1,3}$. 
(As $L_{1,n}$ is reducible, this result will not be useful in our 
reconstruction setting.)

\begin{theorem}\label{thm:l13-aut}
Any linear automorphism 
$\A$ of $L_{1,3}$  is a scalar multiple
of a signed permutation that is induced by a 
vertex relabeling.
\end{theorem}
\begin{proof}
$L_{1,3}$ comprises $4$ hyperplanes.
Each permutation on these $4$ planes gives us
at most a single linear automorphism of $L_{1,3}$ up to scale.
Thus $\PP\Aut(L_{1,3})$ is isomorphic to a subgroup of $S_4$
and, in particular, has order at most $24$.

Meanwhile $\PP\Aut(L_{1,3})$ contains a subgroup of order 
$24$ generated by vertex relabelings and sign flips.  By 
the above, this must be the whole group.
\end{proof}
\begin{remark}
If we want to see $S_4$ acting by sign flips and 
coordinate permutations, we can observe
that these maps are symmetries of the cube that 
permute the opposite corner diagonals.  
\end{remark}

\section{Consistency Implies Correctness}
\label{sec:consist}

We are now in a position to show that:  
If we take $D$ values from our data set of path measurements, and they
are \defn{consistent} with the $D$
edge lengths of a $K_{d+2}$ in $\RR^d$, then in
fact \emph{they do arise} \defn{veridically}, up to scale in this way.
Likewise, in the loop setting, if they 
``look like'' an appropriate \defn{canonical} set of  $D$
loops, 
then in
fact \emph{they do arise}, up to scale, in this way. 

\begin{definition}
\label{def:rational:rank}
Given a finite sequence of $k$ complex numbers $w_i$, we say that 
they are \defn{rationally linearly dependent} if there is a
sequence  of rational 
coefficients $c^i$, not all zero, such that 
$0= \sum_i c^i w_i$. Otherwise we say that they are rationally linearly 
independent.
We define the \defn{rational rank} of $w_i$ to
be the size of the maximal subset that is rationally linearly 
independent. 
\end{definition}

\begin{theorem}
\label{thm:consist}
Let $d \ge 2$.
Let  $\w:=(w_1,\dots,w_D)$ have rational rank of
$D$  and
describe a point in $\CC^D$ that
is a point of 
$L_{d,d+2}$. 

Suppose there is a generic configuration $\p$ in 
$\RR^d$ (or simply such that $l(\p)$ is generic in $\LN$)
and $D$ non-negative rational functionals $\gamma_i$
such that $w_i=\la \gamma_i,\p \ra^2$.

Then
there must be a ${d+2}$ point subconfiguration 
$\p_T$ of $\p$, 
such that
$\la \pmb{\gamma}, \p \ra = \w = s\cdot l(\p_T)$, where $s$ is an unknown positive
scale factor. 
The data $\w$ determines this $\p_T$ up to 
a similarity transform.
\end{theorem}

\begin{proof}
We can use the $D$ functionals, $\gamma_i$, as the rows of a matrix $\E$
to obtain a
linear map from $\LN$ to $\CC^D$. This matrix $\E$ maps
$l(\p)$ to $\w$, which we have assumed to be in 
$L_{d,d+2}$.

i) From Theorem~\ref{thm:prin1} (and using Theorem~\ref{thm:Lvariety}) 
we see that 
$\E(\LN) \subset L_{d,d+2}$. 

ii) From Lemma~\ref{lem:depC} and the assumed rational rank of
$\w$,
the rank of $\E$ must be $D$.
Since the image of $\E(\LN)$ is not $D$-dimensional,
then from Theorem~\ref{thm:linImage},
$\E$ is $K_{d+2}$ supported.

iii)
Let $\pi_{\bar{K}}$ be the edge forgetting map 
where $K$ comprises the edges 
of this $K_{d+2}$, and where $K$ is ordered 
such that 
$\pi_{\bar{K}}(L_{d,n}) = L_{d,d+2}$

Thus $\E$ can be written in the form  
$\A \pi_{\bar{K}}$ where
$\A$ is a $D \times D$ non-singular and
non-negative rational matrix.
We have   $\A(L_{d,d+2}) = \A\pi_{\bar{K}}(L_{d,n}) 
= \E(L_{d,n})
\subset L_{d,d+2}$.

iv) Thus, from Theorem~\ref{thm:prin2}, $\A$ must act as a non-negative linear
automorphism on $L_{d,d+2}$. From 
Theorems~\ref{thm:no-regges3} (for $d \ge 3$)
and ~\ref{thm:auto2} (for $d=2$)
$\A$
cannot be more than a permutation on $d+2$ vertices and a positive global scale.
As $\E=\A\pi_{\bar{K}}$ for such an $\A$, 
there must exist an (ordered) ${d+2}$ point subconfiguration $\p_T$
of $\p$, such that 
$\la \pmb{\gamma}, \p \ra = \w = s\cdot l(\p_T)$.

v) This then 
determines $\p_T$ (up to the stated symmetries) by Lemma~\ref{lem:mds}.
\end{proof}

Theorem~\ref{thm:verLoop}, in 
in Section~\ref{sec:canon} below, is the
generalization to the case of loop ensembles. 
 
\begin{remark}
\label{rem:scale}
If the $\gamma_i$ are 
whole valued, then $s$ must be an integer, 
greater than or equal to $1$
for any such
$\p_T$.
This also means that there is a $d+2$
configuration  $\b$ of ``maximal scale'' such that
$\w = l(\b)$, and that $\p_T=1/s\cdot \b$. 
\end{remark}

\begin{remark}\label{rem:rat-rank}
The rational rank $D$ hypothesis is essential as the following example in 
$2$-dimensional shows.  Let 
$\gamma_i$ be any functional, and measure the set $\{3\gamma_i, 4\gamma_i, 
5\gamma_i, 5\gamma_i, 4\gamma_i, 3\gamma_i\}$. 
These measurement values (with rational rank $1$) correspond
to a $K_4$ made by gluing ``345 triangles'' together, no matter what $\p$ is.
Using an arithmetic 
construction from \cite{AE45} of and
the same idea as above, we can make infinitely many 
non-congruent rational rank $1$ measurement sets that have no 
repeated measurement values or $3$ collinear points.

When we are proving  ``global rigidity'' results in Section~\ref{sec:ugr}, 
the assumed trilateration sequence automatically gives rational rank $D$.
On the other hand, a reconstruction algorithm 
(see Section~\ref{sec:recon}) based on Theorem~\ref{thm:consist} 
will have to find the trilateration sequence as it goes.  
The examples here show that such an algorithm has to test rational rank.
\end{remark}

\begin{remark}
\label{rem:consistOpen}
Theorem~\ref{thm:consist} 
can fail at any non-generic point. Although
the generic point set has full measure, it does not include any
(standard topology) open sets, and the non-generic point set is dense
as well.

As discussed in Remark~\ref{rem:gtoz},
this problem 
will be ameliorated if we restrict ourselves to the
whole \emph{$b$-bounded} setting. In this case, there are only a finite set of 
functionals under consideration. These will determine a ``bad''
algebraic subvariety of $L_{d,n}$ (and a bad subvariety
of configuration space)
where our conclusion does not hold, but
it will leave us with a Zariski open set where it does hold. 
\end{remark}

\begin{remark}
\label{rem:eps}
Theorem~\ref{thm:consist}
tells us that, for any generic $\p$, 
a sequence 
of $D$ linearly independent
non-negative rational functionals $\gamma_i$ will 
only generate a  point
$\w$ in $L_{d,d+2}$, when they describe a veridical D-tuple of measurements from $\p$.
But if we allow the
$\gamma_i$ to be any such non-negative
rational functionals, we should be able
to find a set of non-veridical measurements that produce
a point $\w$ that is 
arbitrarily close to $L_{d,d+2}$.
In an application with finite accuracy, this would mean that in a practical reconstruction setting
we 
could not 
determine if this theorem can be applied.

The situation is better in the $b$-bounded setting
since, once we fix $\p$, the finite set of non-veridical 
$D$-tuples will be bounded away from $L_{d,d+2}$.

More specifically, let $\p$ be fixed.
Suppose there is some polynomial $P(\p,\pmb{\gamma})$
that we want to compare to $0$, where we know that 
$\pmb{\gamma}$ is $b$-bounded.
If $P(\p,\pmb{\gamma}) \neq 0$ then, since there are only 
a finite number of possible $\pmb{\gamma}$, there is an 
$\epsilon$ depending only on $\p$, $P$ and $b$ such that 
$P(\p,\pmb{\gamma})$ is bounded by $\epsilon$ away from $0$.
\end{remark}

\subsection{Loop Setting}
\label{sec:canon}

We also wish to generalize Theorem~\ref{thm:consist} so that 
it can be applied to the loop setting.
In particular, instead of looking for the situation where we 
measure $D$ edges of a $K_{d+2}$ we will
look for the situation where we have made $D$ \emph{canonical}
measurements over a $K_{d+2}$. 
Indeed, we consider two such canonical measurements:
one to identify a $K_{d+2}$ ex-nihilo, and one to identify 
a $K_{d+2}$ using a known $d+1$ point subconfiguration along with 
$d+1$ additional measurements.

\begin{definition}
Given a single $K_{d+2}$, with ordered vertices and edges,
we can describe the  $D$
measurements described in Definition~\ref{def:contained}
using a fixed \emph{canonical}
$D\times D$ matrix $\N^d_1$. Each row represents
the edge multiplicities of one measurement loop.
For notational convenience, we order the rows of this matrix so that
each of the first $C$ rows is supported  only over the $C$ edges of the first
$d+1$ points.

In $2$ dimensions, 
we associate the columns of this matrix with the 
following edge
ordering:
$[1,2], [1,3], [2,3], [1,4], [2,4], [3,4]$.
This then gives us
the following \defn{tetrahedral measurement matrix} that measures
three pings and three triangles
\ba
\N^2_1:=
\begin{pmatrix}
2&0&0& 0&0&0\\
0&2&0& 0&0&0\\
1&1&1& 0&0&0 \\
0&0&0& 2&0&0 \\
1&0&0& 1&1&0 \\
0&1&0& 1&0&1 
\end{pmatrix}.
\ea
See Figure~\ref{fig:pt} (bottom left) for the two-dimensional case.

Given an initial $K_{d+1}$, with ordered vertices and edges,
we can describe an ordered $D$
measurements describing the trilateration of a $d+2$nd vertex
off of the first $d+1$ vertices, as defined in Definition~\ref{def:ensemble}
using fixed a $D\times D$ matrix $\N^d_2$. 
The first $C$ rows measure the  edges of the initial $K_{d+1}$
subconfiguration, and the remaining $d+1$ rows 
measure the appropriate pings and triangles. 

In $2$ dimensions, this gives us
the following \defn{trilateration measurement matrix}
that measures three edges,
one ping, and two triangles
\ba
\N^2_2:=
\begin{pmatrix}
1&0&0& 0&0&0\\
0&1&0& 0&0&0\\
0&0&1& 0&0&0 \\
0&0&0& 2&0&0 \\
1&0&0& 1&1&0 \\
0&1&0& 1&0&1 
\end{pmatrix}.
\ea
See Figure~\ref{fig:pt} (bottom right) for the two-dimensional case.
\end{definition}

\begin{customthm}{\ref*{thm:consist}$'$}
\label{thm:verLoop}
With these definitions in place,
we can generalize the condition
in Theorem~\ref{thm:consist} 
that $\w$ describes a point of $L_{d,d+2}$,  to the condition that
$\w$ describes a of 
point of $\N^d_i(L_{d,d+2})$. And we can generalize the
conclusion to read:
$\la \pmb{\gamma}, \p \ra = \w = s\cdot  \N^d_i(l(\p_T))$.
\end{customthm}
\begin{proof}
We follow the structure of the proof of Theorem~\ref{thm:consist}.
Our loop assumptions give us 
$\w = \E(\bl) \in \N^d_i(L_{d,d+2})$ and thus from genericity of $\bl$, 
$\E(L_{d,n})\subset \N^d_i(L_{d,d+2})$.

From its $K_{d+2}$ support,
$\E$ can be written in the form  
$\B \pi_{\bar{K}}$ where
$\B$ is a $D \times D$ non-singular and non-negative rational
matrix and $\pi_{\bar{K}}(L_{d,n})=L_{d,d+2}$.

We have   $\B(L_{d,d+2}) = \B\pi_{\bar{K}}(L_{d,n}) 
= \E(L_{d,n})
\subset \N^d_i(L_{d,d+2})$.
This makes $\A:=(\N^d_i)^{-1}\B$ a linear automorphism of 
$L_{d,d+2}$. Thus we are left with determining the 
linear automorphisms, $\A$ such that 
$\N^d_i\A=\B$ is non-negative.

For $d\ge 3$, from Theorem~\ref{thm:no-regges3},
it is apparent 
that this only occurs when $\A$ is a positive scale of a permutation
that is induced by a vertex relabeling.

For $d=2$, we need to explicitly
do a non-negativity check on our $\N^2_i\A$ over all 
automorphisms $\A$ of $L_{2,4}$,
as characterized in Theorem~\ref{thm:auto2}.
This only requires checking $23040$ matrices,
for both $\N^2_1$ and $\N^2_2$
and has been done in the Magma CAS. 
In both cases, non-negativity only arises for 
$\A$ that are positive scales of a vertex relabeling. (See supplemental script.)

Thus $\E$ is of the form
$\N^d_i\A\pi_{\bar{K}}$ for such an $\A$, and the result follows.
\end{proof}

\section{Global Rigidity}
\label{sec:ugr}

We can now use the results of the previous section
to prove our main rigidity result.

\begin{lemma}\label{lem:partial-trilat2}
Let $d\ge 2$ and let $\p$ and $\q$ be configurations
so that $\p$ is generic and $l(\q)$ is generic in 
$L_{d,n'}$ (for some $n'$).  
Let 
$\pmb{\alpha}$ and $\pmb{\beta}$ be two ensembles such that
$\la \pmb{\alpha}, \p\ra = \la \pmb{\beta}, \q\ra$.

Suppose
that we have two 
``already visited'' subconfigurations
$\p_V$ and $\q_{V'}$ with 
$\p_V =\q_{V'}$.

Suppose  we can 
find $\pmb{\delta}$, a subset of  
$d+1$ functionals  in 
$\pmb{\alpha}$ that trilaterate (in either the path or loop
setting)
some unvisited
vertex $\p_i\in \p_{\bar{V}}$ over 
some visited ${d+1}$ point
subconfiguration $\p_R$ of $\p_V$.

Then we can find an unvisited
$\q_{i'}\in \q_{\bar{V'}}$ such that 
the two subconfigurations $\p_{V\cup\{i\}}$
and $\q_{V'\cup\{i'\}}$ are equal.

\end{lemma}
\begin{proof}
Let $\p_T$ be a subconfiguration consisting of, in some order, all
the points of $\p_R$ along with $\p_i$. 
Let $\w := \N(l(\p_T))$,
with $\N=\I$ in the path setting and  $\N=\N^d_2$ in the loop 
setting.
We have
$\w \in L_{d,d+2}$. 
Lemma~\ref{lem:Cdep} 
guarantees that $\w$  
has  rational rank $D$.

Using the existence of $\pmb{\delta}$,
the fact that
$\la \pmb{\alpha}, \p\ra = \la \pmb{\beta}, \q\ra$,
together with $\p_{V}=\q_{V'}$,
we can find a measurement ensemble $\pmb{\gamma}$ under which we can apply
Theorem~\ref{thm:consist} in the path setting, or 
Theorem~\ref{thm:verLoop} for the loop setting
to $\q$ using this same 
$\w$. 
This guarantees a 
${d+2}$ point subconfiguration 
$\q_{T'}$ of $\q$ such that
$\w = t \cdot \N(l(\q_{T'}))$,
with $\N=\I$ in the path setting and  $\N=\N^d_2$ in the loop 
setting.
Here $t \ge 1$ is an integer scale factor (see Remark~\ref{rem:scale}).
From Lemma~\ref{lem:mds},
we  conclude that
$\p_T$ and $\q_{T'}$ 
are related by a 
similarity.  

By construction $\p_T$ contains the subconfiguration
$\p_R$, which is also a subconfiguration of $\q_{V'}$.
From genericity of $\q$  and Lemma~\ref{lem:noSim}, 
$\p_R$ is similar to no other
subconfiguration of $\q$. Thus $\p_R$ must be a subconfiguration
of $\q_{T'}$.
Similarly, from genericity of $\p$ and Lemma~\ref{lem:noSim}, 
the remaining vertex $\q_i'$ of $\q_{T'}$ not included in $\p_R$
must be unvisited, 
ie. in $\q_{\bar{V'}}$.

Then from Lemma~\ref{lem:extend}, we must have 
$\p_T=\q_{T'}$ and thus 
$\p_{V\cup\{i\}}=\q_{V'\cup\{i'\}}$.

\end{proof}

Applying the above iteratively yields the following:

\begin{lemma}
\label{lem:punchline1}
Let the dimension $d \ge 2$. 
Let $\p$ be a generic configuration of $n\ge d+2$ points
(or simply such that $l(\p)$ is generic in $L_{d,n}$).
Let $\v= \la \pmb{\alpha}, \p \ra$, where 
$\pmb{\alpha}$ is a path (resp. loop) measurement ensemble
that allows for trilateration. 

Suppose that there is a configuration $\q$ of $n'$ 
points that is generic
(or simply such that $l(\q)$ is generic in $L_{d,n'}$), along with a path 
(resp. loop) measurement ensemble $\pmb{\beta}$ 
such that $\v=\la \pmb{\beta},\q \ra$.

Let $\q_{V'}$ 
be the subconfiguration of $\q$ indexed by the vertices within the 
support of $\pmb{\beta}$.
Then there is a vertex relabeling of $\q_{V'}$ such that,
up to congruence,
$\q_{V'}=1/s\cdot \p$,
with $s$ an integer $\ge 1$.
Moreover, under this vertex relabeling,
$\pmb{\beta} = s\cdot \pmb{\alpha}$.

If we also assume that 
$\pmb{\beta}$ allows for trilateration, then 
there is a vertex relabeling  of $\q$ such that,
up to congruence,
$\q=\p$.
Moreover, under this vertex relabeling,
$\pmb{\beta} = \pmb{\alpha}$.
\end{lemma}
\begin{proof}

For the base case, 
the trilateration assumed in $\pmb{\alpha}$
and Lemma~\ref{lem:Cdep} 
guarantees a $K_{d+2}$ contained in  $\pmb{\alpha}$
over a $d+2$ point  subconfiguration  $\p_T$ of  $\p$.
Define 
$\w :=   \N(l(\p_T))$.
with $\N=\I$ in the path setting and  $\N=\N^d_1$ in the loop 
setting. The $w_i$ have rational rank $D$ from 
Lemma~\ref{lem:Cdep}.
We have 
$\w \in L_{d,d+2}$.

Using the fact that
$\la \pmb{\alpha}, \p\ra = \la \pmb{\beta}, \q\ra$ 
we can apply 
Theorem~\ref{thm:consist} in the path setting, or 
Theorem~\ref{thm:verLoop} for the loop setting
to this $\w$ and $\q$ 
and appropriate subensemble $\pmb{\gamma}$ of $\pmb{\beta}$.
We conclude 
that
there is a ${d+2}$ point 
subconfiguration  $\q_{T'}$ of $\q$ such that
$\w = s \cdot \N(l(\q_{T'}))$,
with $\N=\I$ in the path setting and  $\N=\N^d_1$ in the loop 
setting.
Here $s \ge 1$ is an integer scale factor (see Remark~\ref{rem:scale}).

Also,  
from Lemma~\ref{lem:mds},
up to a similarity, we have
$\p_T = \q_{T'}$.

Then, going forward inductively, assume that we have a 
two ``visited''  
subconfigurations 
$\p_{V}$ and
$s \cdot \q_{V'}$, 
are related by a 
global congruence. 

Continuing with 
the trilateration process allowed by  
$\pmb{\alpha}$, we can iteratively apply (with the 
scale $s$ and congruence  
factored out of $\q$ and the scale $1/s$ factored
out of $\pmb{\beta}$)
Lemma~\ref{lem:partial-trilat2} until we have visited all of  
$\p$. 
At this point we will have, that  
up to a similarity,  
$\q_{V'}=\p$.

Since $\p$ is generic,
From Theorem~\ref{thm:functional}, no two distinct functionals can give
the same measurement. The same is true for $\q$. This gives us, after
vertex relabeling and scale equality between all of $\pmb{\alpha}$
and $\pmb{\beta}$.

For the last statement, if we also assume that
$\pmb{\beta}$ allows for trilateration, then we can
reverse the roles of $\p$ and $\q$ in the above argument, 
to get say, $t\cdot \p_{V}=\q$, with $t\ge 1$.
This allows us to remove both the
scale and the possibility of unmeasured vertices in $\q$.
\end{proof}

Next we want to remove the genericity assumption on $\q$,
which we will do by adding the assumption that $n'=n$.

\begin{definition}
A path or loop measurement ensemble acting on $d+1$ or more points
is \defn{infinitesimally rigid} in $d$ dimensions
if, starting at some (equiv. any) generic 
(real or complex) configuration $\p$,
there are no differential motions of $\p$ that preserve 
all of the measurement values, except for differential congruences.
\end{definition}

\begin{lemma}
\label{lem:punchline2}
In dimension $d\ge 2$, let  $\p$ and $\q$ be two 
configurations with the same number of points $n$.
Suppose that 
$\pmb{\alpha}$ and
$\pmb{\beta}$
are two path or loop
measurement ensembles
with $\pmb{\alpha}$
infinitesimally rigid in $d$ dimensions.
And suppose that  
$\v:=\la \pmb{\alpha},\p \ra =
\la \pmb{\beta},\q \ra$.

If $\p$ is a generic configuration, 
then $l(\q)$ is generic in $L_{d,n}$.
\end{lemma}
\begin{proof}
The proof follows exactly like that of Lemma~\ref{lem:triBK2}, using 
$L_{d,n}$ instead of $M_{d,n}$ and our linear measurement processes instead of
$\pi_{\bar{G}}$ and $\pi_{\bar{H}}$.
\end{proof}

And finally we can conclude our ultimate proof.

\begin{proof}[Proof of Theorem~\ref{thm:punchline}]
The theorem follows directly using Lemmas~\ref{lem:punchline2} 
and~\ref{lem:punchline1} (and Remark~\ref{rem:gtoz}).
\end{proof}

This reasoning also leads directly to our
procedural approach for reconstruction described next.
In the reconstruction setting, we do not know the trilateration
ordering of our data $\v$. Instead, we must search for it.

\section{Reconstruction Algorithm}
\label{sec:recon}
Up until this point, we have been exploring what configurations $\q$
other than the underlying configuration $\p$, satisfying potentially
different assumptions, can be used with some measurement ensemble
$\pmb{\beta}$ to produce a given set of measurements $\v$. We now take
a different direction: Given the measurement data $\v$, we 
essentially want to find $\p$ itself, and thus want to
solve for
a configuration that satisfies the same assumptions as $\p$
and, together with some measurement ensemble with the same properties
as $\pmb{\alpha}$, produces $\v$. In particular, we will be using the
assumption that the measurements $\v$ arise from an unknown
measurement ensemble that allows for trilateration, and we will
be looking to reconstruct a generic configuration of points.

For our path (resp.~loop) algorithm we have the following specification.
{\bf Input:} dimension, bounce bound, and data set $(d, b, \v)$.
{\bf Assumption:} Data set $\v$ arises from some generic configuration
$\p$ of $n$ points in $\RR^d$, for some $n$, under some
path (resp.~loop) ensemble that allows for trilateration.
{\bf Output:} Number of points and configuration $(n,\q)$, where $\q$ is related to $\p$
through a vertex relabeling and a Euclidean congruence.

Unlike the setting of the proof of Theorem~\ref{thm:punchline}, which only needed the existence of a trilateration sequence,
we actually have to search for and find an underlying trilateration.
Armed with Theorem~\ref{thm:consist}, our approach is to apply a 
brute force trilateration search as  done by~\cite{dux1} in the
edge setting (see~\cite{dux2} for more details).

The algorithm first exhaustively searches for all $D$-tuples of $\v$ that describe
$K_{d+2}$ subconfigurations of $\p$. Each of these is treated as a 
``candidate base'': We do not know at first which of these candidates will ultimately 
form the base of a complete trilateration sequence. Each candidate base is then grown into an expanding sequence of ``candidate sub-configurations" by finding new points that connect to an existing candidate sub-configuration through a trilaterating set of 
$d+1$ edges or loops. This step involves exhaustively 
searching over all $d+1$-tuples of points within an existing candidate sub-configuration, and over all $d+1$-tuples of values in $\v$; or all $d+1$-tuples of values in $\v$.
The algorithm stops when no candidate sub-configuration can be expanded any further. 

This approach can be applied in both the path and the loop setting.
In the loop setting, we will make use of the matrix $\N^d_1$
to recognize candidate bases, and the matrix  $\N^d_2$ to 
add on a new point to ${d+1}$ points already in some candidate base.

For fixed dimension $d$, this method has worst case time complexity that 
is polynomial in $(|\v|,b)$, though with a moderately large exponent:
The number of candidate bases is at worst polynomial in $|\v|$.
The number of times each candidate sub-configuration grows by one vertex through trilateration is at most $n$. And each exhaustive search step 
is also polynomial in $(|\v|,b)$. The role of $b$ relates to the issue of rational rank, which we explore in the next sub-section.



\subsection{Testing Rational Rank}
\label{sec:rr}

In order to apply Theorem~\ref{thm:consist} to some $D$-tuple of measurement values under consideration, we need to verify that 
it has rational rank $D$. (See Remark~\ref{rem:rat-rank}.)

\subsubsection{Rational rank for $d=2$}

Theorem~\ref{thm:3to6} 
below uses properties of the Fano varieties 
(see Definition~\ref{def:fano}) of
$L_{2,4}$ to 
prove that
for $d=2$,
if a point $\w$, arising as the measurement
values of a generic $\p$, 
is non-singular in $\LL$
and has rational rank of $3$, then it, in fact,
it must have a rational rank of $6$.
(The singular locus of $\LL$
is characterized in 
Section~\ref{sec:auto} and is easy to test for.)

With this in hand, 
we can
check  $\w$ for rational rank 6, simply by checking for rational rank $3$ along with non-singularity of $\w$ in $L_{2,4}$.
In particular we only have to 
check the rank on one arbitrary size-3 subset of $\w$.
If any three values of $\w$ have rational rank 3, then
Theorem~\ref{thm:3to6}, together with the established consistency
and non-singularity, imply that $\w$ 
has rational rank 6. If the three values do not have rational rank 3,
then $\w$ does not have rank 6 and cannot be used with Theorem~\ref{thm:consist}.

If the three values of $\w$ are not
rationally independent, that is, there is a nontrivial solution
$\sum_{i=1}^3 c^i w_i =0$ with rational
$c^i$, 
then 
Lemma~\ref{lem:Cdep} implies that 
this is a rational relation
on its three underlying functionals.
In the $b$-bounded
setting,  
Lemma~\ref{lem:depC} then implies that the
coefficients $c^i$ are bounded integers. As a result
of this bound, we only need to
examine $(2 b^2 + 1)^3$ possible relations on $\w$,
making this test $O(b^6)$.

The situation is even better during the inductive step of trilateration.
Suppose we are trilaterating a fourth point off of an already localized
triangle.
The functionals $\alpha_1,\alpha_2,\alpha_3$ corresponding
to the triangle edge-lengths $w_1, w_2, w_3$ (which form the first three rows of
$\N^2_2$)
are linearly
independent, thus,
from
Lemma~\ref{lem:Cdep},  $\w$ automatically
has  rational rank $3$.

\subsubsection{Rational Rank for $d\ge3$}

In three dimensions, in order to apply Theorem~\ref{thm:consist},
we need to check the $10$ values of $\w$
for rational rank $10$. 
Using Lemma~\ref{lem:depC}, 
this 
gives a complexity factor of 
 $(2 b^9 + 1)^{10}$ 
making this test $O(b^{90})$, which will never 
be tractable by brute force, even for small $b$.
We do not know if there is any way to  
simplify this step as we did for $d=2$
(see Remark~\ref{rem:fano3d}).

One solution for dealing with the 
complexity of testing rational rank for $d\ge 3$
is to add 
other assurances on the
measurement ensemble. 
For example,  in the loop setting,
suppose we assume that $\pmb{\alpha}$
consists only of pings and triangles, with each
passing through $\p_1$.
Although strong, this is not an 
unreasonable assumption for our 
signal processing scenario described in Section~\ref{sec:intro}.
Under this assumption, we know that if
$\w$ consists of $10$ distinct values, then it has rational rank $10$, and thus no explicit test is needed.
\begin{question}
In three dimensions, 
are there weaker assumptions on $\pmb{\alpha}$
that allow for an efficient rational rank test?
\end{question}

\subsection{Scale}

Theorem~\ref{thm:consist} only gives us a $K_{d+2}$
configuration up to an unknown scale. 
Fortunately, as we grow a candidate base, we will only be able to do so 
using the same shared scale, since the proof of
Theorem~\ref{thm:punchline} guarantees that each candidate base will
grow with a consistent scale.
Still, since our algorithm uses a number of 
candidate bases, it may
be the case that we reconstruct various subsets of $\p$ at various scales. 

In more detail, there may be a sub-ensemble of 
$\pmb{\alpha}$ 
that is actually of the form of a trilateration sequence,
but with each path (resp.~loop) measured $s$ times per measurement.
We treat
these as if they were unscaled during trilateration, which is effectively
equivalent to scaling the length of each underlying edge by $s$. This results in the
reconstruction of a candidate configuration $s\cdot\p$ that is a
scaled up version of the true configuration. 

The existence of at least one correctly scaled reconstruction is
guaranteed by the trilateration assumption on the underlying
$\pmb{\alpha}$.
Thus, at the end of the trilateration process,
we 
identify the true configuration as the reconstructed subset
with the smallest scale among all configurations with
the same maximal number of points. This configuration, $\q$, will be equal (up to vertex relabeling and congruence) to the true configuration $\p$.

\subsection{Accuracy Considerations}\label{sec:accuracy}

At various stages, the reconstruction algorithm involves 
calculations using the measurements $\v$. 
In actual computations, these calculations can only
be performed up to some finite numerical accuracy. Additionally, the
measurements themselves will only be known up to some finite
accuracy. The $b$-boundedness assumptions provides us with some weak
guarantees for correctness even in this situation of finite accuracy.

In particular, the consistency, non-singularity, and rank conditions of
Theorem~\ref{thm:consist} all involve
polynomial calculations to make binary decisions during the
execution of the reconstruction procedure. From Remark~\ref{rem:eps},
the $b$-boundedness assumption guarantees that, with enough accuracy
in our data and our calculations, we only need to check these
conditions approximately. However, we note that for all such checks, we generally do not have
any knowledge of the appropriate $\epsilon$, nor do we have a
guarantee of such accuracy in our input data. Therefore, the
reconstruction procedure described previously does not yield an
algorithm in the Turing machine sense.

A different situation
occurs when we perform numerical calculations to compute configuration
points from finite-accuracy length measurements, or vice-versa. The outputs of these calculations are re-used at later stages of the algorithm, so that errors
propagate to subsequent numerical calculations. This suggests that, in practical settings, it would be advantageous to devise global reconstruction procedures that
jointly examine all measurement six-tuples $\w$ and enforce consensus
among them, instead of examining them sequentially in the greedy manner describe above.


\newpage

\appendix

\section{Algebraic Geometry Preliminaries}\label{sec:geometry}

We summarize the needed definitions and facts about
complex algebraic varieties. For more see~\cite{harris}.

In this section $N$ and $D$ will represent arbitrary numbers.

\begin{definition}
A (complex embedded affine) \defn{variety} 
(or \defn{algebraic set}), $V$,
is a (not necessarily strict)
subset of $\CC^N$, for some $N$,
that is defined by the simultaneous
vanishing of a finite set of polynomial equations 
with coefficients in $\CC$
in the 
variables $x_1, x_2, \ldots, x_N$ which are associated with the 
coordinate axes of $\CC^N$.
We say that $V$ is \defn{defined over} $\QQ$ if it can 
be defined by polynomials with coefficients in $\QQ$.

A variety can be stratified as a union of a finite number of
complex manifolds.

A finite union of varieties is a variety.
An arbitrary intersection of varieties is a variety.

The set of polynomials that vanish on $V$ form 
a radical ideal $I(V)$, which is generated by a finite set
of polynomials.

A variety $V$ is \defn{reducible} if it is the proper union of two
varieties $V_1$ and $V_2$. 
(Proper means that $V_1$ is not contained in $V_2$ and vice versa.)
Otherwise it is called
\defn{irreducible}.
A variety has a unique decomposition as a finite proper
union of its
maximal irreducible subvarieties called \defn{components}.
(Maximal means that a component cannot be 
 contained in a larger  irreducible 
subvariety of $V$.)

A variety $V$ has a well defined (maximal) \defn{dimension} $\Dim(V)$, 
which will agree with the largest $D$ for which there
is an open subset of~$V$, in the standard  topology, 
that is a $D$-dimensional complex submanifold of $\CC^N$.

The \defn{local dimension} $\Dim_\genericpoint(V)$ at a point $\genericpoint$ is the
dimension of the highest-dimensional irreducible component of $V$ that contains $\genericpoint$.
If all components of $V$ have the same dimension, we say it has
\defn{pure dimension}.

Any (strict) subvariety $W$ of an 
irreducible variety $V$ must be of strictly lower dimension.

A \defn{constructible set} $S$ is a set that can be defined using a finite
number of varieties and a finite number of Boolean set operations.
Its \defn{Zariski closure} is the smallest variety containing it.
We define the dimension of a constructible set as that of its
Zariski closure.

The image of a variety $V$ of dimension $D$
under a polynomial map is a constructible set $S$ of dimension
at most $D$.
If $V$ is irreducible, then so 
too is the Zariski closure of $S$.
If $V$ is defined over $\QQ$, then so too is 
$S$~\cite[Theorem 1.22]{basu}.

We call $\A$ a \defn{linear automorphism} of $V$ if it is a bijective linear map on $\CC^N$
such that $\A(V)=V$.

\end{definition}

\begin{theorem}
\label{thm:Zdense}
Any variety $V$ is a closed subset of $\CC^N$ in the standard 
topology.
This means that 
the Zariski topology is coarser than the
standard topology.
Thus, if a subset $S$ of $\CC^N$
is standard-topology dense in a variety $V$, then
$V$ is the Zariski closure of $S$.
\end{theorem}
See~\cite[Page 8]{invite}.

We will need the following easy lemmas.

\begin{lemma}
\label{lem:bij}
If $\A$ is a bijective linear map 
on $\CC^N$. Then the image under $\A$
of a variety $V$ is a variety
of the same dimension.
If $V$ is irreducible, then so too is this image.
\end{lemma}
\begin{proof}
The image $S := \A(V)$ 
must be a constructible set.

Since $\A$ is bijective, then
there is also map $\A^{-1}$ acting on $\CC^N$,
and $S$ must be the inverse image
of $V$ under this map.
Thus, by pulling back the defining equations of $V$ through
this map, we see that $S$ must also be a variety.

The dimension  follows 
from the fact that maps cannot raise dimension, and our map
is invertible.
\end{proof}

\begin{theorem}
\label{thm:comps}
If $\A$ is a bijective linear map on $\CC^N$
that acts as 
bijection between two reducible varieties
$V$ and $W$, then it must bijectively map components
of $V$ to components of $W$.
\end{theorem}
\begin{proof}
From Lemma~\ref{lem:bij},
$\A$ must
map irreducible varieties to irreducible varieties.
As a bijection, it also must preserve subset relations
(which define maximality).
\end{proof}

\begin{lemma}
\label{lem:comp}
Let $V = V_1 \cup V_2$ be a union of varieties.
Then any irreducible subvariety $W$ of $V$ must be 
fully contained in at least one of the
$V_i$.
\end{lemma}
\begin{proof}
If $W$ was not fully contained in either $V_i$, then it could
be written as the proper union of varieties $W = \bigcup_i (W\cap V_i)$
contradicting its irreducibility.
\end{proof}

Next we define a strong notion of generic points in a variety. The motivation
is that nothing algebraically special (and which 
is expressible with rational 
coefficients)
is allowed to happen at such points.
Thus, any such algebraic property holding at such a point must hold at all points.
\begin{definition}
\label{def:gen}
A point in an irreducible variety $V$ defined over $\QQ$ is called 
\defn{generic} if its coordinates do not satisfy any algebraic equation 
with coefficients in $\QQ$ besides those that are satisfied by every point
in $V$.  

The set of generic points has full measure 
in $V$.

A generic real point in $\RR^N$ as in Definition~\ref{def:genR}
is also a generic point in $\CC^N$,
considered as a variety, as in the current definition.

\end{definition}

\begin{lemma}
\label{lem:genPush}
Let  $C$ and $M$ be irreducible affine varieties,
and $m$ be a polynomial map, all defined over $\QQ$, 
such that $m(C)=M$.
If there exists a polynomial $\phi$,
defined over $\QQ$, that does not vanish identically
over $M$ but does vanish at $m(\p)$ for some 
$\p \in C$, then 
there is a polynomial $\psi$,
defined over $\QQ$, that does not vanish identically
over $C$ but does vanish at $\p$.

Thus,
if $\p \in C$ is generic in $C$, 
then $m(\p)$ is generic in $M$.
\end{lemma}
\begin{proof}
Simply pull back $\phi$ through $m$.
\end{proof}

\begin{lemma}
\label{lem:genPull}
Let  $L$ and $M$ be irreducible affine varieties of the same dimension,
and $s$ be a polynomial map, all defined over $\QQ$,
 such that $s(L)=M$.
If there exists a polynomial $\phi$,
defined over $\QQ$, that does not vanish identically
over $L$ but does vanish at some $\bl \in L$, then 
there is a polynomial $\psi$,
defined over $\QQ$, that does not vanish identically
over $M$ but does vanish at $s(\bl)$.

Thus,
if $\bl \in L$ is not generic in $L$, then $s(\bl)$ is not generic in $M$.
\end{lemma}
\begin{proof}
Since $L$ is irreducible, the vanishing locus of $\phi$ must be 
of lower dimension. This subvariety
must map under $s$ into to a lower-dimensional subvariety of $M$ (defined
over $\QQ$). This guarantees the existence of an appropriate $\psi$.
\end{proof}

Ultimately, we will be most interested in properties that hold 
not merely at all generic configurations, but over an open and dense
subset of the configuration space.
 Such a property will
be what we ``generally'' observe when looking at configurations, and
will be stable under perturbations. There can be exceptional configurations
but they are very confined and isolated.

When a property holds at all generic points of an irreducible 
variety, and the exceptions are due to only a \emph{finite} number of
algebraic conditions, then we will be able to conclude that the 
property actually holds over a Zariski open subset.
\begin{definition}
A non-empty subset $S$ of a variety $V$ is \defn{Zariski open} if it can be obtained
from $V$ by cutting out a (strict) subvariety.
A non-empty Zariski open subset of an irreducible variety $V$
has full measure in  $V$.

The real locus of a Zariski open subset of $\CC^N$ 
is a 
real Zariski open subset of $\RR^N$ as in Definition~\ref{def:zopenR}.
(This is because there is no non-trivial polynomial that
vanishes over all of $\RR^N$.)
\end{definition}

\begin{remark}
\label{rem:gtoz}
Lemmas~\ref{lem:genPush} and~\ref{lem:genPull} let us follow
generic points through appropriate maps. They also let us
follow "bad" strict subvarieties in the opposite directions 
through these maps. Thus when we are in a setting, such
as the $b$-bounded setting, where we are only concerned
with a finite collection of algebraic conditions going wrong
and spoiling some property, 
we can then upgrade our statements from being about
generic points, to holding over Zariski open subsets.
\end{remark}

With our notion of generic fixed, we can prove the two
principles of Section~\ref{sec:intro}.

\begin{proof}[Proof of Theorem~\ref{thm:prin1}]
Suppose $\E(V)$ does not lie in $W$. Then the preimage
$\E^{-1}(W)$, which is a variety defined over $\QQ$, does not
contain $V$, and the inclusion of $\genericpoint$ in this preimage
would render $\genericpoint$  a non-generic point of $V$.  
\end{proof}

\begin{proof}[Proof of Theorem~\ref{thm:prin2}]
From Theorem~\ref{thm:prin1}, $\A(V) \subset V$.
From Lemma~\ref{lem:bij}, $A(V)$ is an algebraic 
subvariety of $V$ of the same dimension, which 
from the assumed irreducibility must be $V$ itself.
\end{proof}


There are two approaches for defining smooth and singular points.
One comes from our  
algebraic setting, while the other comes from the more general setting
of complex analytic varieties (which we will explicitly 
refer to as ``analytic''). 
It will turn out that (algebraic) smoothness implies analytic smoothness,
and  that analytic smoothness
implies (algebraic) smoothness.

\begin{definition}
The \defn{Zariski tangent space} at a point $\genericpoint$ of a
variety $V$
is the kernel of the Jacobian matrix of a set of
generating polynomials for $I(V)$ evaluated at $\genericpoint$.



A point $\genericpoint$ is called
(algebraically) \defn{smooth} in $V$ if the dimension of the Zariski
tangent space equals the local dimension $\Dim_\genericpoint(V)$. Otherwise $\genericpoint$ is called
(algebraically) \defn{singular} in $V$.
The \defn{locus} of singular points of $V$ is denoted $\sing(V)$.
The singular locus is itself a strict subvariety of $V$. 
Thus when $V$ is irreducible 
and defined over $\QQ$, 
all generic
points are smooth.


\end{definition}

\begin{theorem}
\label{thm:Tmap}
If $\A$ is a bijective linear map on $\CC^N$ that acts as a 
bijection between two irreducible varieties
$V$ and $W$, then it must map singular points to singular points.
\end{theorem}
This is a special case of the more general setting of 
``regular maps'' and 
``isomorphisms of varieties''~\cite[Page 175]{harris}.

\begin{theorem}
\label{thm:conIrr}
If a point $\genericpoint$ is contained in two distinct components of $V$,
then $\genericpoint$ cannot be 
a  smooth point in $V$. 
\end{theorem}
See~\cite[II. 2. Theorem 6]{shaf}.

\begin{definition}
If a point $\genericpoint$ in a variety
$V$ 
has a neighborhood
in $V$ that is a complex submanifold of $\CC^N$ with  some dimension 
$D$,
then we call the point
\defn{analytically smooth of dimension $D$} in $V$, or just
\defn{analytically smooth} in $V$. Otherwise we call the point
\defn{analytically singular} in $V$. 
This definition makes no use of $\Dim_\genericpoint(V)$.

\end{definition}

The following theorem tells us that there is no difference between these to notions of smoothness.
\begin{theorem}
\label{thm:smpt}
An (algebraically) smooth point $\genericpoint$ in a variety $V$ must be an
analytically  smooth point of dimension $\Dim_\genericpoint(V)$ in $V$.

A point $\genericpoint$ 
that is analytically smooth of dimension $D$
in $V$
must be
an (algebraically) smooth point $\genericpoint$ in $V$ with
$\Dim_\genericpoint(V)=D$.
\end{theorem}
For discussions on this theorem see~\cite[Exercise 14.1]{harris},
~\cite[Page 13]{milnor}. See~\cite[Page 14]{sing} for the setting
where one does not assume irreducibility, or even pure dimension.

Note that the second direction does not have a corresponding statement in the setting of real algebraic varieties.

\section{Determinants and flips}
\label{sec:det}
In this section, we will establish 
a technical
lemma about determinants and sign flips.
Recall from Definition~\ref{def:sf}
that a \defn{sign flip matrix} $\S$ is a diagonal matrix 
with $\pm 1$ on the diagonal.

\begin{lemma}\label{lem:nonsing}
Suppose that $\Z = \S\X + \Y$ is an $r\times r$ matrix 
and
$\det(\Z) = 0$ for all choices of sign flips, $\S$. Then 
$\det(\Y) = 0$.
\end{lemma}
\begin{proof}
Multilinearity of the determinant allows us to 
express $\det(\Z)$ 
as $\det(\Z')+\det(\Z'')$,
where $\Z'$ is the matrix $\Z$
with its first row replaced
by the first row of $\S\X$, and
where $\Z''$ is the matrix $\Z$
with its first row replaced
by the first row of $\Y$.
We can likewise expand out each
of $\det(\Z')$ and $\det(\Z'')$ 
by splitting their second rows.
Applying this decomposition recursively  
we ultimately get:
\[
	\det(\S\X + \Y) = \sum_{I\subset[r]} \det(\Z^\S_I)
\]
where $[r] = \{1,2,\dots,r\}$,
and $\Z^\S_I$ is the matrix that has the rows indexed by $I$ from $\S\X$ and the rest 
from $\Y$.

Now  sum the above over the $2^r$ choices of $\S$ and rearrange
\[	
	\sum_{\S} \det(\S\X + \Y) = \sum_{\S}\sum_{I\subset [r]} \det(\Z^\S_I) = 
    \sum_{I\subset [r]} \overbrace{\sum_{\S} \det(\Z^\S_I)}^{\star}
\]
For fixed $I$, each $\det(\Z^\S_I) =  (-1)^{\sigma(\S,I)}\det(\Z^{\I}_I)$,
where $\sigma(\S,I)$ is the number of rows corresponding to $I$ where 
$\S$ has a diagonal entry of $-1$.  Thus, for each $I$, ($\star$) is
\[
	2^{r-|I|}\cdot\left(\sum_{k=0}^{|I|}\binom{|I|}{k}(-1)^k\right)
    \cdot\det(\Z^{\I}_I)
\]
(The power of two factor accounts for all of the
sign choices in $\S$ over the complement of $I$.)
The coefficient of $\det(\Z^{\I}_I)$
equals $2^{r}$ when $I$ is empty. Otherwise 
it is zero since the inner term is simply the 
binomial expansion of
$(1-1)^{|I|}$. Thus,
\[
	\sum_{\S} \det(\S\X + \Y) =  2^r \det(\Y)
\]
Since this sum vanishes by hypothesis, we get $\det(\Y) = 0$.
\end{proof}

\section{Rational Functionals and Relations}\label{sec:rational:funcs}
In this section, we prove some generally useful facts about
rational functionals and relations acting on generic point configurations $\p$.

\begin{theorem}
\label{thm:functional}
Let $\bl$ be a generic point in $\LN$, with  $d \ge 2$,
and let $\alpha$ be rational length functional. 
Suppose $\la \alpha, \bl\ra=0$, then 
$\alpha=0$. Likewise (due to linearity), 
if $\la \alpha,\bl\ra=\la \alpha',\bl\ra$, then $\alpha=\alpha'$.

Similarly, let $\p$ be a generic configuration in $\RR^d$ with $d\geq 2$.
Suppose $\la \alpha, \p \ra=0$, then 
$\alpha=0$. Likewise, if 
$\la \alpha, \p\ra=\la \alpha',\p\ra$, then $\alpha=\alpha'$.
\end{theorem}

Recall from Theorem~\ref{thm:Lvariety} that, assuming $d\geq 2$, 
$\LN$ is irreducible, hence it has generic points.
Additionally when $\p$ is a  generic  configuration, then 
$l(\p)$ is generic in 
$\LN$.

\begin{proof}
The equation $\la \alpha, \bl\ra=0$ describes an algebraic equation over $\LN$ 
with coefficients in $\QQ$ and that vanishes at $\bl$. 

Next, we want to show that, assuming  $\alpha \neq 0$, 
this equation does vanish identically.
To do this, we only need to find
one point (not necessarily generic) in $\LN$ where it does not vanish.
But since $\LN$ is symmetric under sign negations, 
we can always find a point $\bl$  such the
sign of each coefficient $\alpha^{ij}$ agrees with that of  the coordinate $l_{ij}$.

If this equation does not vanish identically over $\LN$, but does at $\bl$,
then by definition $\bl$ cannot be generic.
\end{proof}
\begin{remark}
When $d=1$, there 
can be  generic  configurations $\p$ such that
$\la \alpha, \p\ra=0$ with $\alpha \neq 0$.
For example, suppose $n=3$, and let
$\alpha_{12}=1, \alpha_{23}=1, \alpha_{13}=-1$.
Then, $\la \alpha, \p\ra=0$, whenever we have the order $\p_1 \leq \p_2 \leq \p_3$, 
or the reverse order $\p_1 \geq \p_2 \geq \p_3$, 
and
$\la \alpha, \p\ra\neq 0$ otherwise. 
In this case,  $L_{1,3}$ is reducible,
and we have an equation that vanishes identically on one component of the variety
but not on the others.
\end{remark}

The following is useful to tell when a set of rational functionals
is linearly dependent.

\begin{lemma}
\label{lem:depC}
Let $\p$ be a configuration, $\alpha_i$ a sequence of $k$ 
rational functionals, and $v_i := \la \alpha_i,\p \ra$. Suppose that
the functionals $\alpha_i$ 
are linearly dependent. Then, there
is a linear 
dependence that can be expressed as 
$\sum_i c^i \alpha_i=0$, where the
coefficients $c^i$ are rational, not all vanishing.
Moreover, this gives us the 
relation $\sum_i c^i v_i =0$ with the same coefficients.  If the functionals $\alpha_i$ are
integer and $b$-bounded, then we can find such coefficients $c^i$ that are integers, 
bounded in magnitude by 
$b^{k-1}$.
\end{lemma}
\begin{proof}
Let $k'<k$ be the dimension of the span of the $\alpha_i$.

i) Let us look at the case $k' < \edgecard$.

Pick a subset of the $\alpha_i$ that is minimally linearly dependent
with size 
$k'+1$.
Let us use these as the $k'+1$ rows of a matrix $\M$ with
$\edgecard$ columns.
Each
of its minors of size $k'+1$ must
vanish. 

Pick $k'$ columns that are linearly independent. Append to these, 
one column made up of $k'+1$ variables. The condition that the
determinant of this $(k'+1)\times (k'+1)$ matrix vanishes gives us
a non-trivial linear homogeneous equation in the variables.
In the $b$-bounded setting, the coefficients $c^i$ of this equation 
are bounded in magnitude by $b^{k'}$. 
As every column of $\M$ is in the span of our chosen $k'$ columns,
the entries in each column of $\M$  must satisfy this equation.
Thus we have found a rational relation on the $k'+1$ rows of $\M$, giving us
a rational relation on the $\alpha_i$.

ii) Let us look at the case $k' = \edgecard$.

Pick a subset of the $\alpha_i$ of size $\edgecard$ that is 
linearly independent
Let us use these as the rows of a square non-singular matrix $\M$.
Pick one more functional $\beta$ from the $\alpha_i$. Let us think of 
$\beta$ as a row vector of length $\edgecard$.
Since $\beta$ is in the span of our selected rows, we have
$[\beta \, {\rm adj}(\M)] \M = \beta [\det(\M)]$.
Here ``adj'' denotes the adjugate matrix.
This gives us a rational relation between the rows of 
$\M$ and $\beta$, with the coefficients
in brackets above.
Again in the $b$-bounded setting, the coefficients
are bounded in magnitude by $b^{k'}$.

In both cases i) and ii), the relation on the $\v_i$ follows
immediately.
\end{proof}

\begin{lemma}
\label{lem:Cdep}
Let $\p$ be a generic configuration in two or more dimensions.
Let $\alpha_i$ be a  sequence  of $k$ 
rational functionals.
Let $v_i := \la \alpha_i,\p \ra$.
Suppose there is a sequence 
of $k$
rational coefficients $c^i$, not all zero, such that 
$\sum_i c^i v_i=0$.
Then, there is a linear dependence in the functionals $\alpha_i$.
\end{lemma}
\begin{proof}
\ba 
0&=& \sum_i c^i v_i \\
&=& \sum_i c^i \la \alpha_i, \p \ra \\
&=& \la \sum_i c^i \alpha_i, \p \ra 
\ea
Then, from Theorem~\ref{thm:functional}, $\sum_i c^i \alpha_i$ must
be the zero functional.
\end{proof}


\section{Fano Varieties of $\LL$}
\label{sec:fano}

In two dimensions, we can prove the following.

\begin{theorem}
\label{thm:3to6}
Let $\N$ be any invertible $6\times 6$ matrix
with rational coefficients.
Let  $\w:=(w_1,\dots,w_6)$ have rational rank of
$3$ or greater and
describe a point in $\CC^6$ that
is a non-singular point of $\N(\LL)$. 
Suppose that there is a generic configuration $\p$ 
in $\RR^2$ 
and six non-negative rational functionals $\gamma_i$, 
such that $w_i=\la \gamma_i,\p \ra$.
Then $\w$ has rational rank $6$.
\end{theorem}

To prove this theorem, we will study the linear subsets in 
$L_{2,4}$. 

\begin{definition}
\label{def:fano}
Given an affine algebraic cone $V \subset \CC^{N}$ (an affine variety defined by a 
homogeneous ideal), its \defn{Fano-$k$} variety 
$\Fano_k(V)$
is the subset
of the Grassmanian $\Gr(k+1,N)$ corresponding to $k+1$-dimensional linear
subspaces that are contained in $V$. 
\end{definition}

\begin{theorem}
\label{thm:3flats}
The only $3$-dimensional linear subspaces that are contained in $\LL$ are the $60$ $3$-dimensional linear spaces comprising its singular locus.
Moreover, there are no linear subspaces of dimension
$\ge 4$ contained in $\LL$.
\end{theorem}
\begin{proof}
This proposition is proven by calculating  the Fano-$2$ variety of 
$\LL$ in the Magma CAS~\cite{magma}, and comparing it to the the Fano-$2$ variety of 
the singular locus of $\LL$. 

We use the approach described in~\cite[Page 70]{harris} to compute 
the $\Fano_2(\LL)$ variety. We summarize this approach here.
We shall order the coordinates of $\CC^6$ in the order
$( l_{12}, l_{13}, l_{23}, l_{14}, l_{24}, l_{34})$.

Let us specify a point in $\CC^6$ as 
\ba
\begin{pmatrix}
1&0&0 \\
0&1&0 \\
0&0&1 \\
\lambda_1&\lambda_2&\lambda_3\\ 
\lambda_4&\lambda_5&\lambda_6\\ 
\lambda_7&\lambda_8&\lambda_9 
\end{pmatrix}
\begin{pmatrix}
t_1\\
t_2\\
t_3
\end{pmatrix}
\ea
where the $\lambda_i$ are variables that specify a three-dimensional
linear subspace of $\CC^6$, and the $t_j$ are variables that
specify a point
on that subspace. Note that this can only represent
an affine open subset of the Grassmanian; it cannot represent three-dimensional
linear subspaces that are parallel to the first three coordinate axes.

We can compute the polynomial in $[\lambda_i, t_j]$
vanishing when the associated 
points in $\CC^6$ are also in $\LL$. We can then 
look at all of the coefficients
(polynomials in $\lambda_i$) of the monomials in $t_j$.
These coefficient polynomials vanish identically iff 
the linear subspace specified by the $\lambda_i$ is in $\LL$.
Thus these coefficients generate an affine open subset of
$\Fano_2(\LL)$.

To study the whole Fano variety, we must also 
look at the other affine subsets of the Grassmanian. Due to 
the vertex symmetry of $\LL$,
we only need to consider the additional two matrices:

\ba
\begin{pmatrix}
1&0&0 \\
0&1&0 \\
\lambda_1&\lambda_2&\lambda_3\\ 
0&0&1 \\
\lambda_4&\lambda_5&\lambda_6\\ 
\lambda_7&\lambda_8&\lambda_9 
\end{pmatrix}
\;\;{\rm and}\;\; 
\begin{pmatrix}
1&0&0 \\
\lambda_1&\lambda_2&\lambda_3\\ 
0&1&0 \\
\lambda_4&\lambda_5&\lambda_6\\ 
\lambda_7&\lambda_8&\lambda_9 \\
0&0&1 
\end{pmatrix}
\ea

These three matrices represent the triplet of coordinate axes corresponding to,
respectively, a triangle, a chicken-foot, and a simple open path. Thus, these $3$ open subsets of 
$\Fano_2(\LL)$, together with vertex relabelings, cover the full Fano variety.

We compute these $3$ open subsets of $\Fano_2(\LL)$ in Magma, and verify that, in each of these open subsets,
$\Fano_2(\LL)$ is $0$-dimensional and 
$|\Fano_2(\LL)| =|\Fano_2(\sing(\LL))|$. 
As  
$\Fano_2(\LL) \supset \Fano_2(\sing(\LL))$, 
we can conclude that  
$\Fano_2(\LL) =\Fano_2(\sing(\LL))$ (see supplemental script).

As Fano-$2$ variety is discrete, the higher Fano
varieties of $\LL$ must also be empty.
\end{proof}

\begin{proof}[Proof of Theorem~\ref{thm:3to6}]

Using the $6$ functionals $\gamma_i$ as rows of a matrix $\E$, we obtain a
linear map from $\Ln$ to $\CC^6$. This matrix $\E$ maps
$l(\p)$ to $\w$, which we have assumed to be in 
$\N(L_{2,4})$.

Since $\N$ is non-singular, we can pre-multiply $\w$ 
and $\E$
by 
$\N^{-1}$, and then wlog treat it as $\I$.

i) From Theorem~\ref{thm:prin1} (and using Theorem~\ref{thm:Lvariety}) 
we see that 
$\E(\Ln) \subset L_{2,4}$. 

ii) From Lemma~\ref{lem:depC} and the assumed rational rank of
$\w$,
the rank of $\E$ must be  $ \ge 3$.

iii) Suppose that the rank $r$
of $\E$ was less than $6$.
Then from Theorem~\ref{thm:linImage},
$\E(\Ln)$ would be an $r$-dimensional constructible
subset $S$ within
an $r$-dimensional linear (and irreducible) space.
Its Zariski closure, $\bar{S}$, would then be equal to  this
$r$-dimensional linear space. 
This closure, $\bar{S}$, must be contained
in any variety, such as $\LL$, that contains $S$.

iv) By assumption, $r \ge 3$.
Thus, from Theorem~\ref{thm:3flats},
$\E(\Ln)$ would be part of $\sing(\LL)$. 

v) But this would contradict our assumption that $\w$ is 
non-singular. Thus $\E$ has rank $6$.

vi) So From Lemma~\ref{lem:Cdep} the rational
rank of $\w$ is $6$.

\end{proof}

\begin{remark}
\label{rem:fano3d}
A useful generalization of 
Theorem~\ref{thm:3to6} to three dimensions would
say that a rational rank of $6$ 
for a non-singular $\w$
implies
a rational rank of $10$. This, for example would 
let us avoid an explicit rational rank test during the trilateration growth phase.

We have been unable to fully
compute any of the
Fano varieties of $L_{3,5}$ in any computer algebra system, but partial
results do not look promising.
We have been able to verify that  
$\Fano_6(L_{3,5})$ 
is not empty (see supplemental script).
This together with our (partial) understanding
of $\sing(L_{3,5})$ suggests that 
$L_{3,5}$ indeed
contains $6$-dimensional linear spaces that are not 
contained in its singular locus. This would thus
rule out such a generalization 
of Theorem~\ref{thm:3to6}.
\end{remark}



\newpage
\bibliographystyle{abbrvnat}
\bibliography{loops}

\end{document}